\documentclass[11pt,a4paper]{article}
\usepackage{amsmath,amssymb,latexsym}
\usepackage{amsthm}
\usepackage{bm}  
\usepackage{overpic}
\theoremstyle{plain}
\newtheorem{theorem}{Theorem}
\newtheorem{corollary}{Corollary}
\newtheorem{lemma}{Lemma}
\theoremstyle{definition}
\newtheorem{example}{Example}
\theoremstyle{remark}
\newtheorem{remark}{Remark}

\newcommand{\cp}{\mathop{\rm cap}}
\newcommand{\supp}{\mathop{\rm supp}}

\newcommand{\Res}{\mathop{\rm Res}}

\newcommand{\field}[1]{\mathbb{#1}} 

\newcommand{\R}{\field{R}} 

\newcommand{\N}{\field{N}}
\newcommand{\C}{\field{C}}

\newcommand{\MM}{{\mathcal M}}

\newcommand{\const}{{\rm const}}

\renewcommand{\Re}{\mathop{\rm Re}}
\renewcommand{\Im}{\mathop{\rm Im}}
\newcommand{\dist}{\mathop{\rm dist}}

\def\XXint#1#2#3{{\setbox0=\hbox{$#1{#2#3}{\int}$}
\vcenter{\hbox{$#2#3$}}\kern-.5\wd0}}

\title{{Phase transitions and equilibrium measures in random matrix models}}

 \author{A. Mart\'{\i}nez--Finkelshtein, R. Orive and E. A. Rakhmanov}

\date{\today}

\begin{document}

%
%

\maketitle

\begin{abstract}
The paper is devoted to a study of phase transitions in the Hermitian random matrix models with a polynomial potential.
In an alternative equivalent language, we study families of equilibrium measures on the real line in a polynomial external field.
The total mass of the measure is considered as the main parameter, which may be interpreted also either as temperature or time. Our main tools are differentiation formulas with respect to the parameters of the problem, and a representation of the equilibrium potential in terms of a hyperelliptic integral. Using this combination we introduce and investigate a dynamical system (system of ODE's) describing the evolution of families of equilibrium measures. On this basis we are able to systematically derive a number of new results on phase transitions, such as the local behavior of the system at all kinds of phase transitions, as well as to review a number of known ones.
\end{abstract}

\tableofcontents

\section{Introduction} \label{intro}

This paper is devoted to a study of families of equilibrium measures on the real line in the polynomial external fields. We consider these measures as  functions of a parameter  $t$ representing either the total mass,  time or temperature. These families are regarded as models of many physical processes, which motivates their intense study in the context of the mathematical physics. But equilibrium problems for the logarithmic potential play an important role also in analysis and approximation theory; in particular, they provide a general method used in the theory of orthogonal polynomials. In this sense, the subject is essentially ``bilingual''  and has so many ramifications that we opted for writing an extended introduction, recreating in part a broader context, instead of a dull enumeration of our results.

We start with Statistical Mechanics in Section~\ref{sectionStatMech}, which we consider our main field of interest inside the Mathematical Physics.
In Section~\ref{subsect:equilibrium_analytic} we formally introduce basic facts related to the equilibrium measure  and mention  some applications in analysis and approximation theory. In Section~\ref{subsec:KdV} we analyze briefly some connections between this work and the equilibrium problems arising in integrable systems theory,  returning at the end to the topic of the Coulomb gas and random matrices. In the course of these discussions  along Section~\ref{intro} we  describe, in general terms, the main results of the paper and its structure. 

Several modifications of an essentially identical extremal problem for the logarithmic potential were introduced in the three fields of mathematics mentioned above, independently and almost simultaneously within few years around 1980. 
The fundamental importance of this problem and of its solution (the equilibrium measure) is nowadays a common knowledge. The intrinsic unity of the underlying potential theory is also known to the specialists. However, 
researchers working in approximation theory, solitons and integrable systems, or random matrices,  are often not aware of all details about related developments in the other disciplines\footnote{This statement, obviously, applies also to the authors of the current paper.}. 

 One of our main goals is a formal investigation of the dependence of the key parameters of the equilibrium measure on the real line from the parameter $t$, and in particular, of their singularities as functions of $t$. In thermodynamical terms, these singularities are closely related to phase transitions, and we start the introduction with the discussion of this problem in the context of statistical mechanics and thermodynamics. It does not mean that we actually intend to interpret our results in any nonstandard way. We want, on the contrary, to make a short review of a few existing standard interpretations, which could be helpful during ``translations'' in our bilingual area. 

We have to point out as a conclusion of the remarks above that the problem of the equilibrium measure at large is much wider than the part under consideration in this  paper. For instance, Riemann and Klein used the (logarithmic) electrostatic models in the theory of algebraic functions and conformal mappings. Even earlier, Gauss, who was in part influenced by Euler and Lagrange, introduced his related variational principle. Other variational aspects of the problem are essentially overlapping with extremal problems of the geometric function theory (Teichm\"uller, Shiffer, and others). Many other great mathematicians made important contributions related to the problem. Our paper is devoted to one particular (though important) problem. Establishing the historical development of each result is a task beyond our possibilities, and  we apologize for possible involuntary omissions.

\subsection{Statistical mechanics}\label{sectionStatMech}

Determinantal random point processes (or self-avoiding point fields) are pervasive in statistical mechanics, probability theory, combinatorics, and many other branches of mathematics \cite{MR2932631}, \cite{Borodin:2012fk}, \cite{MR1754518}, \cite{MR2581882}, \cite{Soshnikov00}, \cite{MR2858436}. Among the examples of such processes we can mention 2-D Fermionic systems or Coulomb gases, Plancherel measures on partitions, two-dimensional random growth models, random non-intersecting paths, totally-asymmetric exclusion processes (TASEP), quantum Hall models, and random matrix models, to mention  a few. Many of these models lead to the so-called \emph{orthogonal polynomial ensembles} \cite{MR2203677}, \cite{MR2858436}, an important subclass of determinantal random point processes. The large scale behavior of such ensembles is described in terms of the asymptotics of the underlying family of orthogonal polynomials, and in the last instance, by the related equilibrium measure solving an extremal problem from the potential theory.

One of the earliest and probably best known examples of orthogonal polynomial ensembles is the joint distribution of the eigenvalues of a random $N\times N$ Hermitian matrix drawn from a unitary ensemble (see \cite{MR0143556}). More precisely, we endow the set of $N\times N$ Hermitian matrices
$$
\left\{ M=\left( M_{jk}\right)_{j,k=1}^N:\; M_{kj}=\overline{M_{jk}} \right\}
$$
 with the  joint probability distribution 
$$
d\nu_N(M)=\frac{1}{\widetilde Z_N}\, \exp\left( -  \mathrm{Tr} \, V(M)\right)\, dM, \qquad 
dM=\prod_{j=1}^N dM_{jj}\prod_{j\not=k}^N d\Re M_{jk}d\Im M_{jk},
$$
where $V : \, \R \to \R$ is a given function with enough increase at $\pm \infty$ to guarantee the convergence of
the integral in the definition of the normalizing constant
$$
\widetilde Z_N = \int \exp\left( -  \mathrm{Tr} \, V(M)\right)\, dM.
$$
Then $\nu_N$ induces a (joint) probability distribution $\mu_N$ on the  eigenvalues $\lambda_1<\dots <\lambda_N$ of these matrices, with the density 
\begin{equation}
\label{Eigdensity}
\mu_N' (\bm \lambda)  = \frac{1}{ Z_N}\, \prod_{i<j} (\lambda_i-\lambda_j)^2 \exp\left( -  \sum_{i=1}^N V(\lambda_i)\right) ,
\end{equation}
and with the corresponding \emph{partition function}
$$
Z_N=\int_\R\dots\int_\R \, \prod_{i<j} (\lambda_i-\lambda_j)^2 \exp\left( -  \sum_{i=1}^N V(\lambda_i)\right) \, d\lambda_1\dots d\lambda_N.
$$
This result can be traced back to a classical theorem by  H.~Weyl, see  \cite{MR1488158} and also \cite{Mehta2004}. Notice that $\nu_N$ is invariant under unitary transformations of the Hermitian matrices, while measure $\mu_N$ defines the equilibrium statistical mechanics of a Dyson gas of $N$ particles with positions $\lambda_j$ on $\R$ in the potential $V$, where they interact by the repulsive electrostatic potential of the plane (the logarithmic potential). On the other hand,  
\eqref{Eigdensity} is also the joint probability density function of an $N$-point orthogonal polynomial ensemble with the weight $w(x) =\exp(-V (x))$ (see e.g.\ \cite{MR2000g:47048}, \cite{MR2203677}).
The \emph{free energy} of this matrix model is defined as 
\begin{equation*} 
F_N=-\frac{1}{N^2}\log Z_N.
\end{equation*}
Regardless its interpretation, a particularly important problem is to analyze its asymptotic behavior in the thermodynamic limit, i.e.~as $N\to\infty$.

A rather straightforward fact, the existence of the limit
\begin{equation*} 
F_\infty=\lim_{N\to\infty}F_N
\end{equation*}
(infinite volume free energy) has been established  under very general conditions on $V$, see e.g.~\cite{MR1487983} and Section \ref{subsect:equilibrium_analytic}  below. An important property of the infinite volume free energy $F_\infty$ is its analyticity with respect to the parameters of the problem. The values of the parameters at which the free energy is not analytic are the \emph{critical points}; curves of discontinuity of some derivatives of the free energy connecting the critical points divide the parameter space into different phases of the model. Thus, critical points are points of phase transition \cite{Yeomans}, and they are going to be a center of our case study.

The observation that the distribution \eqref{Eigdensity} can be regarded as the Gibbs ensemble on the Weyl chamber $\{\bm{\lambda}:\, \lambda_1<\dots <\lambda_N\}$ with Hamiltonian
\begin{equation} \label{discreteIntro}	
\sum_{i< j} \log\frac{1}{|\zeta_i-\zeta_j|} + \frac{1}{2} \sum_{j=1}^N V(\lambda_j)
\end{equation}
allows to foretell the fundamental fact  that the value of $F_\infty$ is given by the solution of a minimization problem for the weighted logarithmic energy. The corresponding minimizer is the \emph{equilibrium measure} associated to the  problem\footnote{We have gathered the related definitions and basic facts in Section \ref{subsect:equilibrium_analytic}.}. 
This measure, which is a one-dimensional distribution on $\R$, is also a model for the limit distribution of the eigenvalues $\lambda_1<\dots <\lambda_N$ in \eqref{Eigdensity}. Indeed, the multidimensional probability distribution $\mu_N$  is concentrated  for large $N$ near a single point $\bm \lambda^* =(\lambda_1^*, \dots \lambda_N^*) \in \R^N$, which is actually the minimizer for the corresponding  discrete energy \eqref{discreteIntro} of an $N$-point distribution. In other words, for large $N$ the measure $\mu_N$  is close to  $\delta(\bm \lambda  - \bm \lambda ^*)$. As $N\to\infty$, the discrete equilibrium measure converges to the continuous one, so that the thermodynamic limits are essentially described by the (continuous) equilibrium measures (see Section \ref{sec:equilibriuminanalysis}). 

 When $V$ is real-analytic, the support  of such a measure (or the asymptotic spectrum of the corresponding unitary ensemble) is comprised of a finite number of disjoint intervals (or ``cuts''), see~\cite{MR2000j:31003}, and the number of these intervals is a fundamental parameter. For instance, the free energy $F_N$ has a full asymptotic expansion in powers of $N^{-2}$ (``topological large $N$ expansion'') if and only if this support is a single interval; otherwise, oscillatory terms are present (this was observed in \cite{MR1790279}, and studied systematically in \cite{MR2187941}, \cite{Ercolani:2003fk}, \cite{MR2495713}). 
 
The particularly interesting  phenomena occur precisely in the neighborhood of the  values of the parameters at which the number of the connected components  of the support of the equilibrium measure changes. Any change in the number of cuts is a phase transition in the sense specified above, but there are also phase transitions of other kinds (not related to a change in the number of cuts): see Sections \ref{subsec:classification} and \ref{sec:phasetrans} for details.
 
We specialize our analysis to the polynomial potential. It is convenient to write $V$ in the form 
\begin{equation} \label{limitbehavior1}	
V(x)=V_{n }(x)= 2 n \, \varphi(x),  
\end{equation}
where $\varphi$ is a polynomial of an even degree and positive leading coefficient. This case is of great interest, see e.g.\ \cite{MR2629605}, \cite{MR1986409}, \cite{Bleher99}, \cite{MR1327898}, \cite{MR1744002}, to cite a few references. A common  situation is when $n, N\to \infty$ in such a way that
\begin{equation} \label{limitbehavior2}	
\lim \frac{N}{n}=t >0.
\end{equation}
Although the coefficients of $\varphi$ (``coupling constants'' of the model) play the role of the variables in the problem, the parameter  $t>0$ stands clearly out, as it was already mentioned above. Recall that it can be regarded either as a temperature (from the point of view of statistical mechanics) or time (from the perspective of a dynamical system), and will correspond to the total mass of the  equilibrium measure in the external field $\varphi$ on $\R$.
  
One of the goals of this paper is the description of the evolution of the limiting spectrum of the unitary ensemble for \eqref{limitbehavior1}--\eqref{limitbehavior2}  as time (temperature) $t$ grows from zero to infinity, paying special attention to the mechanisms underlying the increase (``birth of a cut'') and decrease (``fusion of two cuts'' or ``closure of a gap'') in the number of its connected components. Equivalently, we study one-parametric families of equilibrium measures on $\R$ in a polynomial external field with the total mass $t$ of the measure as the parameter. 

This problem has been addressed in several publications before, here we only mention a few of them. For instance, \cite{MR1986409} studied the fusion of two cuts for the quartic potential $\varphi$, proving the phase transition of the third order for the free energy (in other words, that $F_\infty \in C^2\setminus C^3$). This work found continuation in \cite{MR2629605}, \cite{MR2453314}, where in particular the birth of the cut in the same model was observed (but not proved rigorously). It turns out actually that the third order phase transition of the infinite volume free energy is inherent to all possible transitions existing in this model.

The evolution of the support of the equilibrium measure in terms of the parameters of the polynomial potential has been studied also in \cite{MR2350906}, \cite{MR2262808}, \cite{MR2247930},  \cite{McLeod}, \cite{MR3099797}, where a connection with some PDE's and integrable systems have been exploited. In \cite{MR1744002} the analytic dependence of the endpoints of the support from the total mass for real analytic weights has been proved, among some other results (see  Section~\ref{sec:critical} below).

However, despite of this intense activity, the picture is not complete; even the monographic chapter \cite{MR2932634} contains imprecisions. Open questions exist actually for the first non-trivial case of a quartic field, and this paper adds a number of new details to this particular situation; some of them seem to be significant. For instance, we present a simple characterization of the case when the equilibrium measure has one cut \emph{for all} values of the parameter $t>0$, or when a singularity of type III (see the definition in Section \ref{sec:revisited}) can occur. We also study in more detail the system of differential equations governing the dynamics of the endpoints of the support of the equilibrium measure. In particular, we analyze the behavior of the infinite volume free energy and of other magnitudes of the system near the singular points, revealing some interesting universal properties. 

Finally, but not less important, we show that all the known and new results related to the outlined problem can be systematically derived from two basic facts in the potential theory. One of them is a representation for the Cauchy transform of the equilibrium measure. This representation can be obtained as a corollary of the fact that on the real line any equilibrium measure in an analytic field is a critical measure, which allows us to apply a variational technique presented in \cite{GR:87}, and systematically in \cite{MR2770010} and \cite{Rakhmanov2012}.  Another fact is a Buyarov--Rakhmanov  differentiation formula \cite{Buyarov/Rakhmanov:99} for the equilibrium measure with respect to its total mass and some of its immediate consequences, which we complement with a unified treatment of the differentiation formulas with respect to any coupling constant. 
 We provide further  details in Section \ref{sec:critical}.

\subsection{Equilibrium measure in an analytic external field on $\R$} \label{subsect:equilibrium_analytic}

In this section we introduce basic notation and mention a number of fundamental facts on the equilibrium measures on the real line, necessary for the rest of the exposition. For more details the reader can consult the original papers  \cite{Buyarov/Rakhmanov:99}, \cite{Gonchar:84}, \cite{MR2770010}, \cite{Mhaskar/Saff:85}, \cite{MR638916}, \cite{MR675192} and the monograph \cite{Saff:97}, or, from a slightly different perspective, \cite{Lax/Levermore}.

For a finite Borel measure $\sigma$ with compact support $\supp(\sigma)$ on the plane  we can define its
\emph{logarithmic potential}
$$
 V^{\sigma}(z)=  -\int \log |z-x|\,d\sigma (x),
$$
and its \emph{logarithmic energy},
$$
I[\sigma]= -\iint \log |z-x|\,d\sigma (x)d\sigma (z).
$$
Suppose further that a real-valued function $\varphi$, called the \emph{external field}, is defined on $\supp(\sigma)$. Then we introduce the \emph{total} (or ``chemical'') \emph{potential},
\begin{equation} \label{totalPotential}
W_\varphi^{\sigma}(z) =  V^{\sigma}(z)+\varphi(z),
\end{equation}	
(defined at least where $\varphi$ is) and the \emph{total energy},
\begin{equation} \label{totalcontEnergy}
I_\varphi [\sigma]= I [\sigma] + 2\int \varphi(z) \, dz,
\end{equation}	
respectively. 
This definition makes sense for a very wide class of functions $\varphi$, although for the purpose of this paper it is sufficient to consider basically real-analytic, actually polynomial, external fields on the real axis. As usual, we assume also that 
\begin{equation} \label{cond2}
\varliminf\limits_{|x|\to
+\infty}\varphi(x)/\log |x|=+\infty,
\end{equation}	
condition that is automatically satisfied for any real non-constant polynomial of even degree and positive leading coefficient.  
Then for each $t>0$ there exists a unique measure $\lambda_t=\lambda_{t}(\varphi ,\R)$ with compact support  $S_t =\supp(\lambda_t)$,  minimizing the total energy
\begin{equation}\label{weightedenergy}
I_{\varphi  }[\lambda_t]= \min_{\sigma \in \MM_t } I_{\varphi  }[\sigma]
\end{equation}
in the class $ \MM_t$ of positive Borel measures $\sigma$ compactly supported on $\R$ and with total mass $t$ (that is, $\sigma (\R)=t$).  
Moreover, $\lambda_t\in \MM_t$ is completely determined by the equilibrium condition satisfied by the total potential
\begin{equation}\label{equilibrium1}
W_\varphi^{\lambda_t}(z) \, \begin{cases} =\,c_{t},
&    z\in \supp (\lambda_t)\,, \\
\geq \,c_{t}\,,\, &   z\in \R \,,
\end{cases}
\end{equation}
where the \emph{equilibrium} or \emph{extremal constant} $c_t$ can be written as 
\begin{equation} \label{extremalConst}
c_{ t}=c_t(\varphi)
:= \frac{1}{t}\left( I_{\varphi }[\lambda_t]-\int \varphi \, d\lambda_t \right).
\end{equation}	

When  $K$ is a compact  set on $\R$ and  the external field $\varphi$ is
$$
\varphi(x) =\begin{cases} 0, & \text{if } x\in K, \\
+\infty , & \text{otherwise},
\end{cases}
$$
the corresponding energy minimizer $\lambda_{t}(\varphi ,\R)$  in the class of all probability measures $\mathcal M_1$, denoted by $\omega_K$,  is known as the \emph{Robin measure} of $K$, its energy   is the \emph{Robin constant}   of $K$,
\begin{equation} \label{robin}	
\rho(K) = I[\omega_K],
\end{equation}
 and the logarithmic capacity of $K$ is given by $ \cp (K)=\exp(-\rho(K))$. 
Measure $\omega_K$  is characterized by the  fact that  $\supp(\omega_K)=K$ together with the equilibrium condition
\begin{equation} \label{equilibriumRobin}	
V^{\omega_K}(z) \equiv \rho (K), \quad z \in K.
\end{equation}

For a general external field the main technical problem when finding the equilibrium measure is that its support is not known a priori and has to be established from the equilibrium conditions. Once the support is determined, the measure itself can be recovered from the equations presented by the equality part in \eqref{equilibrium1}, which is an integral equation with the logarithmic kernel. After differentiation it is reduced to a singular integral equation with a Cauchy kernel whose solution (on reasonable sets) has an explicit representation. 

Hence, solving the support problem is the key, and it is essentially more difficult.  

 In the one cut case, its endpoints are determined by relatively simple equations, but in general, straightforward solutions are not available. For a polynomial (actually, real analytic) field the support $S_t$ is a union of a finite number of intervals. Equations may be written for the endpoints of those intervals (and we better know the number of intervals in advance), but these equations are not easy to deal with. For instance, in the case of the polynomial field they may be interpreted as systems of equations on periods of an Abelian differential on a hyperelliptic Riemann surface (see e.g.~Remark \ref{remark:Riemann} in Section~\ref{subsec:variations}), but such systems are usually far from being simple.  In this context, our approach is  based on the dynamics of the family $S_t$. 
 As we have mentioned above, one of our goals is to describe the dynamics of the family of supports  $S_t$  as the mass $t$ (which is also ``time'', or according to \cite{MR1986409}, the ``temperature'') changes from $0$ to $+\infty$. The detailed discussion  starts in Sections \ref{sec:critical} and \ref{sec:polyn} below. 
  
Observe finally that a study of the family $\lambda_t$ with total mass $t$ as a parameter in a fixed external field $\varphi$ may be reduced  by a simple  connecting formula to the family of unit equilibrium measures with respect to the family of fields $\frac1t\varphi$. We have
\begin{equation} \label{homogeneity}
c_t(\varphi)=t \, c_1\left( \frac{1}{t}\, \varphi \right), \quad \lambda_t(\varphi)=t \, \lambda_1\left( \frac{1}{t}\, \varphi \right).
\end{equation}	
In some cases such a reduction is reasonable, but more often it brings more problems than benefits. In particular, we believe that the total mass as a parameter has particular advantages in the problem under consideration.

\subsection{Weighted equilibrium in Analysis} \label{sec:equilibriuminanalysis}
 
The significant progress in the theory of rational approximation of analytic functions (Pad\'e-type approximants) and in the related theory of orthogonal polynomials in the 1980's is due in part to the logarithmic potential method; see the original papers  \cite{MR651757}, \cite{Gonchar:84}, \cite{MR2961473}, \cite{Komlov12}, \cite{MR916090}, \cite{Mhaskar/Saff:85}, \cite{MR769985}, \cite{MR0487173}, \cite{MR638916}, \cite{MR675192}, \cite{MR1322292}, \cite{MR81i:41016}, \cite{Stahl:86} and the monographs \cite{MR1130396}, \cite{Saff:97}, \cite{MR93d:42029}, although the list is not complete. A novel important ingredient of the techniques developed during that period was precisely the introduction of the weighted equilibrium measure (among other types of equilibria). We mention first  two prototypical situations where equilibrium measures are used to study certain forms of asymptotics  of orthogonal polynomials.
 
One scenario where the family of equilibrium measures parametrized by the total mass becomes a natural method of solution is the problem of the rate of convergence in the classical Stieltjes' theorem on Pad\'e approximants to an asymptotic series whose coefficients are moments of a given weight $w$ on the real line. Using standard arguments  this question is reduced to the problem of exterior logarithmic asymptotics of orthonormal polynomials 
$Q_N(x) = Q_N (x;w)= k_N\,  x^N+\dots$, satisfying
\begin{equation}
\label{defOP}
Q_N = k_N x^N+\dots,\quad \quad\int_\R Q_N(x) x^k w(x) dx= 0, \quad \quad k=0, 1, \dots, N-1,
\end{equation}
and where the leading coefficient $ k_N>0$ is defined by the normalization condition
\begin{equation*}
\|Q_N\|^2 =\int_\R Q_N^2(x) w(x) dx= 1.
\end{equation*}

With this normalization  it was shown  in \cite{MR638916}, \cite{MR675192} (see also \cite{Lubinsky88} and \cite[Ch.~VII]{Saff:97})
that if the weight has the form 
$$w (x)= e^ {-2\varphi (x)(1+\varepsilon (x))}\quad \text{with} \quad \varphi (x) = |x|^\rho, \quad \rho>1, $$   
and where $\varepsilon(x)\to 0$ as $|x|\to\infty$, then as $N\to \infty$ and for $z\in \C \setminus \R $ we have
\begin{equation*}
\log|Q_N(z)| \sim c_n - V^{\lambda_N}(z) \sim D(\rho) N^{1- \frac1\rho} |\Im z|,  
\end{equation*}
where $\lambda_N = \lambda_{N}(\varphi)$ and $D$ is a certain explicit function. The notation $F_N(z)\sim G_N(z)$, $z\in \Omega$, means that  $F_N/G_N \to 1$ as $N\to\infty$ uniformly in $z\in\Omega$. The last expression above characterizes the desired rate of convergence of the Pad\'e approximants to the Cauchy transform of $w$, while the middle expression is due to the fact that the final result was obtained by reduction of the original problem to a study of a family of equilibrium measures. In particular, the normalized equilibrium potential turned out to be a good approximation to  $\log|Q_N|$. The parameter $N = \deg Q_N$ is discrete, in accordance to the nature of the problem. The result was immediately generalized to a wider class of fields $\varphi$; we will go into some further details related to the case in Section~\ref{sec:critical} below when we study family $\lambda_{t}(\varphi)$ more systematically.

Another class of problems  leading to a family of unit equilibrium measures with external fields depending on a parameter is the one of the zero distribution for orthogonal polynomials with variable weights.
We mention (in a simplified form) one basic result from  \cite{Gonchar:84} (see also \cite{MR916090}, \cite{Mhaskar/Saff:85}, \cite{Saff:97}): let  $\varphi$ be a continuous external field,  $n, N\in \N$, and  let polynomial $Q_N(x)=Q_{N,n}(x)=x^N+\dots$ be defined by \eqref{defOP}  with $w(x)= w_n(x) = e^{-2n \varphi(x)}$ and normalized this time by $k_N = 1$.
This polynomial is also characterized by the extremal property
\begin{equation}
\label{extremalProp}
m_N:=\int_\R Q_N^2(x) e^{-2n \varphi(x) } dx=\min_{P(x)=x^N+\dots}\int_\R P^2(x) e^{-2n \varphi(x) } dx.
\end{equation}
Now, if $n, N\to \infty$ in such a way that $N/n\to t>0$,  then 
\begin{equation} \label{weakConvergence}	
\lim  \frac{1}{2N}    \log (m_N ) =  c_t, \qquad 
\frac{1}{N}\, \sum_{Q_N(\zeta)=0}\delta_\zeta \stackrel{*}{\longrightarrow} \lambda_t(\varphi),
\end{equation}
where $\stackrel{*}{\longrightarrow}$ is the weak-$*$ convergence. 

The two problems mentioned above  are closely related; to a certain extend, they are actually equivalent. The second scenario may be interpreted as a ``compactification'' of the first one. Such a compactification is in general achieved by a ``scaling'' or ``contraction'', that is a substitution $x' = r x$ with a proper scaling factor $r = r_{n}$. In general, it is a local procedure ($r_n=r_n(x)$). In the case of the fields $\varphi (x) = |x|^\rho$ the substitution is $r = C n^{1/\rho}$. Yet, there is a difference, which is relevant  for our purposes. Depending on the situation one of the forms may be more technically convenient. In this paper we use specifically the first (non compact) form (the different normalization of the polynomials is not  significant).

Notice also that from the point of view of the spectral theory, zeros of $Q_N$ in the second problem are eigenvalues of the truncated $N\times N$ Jacobi matrices associated with the weight  $e^{-2n \varphi }$, and \eqref{weakConvergence} shows that $\lambda_t(\varphi)$ can be naturally interpreted as the limit spectrum of the infinite Jacobi matrix, associated with the scaling $N/n\to  t$. For details related to applications of equilibrium measures in spectral theory of discrete Sturm-Liouville operators see \cite{MR2350223}.

The examples above represent a general fact  that the zero distribution of extremal polynomials (like in \eqref{extremalProp}) is determined by a related equilibrium measure. Another large circle of applications of (continuous) equilibrium measures $\lambda_t(\varphi)$ is related to the discrete analogue of this notion. In particular, in many important cases polynomials whose zeros minimize the discrete energy satisfy linear differential equations.
More specifically,  let
$$
 M_n:=\left\{ \mu=  \sum_{k=1}^n \delta_{\zeta_k} , \; \zeta_k\in \R, \; k=1, \dots, n \right\} \subset \mathcal M_n
$$
denote the set of all point mass measures on $\R$ of total mass $n$. 
The corresponding discrete energies are defined by
$$
E[\mu]:=   \sum_{i\neq j} \log\frac{1}{|\zeta_i-\zeta_j|},
\quad E_\varphi[\mu]:= E[\mu]+2\int \varphi\, d\mu
$$
(cf.\ \eqref{totalcontEnergy}).

Let $\mu_n=\mu_n(\varphi )\in  M_n$ be a minimizer  (not necessarily unique) for $E_\varphi[\mu]$ in the class $M_n$; points $\{\zeta_{1,n}, \dots, \zeta_{n,n}\}=\supp (\mu_n)$ are also known as \emph{weighted Fekete points}. We have
$$
\frac{1}{n}\, \mu_n(\varphi  ) \stackrel{*}{\longrightarrow} \lambda_1(\varphi), \quad E_\varphi\left[\frac{1}{n}\, \mu_n\right] \longrightarrow I_\varphi[ \lambda_1] \quad \text{as } n\to \infty;
$$
for the proof, see e.g.~\cite{Saff:97}, where some applications of Fekete points in approximation theory are also discussed. Furthermore, for a class of external fields $\varphi$, the polynomials 
$$
H_n(x)=\prod_{k=1}^n (x-\zeta_{k,n})
$$
satisfy a second order differential equation with polynomial coefficients and constitute an example of the so-called Heine-Stieltjes polynomials; see   \cite{MR2149266}, \cite{MR2647571}, \cite{MR2770010}, \cite{MR1928260}, \cite{MR2727638}, \cite{Shapiro2008a}, \cite{MR2902190} for further details on the algebraic and analytic properties of these polynomials. As it was shown in \cite{MR1928260}, the limit zero distribution of such polynomials in the classical case studied by Stieltjes \cite{Stieltjes1885} is governed again by the (continuous) one-parametric family of equilibrium measures. Also a related work worth mentioning is \cite{Bertola:2006fk}, where deformation with respect to the parameters of the external field were considered. 

We do not discuss more general equilibrium problems in the context of approximation theory. In particular, we do not mention ``approximational'' applications of constrained and Green's equilibrium problems; we refer the interested reader to the original papers \cite{MR0496592}, \cite{GR:87}, \cite{Stahl:86}, and the reviews \cite{MR2963451} and \cite{Rakhmanov2012}. Instead, in the next section  we  mention briefly one important class of such equilibrium problems, which plays a fundamental role in the soliton theory.

\subsection{Constrained Green's equilibrium problems in soliton theory} \label{subsec:KdV}

In the classical papers \cite{Lax/Levermore}, Lax and Levermore  investigated the problem  of  the small dispersion asymptotics $\varepsilon\to 0+$ for a solution $u(x,t,\varepsilon)$ of the Korteweg--de Vries (KdV) equation
\begin{equation*}
u_t -6u u_x +\varepsilon^2 u_{xxx}=0, \quad \varepsilon>0,\quad u,x\in\R,\quad t\in \R_+ ,
\end{equation*}
with the initial data $u(x) = u(x,0,\varepsilon)$ satisfying conditions
\begin{equation*}
 u(x) \in [-1,0) ,\, x\in \R ,\quad \lim_{|x|\to\infty}u(x) = 0 ,\quad \int_\R | u(x)|(1+x^2) <\infty.
\end{equation*}

As $\varepsilon\to 0$, the solution becomes oscillating for  $t\geq t_0$, where $t_0$ does not depend on $\varepsilon$, with the wave length $\mathcal O(\varepsilon)$ and the amplitude  not depending asymptotically on $\varepsilon$. In other words, there exists a weak limit  as $\varepsilon\to 0$. 
It was proved in \cite{Lax/Levermore} that this limit   
can be written in terms of the function $Q^*_{xx}(x,t)$, where 
\begin{equation*}
Q^*(x,t) = \min  \left \{Q(\psi, x, t): \, 0\leq \psi\leq \phi/\pi \right\},
\end{equation*}
\begin{equation}
\label{GreenEn}
Q(\psi, x,t) = \iint_{[0,1]^2} \log \left| \frac {\eta + \tau}{\eta - \tau} \right|\psi (\eta) \psi(\tau)\,d\eta d\tau
 + 2 \int_{[0,1]} a(\eta, x,t))\psi (\eta)\, d\eta\, ,
\end{equation}
and $a(\eta, x,t) =  x\eta  - 4 t\eta^3  - \theta(\eta)$. Here $\phi (\eta)$ and $\theta (\eta)$ are two analytic functions of $\eta\in [0,1]$, not depending on $x$, $t$, which are determined explicitly in terms of the initial data $u $ 
(we basically keep the notations and the form of the original paper).

The minimizer $d\mu^*(x,t) = \psi^*(x,t) d\eta$ in \eqref{GreenEn} is determined under the ``finite-gap ansatz'': the assumption that its support  is a union of a finite number of intervals. The endpoints $u_k(x,t)$ of these intervals  are the main parameters in the explicit formulas. The compatibility condition $\partial_t\,\psi_x^* = \partial_x\,\psi_t^*$ yields a system of PDE for $u_k(x,t)$, valid in any $(x,t)$-subdomain where the number of connected components of the support is preserved.  These equations are related to the so-called \emph{Whitham equations} or equivalently, \emph{modulation equations} \cite{Flaschka80} (see also \cite{MR2061477}). 

The inverse scattering transform (IST) method created in \cite{Lax/Levermore} was further developed in \cite{Venakides85} (where the original assumption that  $u(x)$ has a single critical point was removed), and in subsequent papers \cite{Deift93}, \cite{Deift97a}, \cite{Deift97b}, \cite{Deift98}, \cite{MR2061477}, \cite{Kamvissis}, \cite{MR2044068}, \cite{MR2376209}, where the method was extended to wider classes of equations  and to more general settings. 
In particular, it was used in the analysis  of the Toda lattices, that is,   a system of difference-differential equations of the form
\begin{equation*}
\dot{a_k} = 2(b_k^2-b_{k-1}^2),\; k=1,\dots,n; \qquad \dot{b_k} = b_k(a_{k+1}-a_k), \; k=1,\dots,n-1\,,
\end{equation*}
with $b_0 = b_n = 0$. The study of its continuum limit, carried out in \cite{Deift98}, presents many similarities with the zero dispersion limit of the KdV equation; in particular, it leads to a similar extremal problem.

For the purposes of our paper it is useful to compare the equilibrium problem \eqref{GreenEn} with the problem \eqref{totalcontEnergy}--\eqref{weightedenergy} above. Observe that  for $\eta,\tau > 0$ the function  $\log (|\eta + \tau|/|\eta - \tau|)$ is the Green function $g_\Omega(\eta,\tau)$  (in the variable $\eta$) of the right half plane $\Omega$ with pole at $\tau$. Therefore, the first term in the right hand side of \eqref{GreenEn} is the (doubled)  Green energy of the distribution $d\mu =  \psi(\eta) d\eta$. The second term is the (again, doubled) energy of this distribution in the external field
\begin{equation}
\label{GreenField}
\varphi(\eta) = a(\eta, x,t) =  x\eta  - 4 t\eta^3  - \theta(\eta). 
\end{equation}

Thus, $Q(\psi, x,t)$ is the total weighted Green's energy of the measure $\mu$, that is
\begin{equation}
\label{GreenEn2}
Q(\psi, x,t) =  I^{\Omega}_\varphi (\mu) =
\iint_{[0,1]^2} g_\Omega(\eta,\tau) \, d\mu (\eta) d\mu(\tau) +
 2 \int_{[0,1]} \varphi(\eta) d\mu( \eta),
\end{equation}
and $Q^*$ is the corresponding minimal (or equilibrium) energy. 

This equilibrium problem is in many ways similar to the problem for the logarithmic potential in Section~\ref{subsect:equilibrium_analytic} above, even though many important details are different. 

First, the Green potential for the right half plane is characterized by the presence of a mirror image of the measure with respect to the origin, contributing to the energy. This changes the form of the related explicit formulas, but not their nature. 

More important is the character of the extremal problem. The upper bound for  the density, $   \psi\leq \phi/\pi$,  in the class of measures $d\mu = \psi(\eta)d\eta$  reflects the fact that the  continuous equilibrium problem has been obtained as a limit of a certain family of discrete equilibria ($\phi$ is the scaled density of the original discrete set in $[0,1]$). This makes the problem essentially more difficult since the support of the equilibrium measure now split into two different parts -- the ``free'' one (aka ``bands'', see \cite[Section 2.1]{Baik}), and  the part where the constraint is in effect (``saturated regions''). Such type of equilibrium problems was later rediscovered in approximation theory in the study of the zero distribution of discrete orthogonal polynomials \cite{Baik}, \cite{Dragnev/Saff}, \cite{MR1662699}, \cite{MR1687531}, \cite{MR1418343}.

Finally, the minimization of the total Green energy is carried out in the class of measures $d\mu = \psi(\eta)d\eta$ with $ 0\leq \psi\leq \phi/\pi$  on $[0,1]$. Observe that   $m = \int d\mu$ is not fixed, which means that the total mass of the measure is not a parameter (or variable) of the problem. Instead, it is determined together with the extremal measure as a function of the parameters $x$ and $t$ ($m = m(x,t)$) by the normalization condition that the equilibrium constant $c_t = 0$.

The extremal problem \eqref{totalcontEnergy}--\eqref{weightedenergy}, under consideration in this paper, is certainly related to the Lax-Levermore minimizer, at least in spirit.  The  family of external fields considered in \cite{Lax/Levermore} contains two parameters, and differrentiation with respect to these parameters is a part of the procedure of solving a PDE. The analogies between both problems can be carried out further, but this is not straightforward. Actually,   the two situations are different in so many details that it is rather difficult to establish clear connections between particular results or formulas. In any case, references to the work of Lax and Levermore will appear along the text.

\subsection{Random matrices and Dyson gas revisited} \label{sec:revisited}
 
Let us return to the random models mentioned in the beginning of this Introduction.  An important fact, already mentioned there, is that the eigenvalue distribution of the unitary ensemble of random matrices \eqref{Eigdensity} may be equivalently written in terms of the discrete energy. Indeed, since we deal only with symmetric functions of a vector $\bm \zeta=(\zeta_1, \dots, \zeta_n)\in \R^n$, we may identify $\bm \zeta$ with the non-ordered collection $\{\zeta_1,  \dots, \zeta_n\}$ of real points. This collection uniquely determines  measure $\eta \in M_n$ with $ \supp(\eta)=\{\zeta_1,  \dots, \zeta_n\}$ by
$$
\eta = \sum_{j=1}^n \delta_{\zeta_j},
$$
so that we can identify also $\bm \zeta\in \R^n \leftrightarrow \eta\in M_n$ and write
$E_\varphi(\bm \zeta)$ instead of $ E_\varphi(\eta)$. 

Now, we define
\begin{equation} \label{Gibbs}	
d \mu_{n,N}(\bm \zeta)=\frac{1}{Z_{n,N}}\, e^{-n^2 E_\psi(\bm \zeta)}, \quad  Z_{n,N}=\int_{\R^n }  e^{-n^2 E_\psi(\bm \zeta)}\,  d \bm \zeta,
\end{equation}
where $\psi =\frac{N}{2n}\varphi$.
Distribution  $ \mu_{n,N}(\bm \zeta)$ coincides with $\mu_N (\bm \lambda)$ in \eqref{Eigdensity} if we identify $\bm \lambda =\bm \zeta$ and use also connection \eqref{limitbehavior1}. As observed before, \eqref{Gibbs} is a canonical Gibbs distribution for a Dyson gas. 

In this connection, the classical Heine's formula
 $$
 Q_N(x)=\int_{\R^N} \prod_{k=1}^N (x-\zeta_k)\,d\, \mu_{n,N}(\bm \zeta), 
 $$
represents the orthogonal polynomial $Q_N$ with respect to the weight $e^{-N \varphi}= e^{-2 n \psi}$ as the average of polynomials with   random real zeros distributed according to $
 \mu_{n,N}$.

In the last decade the distribution $\mu_{n,N}$ and its characteristics, such as the correlation function, has been intensively investigated in relation to the eigenvalue statistics of the unitary ensembles of Hermitian matrices. 

One of the relevant features of $\mu_{n,N}$ in the asymptotic (thermodynamic) regime is that it exhibits non-trivial phase transitions or critical behavior, with universal properties that attract increasing attention. The existence of such phase transitions and their character has important physical interpretations, see e.g.~\cite{Andric:2008fk}, \cite{MR1920071}. 

The equilibrium measure $\lambda_t$ (and consequently, the parameters upon which it depends) is called \emph{regular} (see e.g.~\cite{MR2001g:42050} or \cite{MR1744002}) if the following conditions hold: (a)  $\lambda_t$ is absolutely continuous on $\R$ with respect to the Lebesgue measure and  $\lambda_t'$ does not vanish in the interior of $S_t$; (b) $\lambda_t'$ has a square root vanishing at the endpoints of the connected components of the support $S_t$; and (c) the inequality in \eqref{equilibrium1} is strict in $\R\setminus S_t$. Otherwise, the equilibrium measure is singular. Following the terminology of  \cite{MR2001g:42050}, \cite{MR1744002}, an equality in \eqref{equilibrium1} at a point in $\R\setminus S_t$ is a \emph{singularity of type I} (this is also known as ``birth of the cut'', see below); 
zeros of $\lambda_t'$ that lie in the interior of $S_t$ are called \emph{singular points  of type II} (corresponding to a closure of a gap), while the endpoints of the connected components of $S_t$ where $\lambda_t'$ has a higher order vanishing are called  \emph{singular points  of type III}.

The infinite volume free energy and the correlation kernels are analytic functions of the parameters of the problem as long as the underlying equilibrium measure is regular (see \cite{MR1744002}). Hence, phase transitions (in the physical sense of break of analyticity of the free energy) are intimately connected to the singularities of the equilibrium measure; this is the central idea of this work. Our goal is to discuss the mechanisms of these phase transitions and to analyze the behavior of the free energy near its critical values.

The asymptotic behavior of a random matrix ensemble near singular points of the equilibrium measure has been studied intensively (see \cite{MR2519661},  \cite{Bleher/Its03}, \cite{Bleher/Kuijlaars3}, \cite{MR2434886}, \cite{MR2427458}, \cite{MR2260263}, \cite{MR2244325}, \cite{MR2429248}, \cite{MR1744002}, \cite{MR2436561}, to mention a few). A remarkable feature of this behavior is revealed in the so-called double scaling limit, where the parameters near the critical value are coupled with the $N$ (the size of the  matrix). In this case, with an appropriate scaling of the variables, the asymptotic behavior of the correlation kernel near the critical point depends only on the type of the singularity rather than on the potential itself, which is one of the manifestations of the universality in the random matrix models.

It is known that the asymptotic behavior of the correlation kernel in a double scaling limit at the singular points is connected to the hierarchies of the Painlev\'e equations. In the general polynomial case, the local behavior at the birth of the cut was described in \cite{MR2244325} (in terms of orthogonal polynomials with respect to the Freud weight $e^{-x^{2\nu}}$ on $\R$), and proved almost simultaneously in \cite{MR2519661}, \cite{MR2427458}, \cite{MR2436561} using the Riemann-Hilbert techniques. For the type II singularity \cite{Bleher99}, \cite{MR2434886},  the kernel can be written in terms of the Hastings--McLeod solution of the Painlev\'e II equation, while for the type III singularity, the kernel is related to the Painlev\'e I transcendent \cite{Claeys:2007fk}, \cite{MR2260263}. The double-scaling asymptotic analysis of this sort is beyond the scope and techniques of this paper.

\section{Families of the equilibrium and critical measures on the real line} \label{sec:critical}

In the study of measures minimizing the total energy functional \eqref{totalcontEnergy} it is often convenient to consider simultaneously solutions of certain more general but closely related types of equilibrium problems and the associated  measures, which we call \emph{critical measures}. We discuss here two kinds of critical measures: the local minima and the saddle points of the energy, used to derive an equation for the Cauchy transform of equilibrium measure (see Theorem \ref{lem:MFRa09a} below). In this connection we make also a few general  remarks on the equations that can be useful for the constructive determination of the equilibrium measures.

\subsection{Integral equations for the density}

We fix a smooth  external field $\varphi \in C^{1+\varepsilon}(\R)$ and $t>0$. Recall that the pair consisting of a positive measure $\lambda = \lambda_t \in \mathcal M_t (\R)$ (the equilibrium measure with total mass $t$) and and an (equilibrium) constant $c=c_t\in \R$ is uniquely defined by the (equilibrium) conditions  \eqref{equilibrium1}; measure $\lambda_t $ is also the unique  global minimizer of the total energy in the class $\mathcal M_t (\R)$. Under the assumptions above, it is absolutely continuous with respect to the Lebesgue measure, so that $d\lambda (y) = \lambda'(y) dy$ on $\R$.

Consequently, the problem of the constructive determination of $\lambda_t$ boils down to finding its support $S_t$, a problem that has been under investigation for a few decades. The main conclusion is that this problem has a simple constructive solution if $S_t$ is a single interval (see e.g.~\cite{MR973424} or \cite[Ch.~IV]{Saff:97}); when the support is not connected, a satisfactory constructive solution is not available. 

Let us provide some details in this respect. Assume that the support is a union of $p\geq 1$ disjoint intervals,
\begin{equation} \label{assupport}
S_t=\bigcup_{k=1}^p \Delta_k, \quad \Delta_k= [a_{2k-1},a_{2k}], \quad a_1<a_2<\dots<a_{2p}.
\end{equation}
Differentiating the equality part of the equilibrium conditions \eqref{equilibrium1} one obtains the singular integral equation 
\begin{equation}\label{SinEqu} 
\int_{S_t} \frac {\lambda'(y)}{x-y}\, dy = \varphi'(x) ,\quad x \in S_t.
\end{equation}
This equation on the unknown function $\lambda' \in L^1(S_t)$, given a fixed $S_t$ in \eqref{assupport} and a fixed $\varphi'\in C^{\varepsilon} (S_t)$, is a well known object in the theory of singular integral equations. A general solution of this equation may be written  in terms of singular integrals (see \cite[Ch.~5, \S 4]{MR587310}, as well as \cite{Gakhov1990}, \cite{MR1215485}). The explicit expression for this solution is not so relevant here, but it renders a system of $p$ equations on the $2p$ endpoints  $a_j$ of the support $S_t$ (see equations \eqref{orthCond} below) that are central to our discussion. 

Condition $\varphi \in C^{1+\varepsilon}(\R)$  implies that $\lambda'_t \in C(S_t)$, and in particular, that the density of the equilibrium measure is bounded on $S_t$. Thus, we are interested in the bounded solutions of \eqref{SinEqu}: 
\begin{theorem} Assume $t>0$ fixed. A bounded solution for \eqref{SinEqu} exists if and only if 
\begin{equation}\label{orthCond}
 \int_{S_t} \frac{y^j\,\varphi'(y)}{\sqrt{|A(y)|}}\,dy = 0  , \quad j=0,\dots ,p-1  , \quad \text{where }  A(z)\,=\,\prod_{k=1}^{2p}\,(z-a_k).
\end{equation}	
Furthermore, in this case the bounded solution  is unique and is given by the formula
\begin{equation} \label{expressionForLambda}	
\lambda'(y) = \frac1\pi \sqrt{|A(y)|}  \int_{S_t} \frac{\varphi'(x)}{(x-y)\sqrt{|A(x)|}}\,dx. 
\end{equation}
\end{theorem}
This result is well known. Equations \eqref{orthCond} can be found, e.g., in \cite[Ch.~6]{MR2000g:47048} and in \cite{Saff:97}. Alternatively, they can be regarded as  a corollary of a representation for a general solution of \eqref{SinEqu}, see \cite{MR587310}, \cite{Gakhov1990}, \cite{MR1215485}. 

A general solution of \eqref{SinEqu} satisfies
\begin{equation}\label{CritCond}
W^{\lambda_t}(x):= V^{\lambda_t} (x) +\varphi (x) = c_t^{(k)}  , \quad  x \in \Delta_k,\quad k=1,2,\dots, p,
\end{equation}
where constants $c_t^{(k)}$ are not necessarily the same on different intervals $\Delta_k$. In order to get \eqref{equilibrium1} we must impose  $p-1$  extra conditions,
\begin{equation}\label{Periods_0}
c_t^{(1)} = c_t^{(2)} =\dots= c_t^{(p)}.
\end{equation}
Since $\lambda_t$ is uniquely determined by $\varphi$ and $ S_t$,  these are equations on $a_k$'s. Actually, equations \eqref{orthCond} and \eqref{Periods_0}, complemented with the normalization condition $\lambda_t(S) = t$, constitute a system of $2p$ equations on $2p$ unknowns, the endpoints $a_k$ of  the intervals in $S_t$. Using the  expression for $\lambda_t$ in \eqref{expressionForLambda} they can be explicitly rewritten in terms of $\varphi$. We omit these calculations here, but see \eqref{periods} below, where these equations will be rewritten in a form more suitable for our needs. 

The question in what sense these equations  determine $S_t=\supp(\lambda_t)$ is not trivial. Actually, we still need to verify the inequalities $\lambda'(y) \geq 0$,  $y\in S_t$, along with  $V^{\lambda} (x) +\varphi (x)  \geq c_t$, $ x \in \R$. If they are satisfied, the problem is solved. Can this solution be regarded as constructive? We can certainly assert this in the one-cut situation. But we are inclined to give a rather negative answer in the multi-cut case. To the best of our knowledge, the general support problem is actually a whole research program. 

It is appropriate to mention here the so-called Lax pair method \cite{McLeod} for finding the eigenvalue densities in the Hermitian matrix models. It is based on string equations, that are  the consistency conditions for the discrete-continuous differential relations satisfied by the associated orthogonal polynomials, see e.g.~\cite{MR3099797}.

Let us point out that once $S_t$ has been established, it determines the partition of the total mass $t$ into its ``portions'' 
$$
|\lambda_t|({\Delta_k}) = \int_{\Delta_k} d\lambda_t, \quad k = 1, \dots, p.
$$
These values are closely related to the so-called 't Hooft parameters in the gauge theories.

\subsection{Critical measures on the real line} 

We can make additional progress under the assumption that the external field is real-analytic  on the real line. More exactly, we assume that there is a domain $\Omega$ containing $\R$ and a function $\Phi\in H(\Omega)$ such that $\varphi(x) = \Phi(x) = \Re \Phi(x) $, $ x\in\R$. 

\begin{theorem} \label{lem:MFRa09a}
With the assumptions and notation introduced above, there exists an analytic function $R=R_t$ in $\Omega$, real-valued on $\R$,  such that
\begin{equation}\label{charactMeasure}
\left(C^{\lambda_t}(z) + \Phi'(z)
\right)^2=R(z), \qquad z\in \Omega \setminus S_t,
\end{equation}
where 
$$
C^{\lambda_t}(z)= \int  \frac{d\lambda_t (y)}{ y-z } 
$$
is the Cauchy transform of $\lambda_t$.
\end{theorem}
Formula \eqref{charactMeasure} can be obtained, for instance,  considering variations of the plane of the form $x \to x^\tau= x+ \tau \Re \frac{e^{i\theta}}{x-z}$, with $\theta\in \R$, $\tau>0$ and $z\in  \C\setminus \R$, which induce  variations on measures, and equating to zero the corresponding first variation of the total energy. This idea can be traced back to the original variational arguments of  Schiffer, see e.g.~\cite{Jenkins}, \cite{MR0065652}, \cite[Section 21.11]{Strebel}.

In this way, the theorem is a direct corollary of a more general theorem on equilibrium measures of $S$-curves from \cite[p.~333]{GR:87} or \cite{MR2770010}, and the fact that the real line is an $S$-curve in any real-analytic field. The notion of the $S$-property turned out to be crucial in the extension of the results on the real line to the equilibrium in a harmonic external field in $\C$, considering the complex plane as a conductor. In the described situation the representation \eqref{charactMeasure} remains valid, as well as many of its corollaries derived in this paper, with some obvious modifications. However, we will not go into details here and refer the reader to  \cite{MR2770010}, \cite{Rakhmanov2012} for a more thorough discussion. 
 
In the important paper \cite{MR2000j:31003}, the connection between the discrete and continuous equilibria (in other words, the weighted Fekete points) was explored. In this way, the authors proved formula \eqref{charactMeasure} and established the analyticity of the density of the equilibrium measure in the real-analytic field.
 
The proof of Theorem \ref{lem:MFRa09a} actually carries over to a more general situation, e.g.~when  $\Phi'$ is meromorphic  (in this case $R$ might have poles at the location of poles of $\Phi'$). 
In a more physical context, \eqref{charactMeasure} was obtained for instance as the saddle-point approximation to the large $N$ limit of the number of planar Feynman diagrams of quantum field theory \cite{MR0471676}. The meromorphic $\Phi'$ has found applications also in the study of the asymptotic distribution of zeros of Heine-Stieltjes polynomials \cite{MR2770010} or in the analysis of the phase transition in matrix model with logarithmic action \cite{MR2222967}.

Let us point out briefly some corollaries of Theorem \ref{lem:MFRa09a}.

Clearly, \eqref{charactMeasure} allows us to recover easily the density of the equilibrium measure $\lambda_t\in \MM_t$ on $\R$:
\begin{equation}\label{density}
 d   \lambda_t(y)=\frac{1}{\pi}\,\sqrt{|R_t (y)|}\,dy,\quad y\in \supp(\lambda_t)\,,
    \end{equation}
and its total potential,
\begin{equation}\label{totalpotential}
    W_\varphi^{\lambda_t}(z)= \Re \,\int^z\,\sqrt{R_t(y)}\,dy\,
    \end{equation}
with an appropriate choice of the lower limit of integration and of the branch of the square root. It follows further (as it was shown in  \cite{MR2000j:31003}) that the support  $S_t$ of $\lambda_t$ consists of a finite number of intervals, whose endpoints are simple real zeros of $R_t$; the unknown zeros of $R_t$ become the main parameters determining $\lambda_t$ and its support.

In this representation we do not have equations \eqref{orthCond} anymore (they are now ``embedded'' in the representation 
\eqref{density}, which automatically defines a bounded density). Equations \eqref{Periods_0} now take the form
\begin{align} 
\label{periods}
\int_{a_{2k}}^{a_{2k+1}}\,\sqrt{R_t(y)}\,dy = & 0\,, \quad  k=1,\dots ,p-1 ,
\end{align}
which again can be found in \cite{MR2000j:31003}. We could continue rewriting further the conditions on the unknown parameters, taking into account e.g.~the explicit form of $\Phi'$ and identity \eqref{charactMeasure}, but in no case the equivalent equations will be elementary, or at least suitable for a truly constructive solution. 

One approach that renders an effective procedure for investigation of those equations is to consider the dependence of the system from its parameters. The main parameter is $t $, and it is the central object in this paper. 
Details are presented in subsections \ref{subsec:dynamical_first} and \ref{subsec:dynamical} below.

\subsection{Families of equilibrium measures parametrized by the total mass} \label{subsec:dynamical_first}

In this section we discuss the differentiation formula $\frac d{dt}\lambda_t$ =  Robin measure of $\supp (\lambda_t) $, proposed by Buyarov and Rakh\-ma\-nov in  \cite{Buyarov/Rakhmanov:99},  and the corresponding integral formula, which recovers the field from the family of supports. We present a rather general form of the  theorem below, having in mind that results of the paper may be extended directly to the case of piece-wise analytic fields,  which may  be discontinuous. 

 For an external field $\varphi $, $t>0$ and a Borel measure $\sigma \in \MM_t$ let us introduce two sets,
\begin{equation*} 
S_t(\sigma):=\supp(\sigma), \qquad S^t(\sigma):=\{x\in \R: \, W_\varphi^{\sigma}(x)=\min_{s\in \R} W_\varphi^{\sigma}(s) \},
\end{equation*}	
see the definition in \eqref{totalPotential}. When $\sigma=\lambda_t$, we omit the reference to the measure, writing $S_t$ instead of $S_t(\lambda_t)$ and $S^t$ instead of $S^t(\lambda_t)$.  We assume that the family $\lambda_t = \lambda_{t}(\varphi)$ is associated with a lower semicontinuous function $\varphi $ satisfying  \eqref{cond2}.  Since $c_t = \min_\R W^{\lambda_t}$,  the equilibrium condition \eqref{equilibrium1} can be expressed in the following equivalent form:
\begin{equation}
\label{Sequil}
S_t  \subseteq S^t ,
\end{equation}
and $\lambda_t$ is the unique measure in $\MM_t$ satisfying \eqref{Sequil}.

The following result is Theorem 2 from \cite{Buyarov/Rakhmanov:99} stated in a somewhat simplified form:
\begin{theorem}
\label{lem:Buyarov/Rakhmanov:99}
Under the assumptions on $\varphi$ formulated above, the following assertions hold:
\begin{enumerate}
\item[(i)] The family of  equilibrium measures $\lambda_t = \lambda_{t}(\varphi)$ is monotonically increasing, as well as the families of supporting sets $S^{ t},\, S_{t}$; furthermore,  $S^{ t}=\bigcap_{\tau>t}S_{\tau}$.

\item[(ii)] There exist an at most countable set $\mathcal N \subset \R_+$ such that $\cp (S^t \setminus S_t) >0$ for $t \in \mathcal N $. 

\item[(iii)] For every $t \in \R_+ \setminus \mathcal N $ we have 
\begin{equation} \label{derivativesMeasure}
\frac{d \lambda_{t}}{d t}  =\omega_{S_t}, \quad \frac{d c_{t}}{d t}  =I[\omega_{S_t}]=\rho(S_{t}), \quad \frac{d I_{\varphi }[\lambda_t]}{d t}= 2 c_t.
\end{equation}	
Here $\omega_K$ and $\rho(K)= - \log \cp(K)$ are correspondingly the Robin measure and the Robin constant of the compact set $K$, and  $c_t$ is the extremal constant defined in \eqref{extremalConst}.

\item[(iv)] The following representation holds for the external field
\begin{equation} \label{phirepre}
\varphi(x)= c + \int_0^{\infty} g_{\Omega_\tau} (x, \infty) \, d\tau, \qquad x \in \bigcup_{t>0} S_{t},
\end{equation}	
where $g_{\Omega_\tau} (x, \infty)$ is the Green function of the domain $\Omega_\tau = \overline{\C}\setminus\! S_{t}$ with pole at infinity, and 
$c = \min_{x\in \R}\varphi(x) = \lim_{t\to 0}c_t$.
\end{enumerate}
\end{theorem}
In the terminology of the random matrix ensembles, formulas \eqref{derivativesMeasure} state that, up to a factor 2, the first variation of the infinite volume free energy is the equilibrium constant, while the second variation is the Robin constant of the support (limit spectrum) $S_t$.

The original theorem  in \cite{Buyarov/Rakhmanov:99} also asserts  the existence of the right and left derivatives of $\lambda_t$, $c_t$, $I_\varphi[\lambda_t]$ at any $t\geq 0$, including $t\in\mathcal N $. For analytic fields the set $\mathcal N $ is empty and \eqref{derivativesMeasure} is true for any $t\geq 0$. Furthermore, the first formula in \eqref{derivativesMeasure} should be understood  in general in a weak sense: for any continuous and compactly supported function $f$ on $\R$,
$$
\frac{d}{dt}\, \int f(x) d\lambda_t(x) = \int f(x) d\omega_{S_t}(x).
$$
For analytic fields the support $S_t$ is a finite union of intervals, and measure $\lambda_t$ has an analytic density inside each interval. In this case,  differentiation is understood in the strong sense. The proof of this part of the theorem is substantially simplified under assumption of the analyticity of the field, and we outline it next:
\begin{proof}
For  real-analytic external fields formulas \eqref{derivativesMeasure} can be derived directly from Theorem \ref{lem:MFRa09a}. Indeed, a variation with respect to $t$ of the equilibrium property \eqref{equilibrium1} and of the normalization condition $\lambda_t(\R)=t$ yield that $d\lambda_t/dt$ is a probability measure with the logarithmic potential constant on $S_t$. Formula \eqref{density} shows that it is absolutely continuous with respect to the Lebesgue measure on $\R$, and that its density is of the form
$$
     \lambda_t'(x) = \frac{1}{2\pi}\,\left| \frac{P(x)}{\sqrt{R_t (x)}}\right|, 
$$
where $P$ is a polynomial. These properties characterize the Robin measure of $S_t$, which establishes the first formula in \eqref{derivativesMeasure}. Moreover, the constant value of $V^{d\lambda_t/dt}$ on $S_t$ must be the Robin constant of $S_t$, so that by 
differentiating \eqref{equilibrium1} we get that $dc_t/dt =\rho(S_t)$. 

Finally, by \eqref{extremalConst},
$$
 I_{\varphi } [\lambda_t] = t c_{ t} + \int \varphi \, d\lambda_t.
  $$
From the first two identities in \eqref{derivativesMeasure}, just established, it follows that
$$
\frac{d I_{\varphi }[\lambda_t]}{d t}
=  c_{ t} + t \frac{d c_t}{d t} + \frac{d }{d t} \int \varphi \, d\lambda_t = c_{ t} + t \rho(S_{t}) +  \int \varphi \, d\omega_{S_t},
  $$
and it remains to observe that
$$
 t \rho(S_{t}) +  \int \varphi \, d\omega_{S_t} = \int V^{\omega_{S_t}} \, d\lambda_t +  \int \varphi \, d\omega_{S_t} = \int \left( V^{\lambda_t}  +    \varphi \right)\, d\omega_{S_t} = c_t.
$$
\end{proof}
 
Formulas \eqref{derivativesMeasure} have been independently proved in less generality in several works, see e.g.~\cite{MR1986409}, \cite{MR1139696}, \cite{MR2244325} (see also \cite[formula (1.8) and Lemma 1]{MR1322292}, where the particular one-cut case was analyzed).

An important corollary of the assertion \emph{(iv)} of the theorem is that under general assumptions on $\varphi$ (for instance, continuity), the external  field  can be recovered (up to an additive constant) from the family of supports $S_t$, using the explicit expression \eqref{phirepre}. 

Alternatively, the whole family $\{ S_t: \, t\in \R_+\}$ is characterized by a single function (``$\nu$-transform''),
\begin{equation} \label{nuTransform}
\nu(x) = \nu(x, \varphi) = \inf \{ t\in \R_+:\, x\in S_t\},
\end{equation}
so that
\begin{equation} \label{nuTransform2}
 S_t =  \{x\in \R:\, \nu(x) \leq t \}.
\end{equation}
The $\nu$-transform is a useful tool, in particular, for characterizing the class $\mathfrak F$  of external fields $\varphi$ for which the one-cut case (or equivalently, the zero-gap ansatz) holds for any $t\geq 0$:
\begin{equation}
\label{frakF}
\mathfrak F = \left\{ \varphi:\, S_t = [a(t), b(t)],\, \forall t\geq 0\right\}.
\end{equation}
For symmetric and absolutely continuous fields $\varphi(x) = \varphi(-x) $, nondecreasing for $x>0$, we have  that 
$\varphi \in \mathfrak F$ if and only if the function
\begin{equation} \label{nuT}
\nu(x) = \frac 1\pi \int_0^x t\varphi'(t) (x^2 - t^2)^{-1/2} dt
\end{equation}	
is nondecreasing in $\R_+$; in this case, function $\nu$ above is actually the $\nu$-transform of $\varphi$, and equation $\nu(x) = t$  determines $b(t) = - a(t)$ (uniquely, if $\nu $ is strictly increasing, see \cite {MR1322292},  formula (1.8) and Lemma 1 therein). The  solutions of $\nu(x) = n$, $n\in \N$, are known in the theory of orthogonal polynomials as Mahskar-Rakhmanv-Saff numbers, see \cite[Chapter IV]{Saff:97}.

 For non-symmetric fields the situation is more complicated. Let us assume again that  $\varphi(x)$ and $ \varphi(-x) $ are absolutely continuous and nondecreasing for $x>0$. If $S_t$ is an interval then its endpoints $a=a(t) \leq b=b(t)$  satisfy the system (see e.g.~\cite[Chapter IV]{Saff:97})
\begin{equation} \label{System}
\frac 1\pi \int_a^b  \frac {\varphi'(x)\, dx }{\sqrt {(x  - a)(b -x)}} = 0, \qquad  
\frac 1{2\pi} \int_a^b  \frac {x\varphi'(x)\, dx }{\sqrt {(x  - a)(b-x)}} = t.
\end{equation}
For a particular $t$ the converse is not true: a solution of the system \eqref{System} may have nothing to do with the support of the equilibrium measure. However, under the following additional assumptions, 
\begin{enumerate}
\item[(a)] \eqref{System} has a solution $ a(t) \leq  b(t)$ for any $t\geq 0$, 
\item[(b)] intervals $\Delta_t = [a(t), b(t)]$ are increasing, and their union covers $\R$,   
\end{enumerate}
it is easy to prove, using assertion \emph{(iv)} of Theorem~\ref{lem:Buyarov/Rakhmanov:99}, that $\varphi \in \mathfrak F$, and that $\Delta_t = S_t$ for all $t>0$. 

The following corollary belongs to the third author (it was reported at a meeting on Approximation Theory in Oberwolfach, in February 1992):
\begin{theorem}
If  both functions $\sqrt x \frac d{dx}\varphi (x) $ and $\sqrt x  \frac d{dx}\varphi(- x) $ are increasing for $x>0$ then conditions (a), (b) above are satisfied and consequently, $\varphi \in \mathfrak F$.
\end{theorem}
The assumptions  of this theorem are satisfied, for instance, for  $\varphi(x) = x^p$, $ x\geq 0$, and $\varphi(x) = |x|^q$, $ x < 0$, when  
if $p, q > 1/2$, and this condition is sharp.  A weaker asymptotic result was obtained earlier in \cite{MR1139696} in a context of investigation of logarithmic asymptotics for orthogonal polynomials with non-symmetric weights. 

In the multi-cut case, when each support is a finite union of intervals, representation \emph{(iv)} still may be effectively used, at least for estimates, even though we do not actually do it in this paper. Anyway, this representation and its corollary that the external field can be effectively recovered from the family of supports are  the main point of Theorem~\ref{lem:Buyarov/Rakhmanov:99}.

\subsection{Variations of the external field} \label{subsec:variationExtField}

Here we bring to consideration  another useful technique related to the variation of the external field. Let $\varphi(x,\tau)$ be a one-parametric family of external fields depending on a parameter $\tau$. Without loss of generality, let $\tau\in (-\varepsilon, \varepsilon)$; for simplicity, we assume here that $\varphi$ is real-analytic in both variables. The following result, extending the previous theorem, holds:
\begin{theorem}
\label{thm:generalparameter} Let $t>0$ be fixed. 
Assume that $\varphi(x,\tau): \R \times (-\varepsilon, \varepsilon)\to \R$ is real-analytic in both variables, and 
\begin{equation} \label{derivativegeneral}
\frac{\partial \varphi(x,\tau)}{\partial \tau} \bigg|_{\tau=0}=\psi(x)
\end{equation}
uniformly in $x$ in a neighborhood of $S_t=\supp(\lambda_t(\varphi(\cdot, 0))) $.
 
Then
\begin{equation*} 
\frac{\partial \lambda_{t}}{\partial \tau} \bigg|_{\tau=0}  =\mu,
\end{equation*}
where the (in general) signed measure $\mu$ is uniquely determined by the following conditions:
\begin{equation} \label{determinationMu}	
\supp(\mu)\subset S_t, \quad \mu(S_t)=0, \quad V^{\mu}+ \psi = \frac{\partial c_{t}}{\partial \tau} \bigg|_{\tau=0} = \const \text{ on } S_t.
\end{equation}
\end{theorem}
\begin{proof}
Let $\mu$ be any limit point (in the weak topology) of the family
\begin{equation} \label{familyNeutralMeasures}	
\frac{\lambda_t(\tau)-\lambda_t(0)}{\tau}
\end{equation}
(which exists due to its weak compactness). The first property of $\mu$ in \eqref{determinationMu} is obvious; the second one is a consequence of the fact that measures in \eqref{familyNeutralMeasures} are neutral (their mass is 0). Finally, with our analyticity assumptions on $\varphi$ it  easily follows from \eqref{charactMeasure} (see also \cite[Lemma 5.3]{MR2770010}) that 
$$
  C^{\lambda_t(\tau)}(x) + \varphi(x,\tau) =0, 
$$
in the interior of the support of $\lambda_t(\tau)$. Consequently, 
$$
C^\mu(x) + \psi'(x) =0, \qquad x\in S_t,
$$
which implies the third equality in \eqref{determinationMu}.

Conditions \eqref{determinationMu}  determine $\mu$ uniquely. Indeed, the logarithmic potential of the difference of two signed measures $\mu$ satisfying \eqref{determinationMu} is constant and equal to $0$ on $S_t$. It remains to use that the logarithmic energy is a positive definite functional on neutral signed measures (see e.g~\cite[Theorem 1.16]{Landkof:72}). 
\end{proof}

\begin{remark}
As we mentioned in Section~\ref{subsec:KdV}, variations with respect to specific parameters of an equilibrium problem can be found already in the seminal work of Lax and Levermore \cite{Lax/Levermore}. However, we have not found in literature the statement of Theorem~\ref{thm:generalparameter} in the form it is presented above.
\end{remark}

\begin{remark}
Theorem \ref{thm:generalparameter} can be extended to more general (e.g., $C^{1+\varepsilon}$) external fields $\varphi$, but we prefer to keep the settings as simple as possible, and still sufficient for our purposes. 
\end{remark}

\begin{remark} As it was mentioned above (see \eqref{homogeneity}), we can reduce our analysis to unit (probability) measures on $\R$, regarding the total mass or temperature $t$ as a particular parameter of the external fields of the form $\varphi/t$. Consequently, it is convenient to clarify the relation between Theorems \ref{lem:Buyarov/Rakhmanov:99}  and \ref{thm:generalparameter}.

According to \eqref{homogeneity},
$$
\lambda_t(\varphi) = t \lambda_1^*, \quad \text{where} \quad \lambda_1^*=\lambda_1(\varphi(\cdot, t )), \quad \varphi(\cdot, t )= \frac{1}{t}\, \varphi.
$$
The derivative of $\varphi(\cdot, t )$ at a given $t>0$ satisfies \eqref{derivativegeneral} with $\psi = -\varphi/t^2$, and from Theorem \ref{thm:generalparameter} it follows that $\partial \lambda_1^* /\partial t$  is  a signed measure $\mu$ characterized by the equilibrium conditions \eqref{determinationMu}. In this particular case $\mu$ can be computed explicitly: observe that
\begin{equation} \label{connectionThm2and3}	
\frac{\partial \lambda_t}{\partial t}= \lambda_1^* + t\,  \frac{\partial \lambda_1^*}{\partial t},
\end{equation}
so that applying Theorem \ref{lem:Buyarov/Rakhmanov:99} we get
$$
\frac{\partial \lambda_1^*}{\partial t}=\mu=\frac{1}{t}\left( \omega_{S_t} - \lambda_1^*\right);
$$
obviously, the right hand side is a neutral measure ($\mu(S_t)=0$) satisfying \eqref{determinationMu}, with the equilibrium constant equal to
$$
\frac{\partial}{\partial t} \left( \frac{c_t}{t}\right) = \frac{\rho(S_t)}{t}- \frac{c_t}{t^2}.
$$

The reciprocal statement is also valid: formula \eqref{connectionThm2and3} shows that the logarithmic potential of $\partial \lambda_t/\partial t$ is constant on $S_t$, fact that combined with the monotonicity of $\lambda_t$ yields the first formula in \eqref{derivativesMeasure}. 
\end{remark}

\section{A polynomial external field} \label{sec:polyn}

After the brief review of the necessary technical background we have carried out in the previous section, we can start our study of equilibrium measures in a polynomial external field. We define 
\begin{equation} \label{phiPolyn}
\varphi(x)=\sum_{j=1}^{2m} t_j \, x^{j}, \quad t_j\in \R, \quad t_{2m}>0,
\end{equation}	
where the coefficients $t_j$ (the \emph{coupling constants}) are the parameters of the problem. 
For convenience, we adopt the normalization
\begin{equation} \label{normalizationLeading}	
t_{2m}=\frac{1}{2m},
\end{equation}
as well as include the total mass or temperature $t_0=t$ as an additional coupling constant. In this way, the equilibrium measure $\lambda=\lambda_t(\varphi, \R)$ is a function of $2m$ real parameters  $t_j$, $j=0, 1, \dots, 2m-1$.

\subsection{$R$-representation and $(A,B)$-representation} \label{subsec:R-representation}

First, we reformulate Theorem \ref{lem:MFRa09a} and its corollaries specifically for the case at hand, that is,  when $\varphi$ is defined by \eqref{phiPolyn}--\eqref{normalizationLeading}. The theorem implies that for the associated equilibrium measure $\lambda_t\in \MM_t$ with total mass $t>0$ and support $S_t\subset \R$, the function 
\begin{equation}\label{charactMeasure1}
R_t(z) = \left(C^{\lambda_t} + \varphi' \right)^2(z)
= z^{4m-2} +\text{lower degree terms}
\end{equation}
is a polynomial in $z$ of degree $4m-2$ with coefficients depending of $t$. Actually, if we denote by $(f)_\oplus$ the  polynomial part of the expansion of $f$ at infinity, then the normalization $\lambda_t(\R)=t$ implies that
\begin{equation} \label{RcoefGeneral} 
 R_t=\left(\varphi'\right)^2+2 \left( \varphi' C^{\lambda_t}\right)_\oplus=\left(\varphi'\right)^2(z)-2t \, z^{2m-2}+\text{lower degree terms}.
\end{equation}	

Let  $a_1<a_2<\dots<a_{2p}$, $1\leq p\leq m$, denote the real zeros of $R_t$ of  odd multiplicities. Observe that $R_t$ cannot have complex zeros of odd multiplicity, and each real zero of odd multiplicity of $R_t$ is an end point of $S_t$, so that expression \eqref{assupport} holds.
We  define monic polynomials $A$ and $B$ by
$$
A(z)= \prod_{k=1}^{2p}\,(z-a_k), \quad \text{and}\quad  B^2(z) = \frac{R_t(z)}{A(z)}  
$$ 
(notice that $ R_t(z)/A(z)$ is a monic polynomial of degree $2(2 m-p-1)$ whose zeros have even multiplicities and, therefore, the  monic polynomial $B$ of degree $2m-p-1 \geq 0$ is correctly defined). By \eqref{periods}, for $p>1$ there should be at least one zero of $B$ in each gap $[a_{2k}, a_{2k+1}]$; this readily yields the bound $p\leq m$. Naturally, both polynomials $A$ and $B$, as well as the endpoints $a_k$ of the support $S_t$, are functions of the parameter $t$, fact that we usually omit from notation for the sake of brevity. 

In what follows we understand by $A^{1/2}$  and $R_t^{1/2}$ the holomorphic branches of these functions in $\C\setminus S_t$ determined by the condition
$A^{1/2}(x)>0$  and $R_t^{1/2}(x)>0$ for $x>a_{2p}$, as well as $f_+$ denotes the boundary value of $f$ on $\R$ from the upper half-plane.

By \eqref{density},
\begin{equation}\label{R-representation} 
d\lambda_t(x) =  
 \frac{1}{\pi i}\,(R_t^{1/2})_+(x)  \,dx,\quad x\in S_t.
\end{equation}

The factorization 
\begin{equation} \label{AB-repr}	
R_t(z)=A(z)B^2(z) = A(z; t)B^2(z; t)
\end{equation}
 plays a fundamental role in the  sequel.  
In particular,  the $R$-representation of the equilibrium measure in \eqref{R-representation} may be equivalently written in terms of the following $(A,B)$-representation:
\begin{equation}\label{(A,B)-representation} 
 d\lambda_t(x) =   \frac{1}{\pi i}\,A^{1/2}_+(x) B(x)\,dx, \quad x\in S_t.
\end{equation}

In a one-cut case ($p=1$) the  equalities \eqref{RcoefGeneral} and \eqref{AB-repr} render a system of $2m$ algebraic equations on the zeros of $A$ and $B$, matching the number of unknowns. Taking into account additionally that for $p=1$ equation \eqref{orthCond} is algebraic in $a_j$'s, this shows that in this (and basically, only in this) case all parameters of the equilibrium measure $\lambda_t$ are algebraic functions of the coefficients $t_k$ of the external field \eqref{phiPolyn}.

We conclude this section with two trivial but useful identities: 
\begin{equation} \label{polynomialpart}
\varphi'=\left(A^{1/2} B\right)_\oplus= \left(R^{1/2}\right)_\oplus , \quad B=\left(\frac{\varphi'}{A^{1/2}}\right)_\oplus, 
\end{equation}	
and
\begin{equation} \label{limitT0}
\lim_{t\to 0+} R_t(x)=\left(\varphi'\right)^2(x)
\end{equation}	
(see \eqref{RcoefGeneral}).

The $(A,B)$-representation introduced above has several advantages: it allows us to recast the discussion on the evolution of the support $S_t$ and of the phase transitions for the free energy in terms of the zeros   of $A$ and   $B$. This will be carried out in the next two subsections. A more general analysis of the variation of $S_t$ as a function of all coupling constants is done in Subsection \ref{subsec:variations}.

\subsection{ Theorem~\ref{lem:Buyarov/Rakhmanov:99} revisited} \label{subsec:dynamical}

Our first task is to rewrite the differentiation formulas of Theorem~\ref{lem:Buyarov/Rakhmanov:99} in terms of the zeros   of $A$ and   $B$, which can be regarded as phase coordinates of a dynamical system. 

Recall that if $S_t$ has the form \eqref{assupport}, then its Robin measure $\omega=\omega_{S_t}$ (see the definition in Section \ref{subsect:equilibrium_analytic}) is  
 \begin{equation*} 
 d\omega_{S_t}(x)= \frac{1}{\pi i}\, \frac{h(x)}{A^{1/2}_+(x)}\, dx=\frac{1}{\pi }\, \left| \frac{h(x)}{\sqrt{A  (x)}}\right| \, dx, \quad x \in S_t, 
\end{equation*}	
while for the complex Green function $G(z)=G_{S_t}(z;\infty)$ with pole at infinity we have
\begin{equation} \label{greenF}	
G'(z)=-C^\omega(z)= \frac{h(z)}{A^{1/2} (z)}, \quad z\in \C\setminus S_t.
\end{equation}
Here $h$ is a real monic polynomial of degree $p-1$;  for $p=1$, $h\equiv 1$, while for $p>1$ it is uniquely determined  by conditions
 $$
 \int_{a_{2k}}^{a_{2k+1}} \frac{h(x)}{A^{1/2} (x) } \, dx =0, \quad k=1, \dots, p-1
 $$
(an ``integrable-systems-oriented'' reader will find these conditions already in  \cite{Lax/Levermore}, formulas (5.20)--(5.21)).
It follows in particular that for $p\geq 2$,
 \begin{equation*} 
h(x)=\prod_{j=1}^{p-1} (x-\zeta_j), \quad \text{with} \quad \zeta_j\in [a_{2j}, a_{2j+1}], \quad j=1, \dots, p-1.
\end{equation*}

In order to denote the differentiation with respect to the total mass $t$ we will use either $d/dt$ or the dot over the function, indistinctly. Recall that $a_k$ are the zeros of $A$, and denote by $b_k$ the zeros of $B$.
\begin{theorem}\label{thm:Dynamical System-t}
If  $A$ and $B$ have no common zeros then
\begin{equation}
\label{odeB}
\begin{split}
\dot{a}_k & =   \frac{2 h(a_k)}{A'(a_k) B(a_k)}, \quad  k=1, \dots, 2p, \\
\dot{ b}_k & =    \Res_{x=b_k} \frac{  h(x) }{A(x) B(x) }  , \quad k=1, \dots, \deg (B) = 2m-p-1.
\end{split}
\end{equation}
If additionally all zeros of $B$ are simple, then the second set of equations in \eqref{odeB} simplifies to
$$
\dot{ b}_k   =    \frac{  h(b_k) }{A(b_k) B'(b_k) }  , \quad k=1, \dots, \deg (B)= 2 m-p-1.
$$
\end{theorem}
\begin{proof}
Taking the square root in both sides of \eqref{charactMeasure1} and  differentiating the resulting formula with respect to $t$ (recall that $\varphi$ does not depend on $t$), we get
$$
\frac{\dot R_t(z)}{2R_t^{1/2}(z)} = \frac d{dt} C^{\lambda_t}(z).
$$
By Theorem \ref{lem:Buyarov/Rakhmanov:99} and \eqref{greenF},
$$
\frac d{dt} C^{\lambda_t}(z) = C^\omega(z)= -G'(z) = -\frac {h(z)}{A^{1/2} (z)} , \quad z\in \C\setminus S_t,
$$
and further, with the help of \eqref{AB-repr},
\begin{equation} \label{dynamics}	
B \, \dot A + 2 A \,\dot B =   -2 h \qquad \text{or}\qquad
\frac{ \dot A} A + 2 \, \frac{\dot B} B =  -\frac{ 2 h }{AB}.
\end{equation}
Assertions of the theorem follow by equating residues of the rational functions in the last equation above.
\end{proof}

\begin{remark}
A subset of equations from \eqref{odeB}  corresponding to $\dot{a_k}$ has been obtained in several places before, in particular, in the relevant work  of Bleher and Eynard \cite{MR1986409}, although using a different approach.  
That paper is, also probably, one of the first works where some analytic properties of phase transitions in Dyson gases were rigorously studied.
It contains several noteworthy results, such as a discussion of the string equations and the  double scaling limit of the  correlation functions when simultaneously the volume goes to infinity and the parameter $t$ approaches its critical value with an appropriate speed. 

Equations \eqref{odeB} are similar in form to a system of ODEs studied by Dubrovin in \cite{Dubrovin:1975fk} for the dynamics of the Korteweg-de Vries equation in the class of finite-zone or finite-band potentials. Curiously, Dubrovin's equations govern the evolution of the ``spurious'' poles of the diagonal Pad\'e approximants to rational modifications of a Markov function (Cauchy transform $C^\mu$) \cite{MR1918205}, and are equivalent to the equations obtained by one of the authors in \cite{Rakhmanov:1977uq} in terms of the harmonic measure of the support of $\mu$.
\end{remark}

\subsection{Classification of singularities} \label{subsec:classification}

We will see in the sequel that conditions of Theorem \ref{thm:Dynamical System-t} are satisfied for any $t$ except for a finite number of values. These values of $t$  may be called critical; they correspond to some of the phase transitions but not to all of them. There are other significant values,  also presenting phase transitions, which  are not critical in the above mentioned sense. 

As it follows from Theorems \ref{lem:Buyarov/Rakhmanov:99} and \ref{thm:generalparameter} (see also \cite{MR2187941} and \cite[Theorem 1.3. (iii)]{MR1744002}), the endpoints of the support of the equilibrium measure, its density function, and the corresponding equilibrium energy (or the infinite volume free energy) are analytic functions of the coefficients of the external field, except for a finite number of values where the analyticity breaks down. These critical points divide the parameter space  into different phases of the model and represent phase transitions. Some of them (but not all) correspond also to the change of topology of the support of the equilibrium measure.

Recall the classification of the singularities of the equilibrium measure we summarized in Subsection \ref{sec:revisited}: 
within each of the three types there is an infinite discrete collection of cases, depending on the order of vanishing of either the density $\lambda_t'$ or the  function $W^{\lambda_t}(x)-c_t$ in $\R\setminus S_t$. In this paper we restrict our attention to the ``generic'' singularities, characterized by the lowest possible order of vanishing. ``Non-generic'' singularities, with a higher order vanishing, can be seen as a confluent case of two or more generic ones, and require additional considerations.

From the point of view of the evolution in $t$ and with the $R$- (or $(A,B)$-) representation at hand these generic singularities are now  classified as follows. 
\begin{itemize}
\item  \textbf{Singularity of type I} 
is a bifurcation $b \rightarrow (a^+,a^-)$, representing a birth of the cut; it is the event at  a critical time $t=T$ when a simple zero of $B$ ceases to exist and at its place two new zeros of $A$ are born. Formally, at $t=T$ the inequality \eqref{equilibrium1},
$$
W_\varphi^{\lambda_t}(x)=  V^{\lambda_t} (x) +\varphi (x)  \geq \max_{x\in S_t}W_\varphi^{\lambda_t}(x) ,
$$
is no longer strict in $\R\setminus S_t$, and the equality is attained at some point   $b \in \R\setminus S_t$, where we will have $B(b)=0$. At this point a bifurcation of the zero $b$ occurs.

A significant property of this situation is that the phase transition occurs by saturation of the inequality \eqref{equilibrium1}; the moment $t=T$ of the bifurcation is not defined by the dynamical system, i.e.~it is not its singular point. All the phase parameters $a_j, b_k$ may be analytically continued through $t=T$. The solution of the system for $t>T$ would give us a critical, but not an equilibrium, measure.

 \item \textbf{Singularity of type II} 
 is the opposite event $(a^+,a^-) \rightarrow b $, or fusion of two cuts, consisting in the collision and subsequent disappearance of two zeros of $A$ (note that a zero of the complex Green function, ``trapped'' in the closing gap, disappears simultaneously), and an  appearance of a (double) zero of $B$, followed by the splitting of this real double zero into two complex simple zeros. The collision of two zeros of $A$ is a critical point of the dynamical system. 

We note that as a rare event (event of a higher co-dimension) it may happen that a number of  other cuts were present  in the vanishing gap $ (a^{-} ,a^+) $ 
immediately before the collision; they all disappear at the moment of collision of $a^-$ and $a^+$. This will be accompanied by an appearance for a moment of a zero of $B$ of an even multiplicity higher than two.
 
Thus, a double zero (or of an even higher multiplicity)  of $B$ is present at the moment of collision inside $S_t$.  
According to   \eqref{(A,B)-representation}, the density of $\lambda_t$ vanishes at these points with an integer even order. 
 
 \item \textbf{Singularity of type III} are the endpoints of $S_t$ where $R_t$ has a multiple zero; they correspond to the case when $A$ and $B$ in \eqref{AB-repr} have a common real zero.
 
\end{itemize}

Additionally, a special situation is created when two complex-conjugate simple zeros of $B$ collide on $\R\setminus S_t$, and either bounce back to the complex plane or continue their evolution as two real simple roots of $B$ (at this moment, two new local extrema of the total potential $W^{\lambda_t}$ in $\R\setminus S_t$ are born). At these values of the parameters the free energy is still analytic, so this is not a phase transition in the sense we agreed to use in this study, although the colliding zeros of $B$ lose analyticity with respect to $t$. In a certain sense, the singularity of type III is a limiting case of this phenomenon, when it occurs at an endpoint of $S_t$.

It is worth insisting that, according to \cite{MR1744002} and our analysis, the three singularities listed above are the only mechanisms of phase transitions in the parameter  $t$, described using the  $R$-representation. In this sense, no extra assumptions are made beyond the ``generic'' character of each singularity, as explained.

The reader should be aware of a certain freedom in our terminology regarding the singular points: we refer to singularities meaning both the value of the parameter $t$ at which a bifurcation occurs, and the point on $\R$ where the actual bifurcation takes place. We hope that the correct meaning in each case is clear from the context and will not lead the reader to confusion.
  
We return to the analysis of the local behavior at the singularities of the system in terms of $t= t_0$ in Section \ref{sec:phasetrans}.

\subsection{Other variations of the external field} \label{subsec:variations}

So far we have been regarding the equilibrium measure $\lambda_t$ and its $R$-  and $(A,B)$-representations as functions of the total mass $t$, assuming that $\varphi$ is fixed. Now we discuss a more general problem: the dependence of $\lambda$ and, correspondingly,  of $R$, $A$,  and $B$ from all the coefficients $t_j$ of the external field \eqref{phiPolyn}. It is a remarkable fact that the differentiation formulas with respect to the coupling constants $t_j$, $ j = 1,2,\dots$, have the same form as the differentiation formula   with respect to $t=t_0$,  and  that  they all can be obtained in a unified way.

We start with a simple technical  observation:
\begin{lemma}\label{existence_H}
Given  points $a_1<a_2<\dots<a_{2p}$ and the corresponding monic polynomial $A(z)= \prod_{k=1}^{2p}\,(z-a_k)$, there exist polynomials $h_j$, $\deg (h_j)=j+p-1$,  $j=0, 1,  \dots$,  uniquely determined by the following conditions:
\begin{equation}\label{cond2_hk}
h_0(z)=-z^{p-1}+\dots, \qquad  \frac{h_j(z)}{A^{1/2} (z) }   =j z^{j-1} +\mathcal O\left( \frac{1}{z^2}\right), \quad z\to \infty, \quad j\geq 1, 
\end{equation}
and 
\begin{equation}\label{cond1_hk}
\int_{a_{2k}}^{a_{2k+1}} \frac{h_j(x)}{A^{1/2} (x) } \, dx   =0, \quad k=1, \dots, p-1, \quad j=0, 1,  \dots.
\end{equation}
\end{lemma}
\begin{proof}
Clearly, $h_0(x)=-h(x)$, where $h$ is the numerator of $G'$ introduced in Subsection \ref{subsec:dynamical}, see \eqref{greenF}. 

Furthermore, for each $j\in \N$,  \eqref{cond2_hk} means that the coefficients corresponding to powers $z^{-1}$, $1$, $z, \dots, z^{j-2}$ of the Laurent expansion of $h_j A^{-1/2}$ at infinity vanish, which renders $j$ linear equations, additional to $p-1$ linear equations \eqref{cond1_hk} on the $p+j-1$ coefficients of $h_j$. The corresponding homogeneous linear equations are obtained by setting
$$
\frac{h_j(z)}{A^{1/2} (z) }   = \mathcal O\left( \frac{1}{z^2}\right),   \quad z\to \infty, 
$$
along with \eqref{cond1_hk}. In particular, every $h_j$ is of degree at most $p-2$, and according to \eqref{cond1_hk}, has a zero in each interval $(a_{2k}, a_{2k+1})$, $k=1,\dots, p-1$, which yields only the trivial solution for this system. 
\end{proof}

Since we are going to write all differentiation formulas in a unified way, we prefer to use here the notation
\begin{equation*} 
\partial_j := \frac{\partial }{\partial t_j}, \quad j=0, 1, \dots, 2m-1 ,
\end{equation*}
with $t_0=t$.
\begin{theorem}\label{thm:generalDiff}
Let the polynomial external field $\varphi$ be given by \eqref{phiPolyn}--\eqref{normalizationLeading}. Let also $\lambda_t=\lambda_t(\varphi,\R)\in \mathcal M_t$ denote the corresponding equilibrium measure of mass $t>0$, and let polynomials $A$ and $B$ be the $(A,B)$-representation of this equilibrium measure, see \eqref{(A,B)-representation}. 

Then 
\begin{equation} \label{derivParamGeneral}	
\partial_j \left( C^{\lambda_t} + \varphi'\right)(z)=\frac{h_j(z)}{A^{1/2} (z) }, \quad z\in \C\setminus S_t , \quad j=0, 1, \dots, 2m-1,
\end{equation}
where polynomials $h_j$ are given in Lemma \ref{existence_H}.

Moreover,
\begin{equation} \label{Euler1}	
\sum_{j=0}^{2m-1} t_j\,  \partial_j \left( C^{\lambda_t} + \varphi'\right)(z)=\left( C^{\lambda_t} + \varphi'\right)(z), \quad z\in \C\setminus S_t .
\end{equation}

In consequence, 
\begin{equation} \label{dynamics0}	
B \, \partial_j A + 2 A \,\partial_j B =  2 h_j, \quad j=0, 1,  \dots, 2m-1,
\end{equation}
and
 \begin{equation} \label{Euler2}	
\sum_{j=0}^{2m-1} t_j\,  h_j (z)=A(z) B(z).
\end{equation}
\end{theorem} 
\begin{proof}
For $j=0$, formulas \eqref{derivParamGeneral} and \eqref{dynamics0} are just a restatement of Theorem \ref{thm:Dynamical System-t}.
 
Furthermore, from Theorem \ref{thm:generalparameter} it follows that for $j=1, \dots, 2m-1$,
$$
\partial_j C^{\lambda_t}(z)=C^{\partial_j \lambda_t}(z)=C^{\omega_{j}}(z),
$$
where $\omega_j$ is a signed measure on $S_t$ satisfying 
\begin{equation} \label{equilCondonK}	
\omega_j(S_t)=0, \quad V^{\omega_j}(x)+x^j=\hat c_j =\const \text{ on } S_t.
\end{equation}

It follows from \eqref{cond2_hk}--\eqref{cond1_hk} that the multivalued analytic function
$$
G_j(z)=\int_{a_{2p}}^z \frac{h_j(x)}{A^{1/2} (x) } \, dx, \quad z\in \C\setminus S_t,
$$
has a single-valued real part, which is continuous in $\C$ and satisfies
$$
\Re G_j(z)=0,\quad x\in S_t, \quad \text{ and } \quad G_j(z)=z^j + \mathcal O\left( 1 \right), \quad z\to \infty. 
$$
Taking into account \eqref{equilCondonK} we conclude that
$$
\Re G_j(z)=V^{\omega_j}(z)+z^j-\hat c_j , \quad z\in \C.
$$
Since for $j=1, \dots, 2m-1$, 
$$
\partial_j \varphi(x)=x^j, \quad \partial_j \varphi'(x)=j x^{j-1},
$$
identity of \eqref{derivParamGeneral} follows for all remaining $j$'s.

Finally, \eqref{homogeneity} shows that $C^{\lambda_t} + \varphi'$ is a homogeneous function of degree 1 of the vector of coupling constants $t_0, t_1, \dots$, so that \eqref{Euler1} is just Euler's theorem for such a function. Formula \eqref{Euler2} is obtained by replacing \eqref{charactMeasure1} and \eqref{derivParamGeneral}  correspondingly in the left and right hand sides of \eqref{Euler1}.
\end{proof}

Evaluating  \eqref{Euler2} at the zeros of $R_t$ we obtain a set of algebraic identities,
\begin{align}\label{hodograph1}
&\sum_{j=0}^{2m-1} t_j\,  h_j (a_k)  = 0, \quad k=1, \dots, 2p, \\
& \sum_{j=0}^{2m-1} t_j\,  h_j (z)\bigg|_{B(z)=0}   = 0, \nonumber 
\end{align}
called \emph{hodograph equations} in \cite{MR2629605}. 
Solving them we could find the main parameters of the equilibrium measure and of its support. However, their explicit character is misleading: in the multi-cut case the dependence of the coefficients of $h_j$ from the coupling constants $t_j$'s is highly transcendental, and as the authors of \cite{MR2629605}  point out, except for the simplest examples, equations \eqref{hodograph1} are extremely difficult to solve, ``even by numerical methods''.

\begin{remark} \label{remark:Riemann}  
Alternatively, following the general methodology put forward in \cite{MR961760}, we can derive the identities on the endpoints of the connected components of $S_t$ considering the hyperelliptic Riemann surface $\mathcal R$ of $w^2=A(z)$ and its deformations depending on the set of coupling constants  $t_j$, imposing the condition that the partial derivatives with respect to the parameters $t_j$ of the corresponding normalized Abelian differentials of this surface are given by a meromorphic differential $\mathcal D$ on $\mathcal R$. This is equivalent to the set of the so-called  \emph{Whitham equations}  \cite{MR2061477} on $\mathcal R$. Actually, polynomials $h_j$ defined in Lemma \ref{existence_H}, appear in the explicit representation of these normalized Abelian differentials of the third  ($j=0$)  and second kind ($j\in \N$) on $\mathcal R$. The key connection with the equilibrium problem is provided by identity \eqref{charactMeasure1}, which shows that  the  differential $(\varphi'(z)+ R_t^{1/2}(z) )\, dz$, with $\varphi$ given by \eqref{phiPolyn}, can be extended as the above mentioned meromorphic differential $\mathcal D$ on $\mathcal R$. 
This approach was used in \cite{MR2240464} to obtain in particular an analogue of \eqref{derivParamGeneral}, and developed further  in  \cite{MR2629605}.
 \end{remark}

Again, a direct consequence of Theorem \ref{thm:generalDiff} is the possibility to rewrite the differentiation formulas \eqref{dynamics0} in terms of the zeros $a_k$ of $A$ and $b_k$ of $B$: 
\begin{theorem}\label{thm:odeBis}
If under assumptions of Theorem \ref{thm:generalDiff},  $A$ and $B$ have no common zeros then for $j=0, 1,  \dots, 2m-1$, 
\begin{equation}
\label{odeBis}
\begin{split}
\partial_j a_k & =-   \frac{2 h_j(a_k)}{A'(a_k) B(a_k)}, \quad  k=1, \dots, 2p, \\
\partial_j b_k & =  -  \Res_{x=b_k} \frac{  h_j(x) }{A(x) B(x) }  , \quad k=1, \dots, \deg (B).
\end{split}
\end{equation}
If additionally all zeros of $B$ are simple, then the second set of equations in \eqref{odeBis} simplifies to
$$
\partial_j b_k   =  -  \frac{  h_j(b_k) }{A(b_k) B'(b_k) }  , \quad k=1, \dots, \deg (B).
$$
\end{theorem}
\begin{proof}
It is a consequence of \eqref{dynamics0} that
\begin{equation} \label{consequenceThm4}	
\frac{ \partial_j A}{A}(x) + 2 \frac{\partial_j B}{B}(x) =  2 \frac{h_j}{AB}(x), \quad j=0, 1,  \dots, 2m-1.
\end{equation}
Observe that
$$
\frac{ \partial_j A}{A}(x) =   \partial_j \log(A ) (x)=\sum_{k=1}^{2p}\partial_j \log(x-a_k ) =- \sum_{k=1}^{2p}\frac{\partial_j a_k }{x-a_k}.
$$
Analogous formula is valid for $\partial_j B/B$. Hence, both the left and the right hand sides in \eqref{consequenceThm4} are rational functions in $x$, with possible poles only at the zeros of $A$ and $B$. The necessary identities are established by comparing the corresponding residues at each pole.

For instance, with the assumption that the zeros of $A$ and $B$ are disjoint, the residue of the left hand side of \eqref{consequenceThm4} at $x=a_k$ is equal to $-\partial_j a_k$, which yields the first set of equations in \eqref{odeBis} (recall that by construction, all zeros of $A$ are simple). The analysis of the residues at $x=b_k$ gives us the remaining identities.
\end{proof}
The proof  shows how the statement can be modified in the case of   coincidence of some zeros of $A$ and $B$ (in other words, in the case of roots of $R_t$ of degree higher than 2). Moreover, the evolution of $R_t$ is such that as long as the right-hand sides in \eqref{odeBis} remain bounded, all zeros of $R_t$, and hence, $R_t$ itself, are in $C^1$.

We can rewrite the equations on $a_k$'s in \eqref{odeBis} in a weaker form:
$$
h_i(a_k)\, \partial_j a_k  =    h_j(a_k) \, \partial_i a_k, \quad  k=1, \dots, 2p, \quad i, j=0, 1,  \dots, 2m-1,
$$
which are the \emph{Whitham equations in hydrodynamic form} (see \cite{Lax/Levermore}, and also \cite[Eq.~(77)]{MR2629605}).
 
\begin{example}
The simplest case to consider is when $A(x)=(x-a_1)(x-a_2)$, so that $S_t$ consists of a single interval $[a_1,a_2]$. In this situation, $h_0(x)=-1$, $h_1(x)=x-(a_1+a_2)/2$, and under assumptions that $A$ and $B$ have disjoint zeros we obtain from  \eqref{odeBis} that for $ j=1,   2$,
\begin{equation*}
\begin{split}
\partial_0 a_j   =    \frac{2 }{A'(a_j) B(a_j)} \quad \text{and} \quad \partial_1 a_j   =  -   \frac{  1}{  B(a_j) }  .
\end{split}
\end{equation*}
This can be rewritten as a system of PDE, 
\begin{equation*}
\begin{split}
\partial_1 a_1 - \frac{a_2-a_1 }{2}\, \partial_0 a_1 & = 0,     \\
\partial_1 a_2 + \frac{a_2-a_1 }{2}\, \partial_0 a_2 & = 0,
\end{split}
\end{equation*}
which is a rescaled form of the continuum limit of the Toda lattice  in Riemann invariant form (see \cite[Chapter 2]{Deift98}). This set of equations for the endpoints of $S_t$ in the one-cut case appears also in \cite{MR2350906}, \cite{MR2262808}, \cite{MR1877482}.

In an analogous fashion,
\begin{align*}
h_2(x)& =2 x^2 -(a_1+a_2) x -\frac{(a_1-a_2)^2}{4}, \\
h_3(x)& =3 x^3   - \frac{3}{2} (a_1 + a_2) x^2 - \frac{3}{8} (a_1 - a_2)^2 x -\frac{3}{16} (a_1 - a_2)^2 (a_1 + a_2),
\end{align*}
which yields the following differential relations:
\begin{equation} \label{newDiffRelat2_1}
\begin{array}{l} \displaystyle 
\partial_2 a_1 =- \frac{ 3a_1+a_2 }{2 B(a_1)} ,     \\
\displaystyle \partial_2 a_2 = -\frac{\strut a_1+3 a_2 }{2 B(a_2)} ,   
\end{array} \quad \text{and} \quad 
\begin{array}{l} \displaystyle 
\partial_3 a_1 =- \frac{ 3 (4a_1^2 +(a_1+a_2)^2) }{8 B(a_1)} ,     \\
\displaystyle \partial_3 a_2 = -\frac{\strut 3 (4a_2^2 +(a_1+a_2)^2)  }{8 B(a_2)} .
\end{array} 
\end{equation}
Again, these can be rewritten as dynamical systems,
\begin{equation*} 
\begin{array}{l} \displaystyle 
\partial_2 a_1 - \frac{(a_2-a_1)(3a_1+a_2) }{4}\, \partial_0 a_1   = 0,     \\
\displaystyle
\partial_2 a_2 + \frac{\strut (a_2-a_1)(a_1+3 a_2) }{4}\, \partial_0 a_2   = 0,
\end{array}
\end{equation*}
and
\begin{equation} \label{newDiffRelat2_3}
\begin{array}{l} \displaystyle 
\partial_3 a_1 - \frac{3(a_2-a_1)(4a_1^2 +(a_1+a_2)^2) }{16}\, \partial_0 a_1   = 0,     \\[2mm]
\displaystyle
\partial_3 a_2 + \frac{3(a_2-a_1)(4a_1^2 +(a_1+a_2)^2) }{16}\, \partial_0 a_2   = 0.
\end{array}
\end{equation}
\end{example}

\section{Local behavior at phase transitions in the polynomial external field} \label{sec:phasetrans}

In this section we take a closer look at the behavior of the equilibrium energy for the external field \eqref{phiPolyn}--\eqref{normalizationLeading} when these phase transitions occur, considering only the variation of the total mass (equivalently, temperature or time) $t$. As explained in Section~\ref{subsec:classification},  we restrict our attention to the case of generic singularities. Along this section we denote by $t=T$ the critical time at which the phase transition occurs. Using the $R$- and $(A, B)$-representations of the equilibrium measure $\lambda_t$  introduced in Section \ref{subsec:R-representation}, the phase transitions are classified as follows: 
\begin{itemize}
\item \textbf{Singularity of type I:} at a time $t=T$ a real zero $b$ of $B$ is an isolated point of the set $S^T\setminus S_T$ (in other words, $W_\varphi^{\lambda_T}(b)=c_T$, $b\notin S_T$; see \eqref{equilibrium1});  the only additional assumption is that for $t=T$, $b$ is a simple zero of $B$;
 \item \textbf{Singularity of type II:} at a time $t=T$, a real zero $b$ of $B$ (of even multiplicity) belongs to the interior of the support $S_T$; according to \eqref{(A,B)-representation}, the density of $\lambda_T$ vanishes in the interior of its support. The only additional assumption is that for $t=T$, $b$ is a double zero of $B$;
 \item \textbf{Singularity of type III:} at a time $t=T$, polynomials $A$ and $B$ have a common real zero $a$; the only additional assumption is that   $a$ is a double zero of $B$, so that $\lambda_T'(x)=\mathcal O(|x-a|^{5/2})$ as $x\to a$.
 \end{itemize}

This classification coincides with the one from  \cite{MR1744002}. It was shown there (see \cite[Theorem 1.3(iv)]{MR1744002}) that this  ``static'' definition is equivalent to the following ``dynamic'' description of the three types of singularities, again in terms of the $(A, B)$-representation:
\begin{itemize}
\item \textbf{Singularity of type I:} at a time $t=T$ a real zero $b$ of $B$ (a double zero of $R_t$) splits into two simple zeros $a_-<a_+$, and the interval $[a_-,a_+]$ becomes part of $S_t$ (\emph{birth of a cut});  the only assumption is that for $t$ in a left neighborhood of $t=T$, $b$ is a simple zero of $B$;
 \item \textbf{Singularity of type II:} at a time $t=T$ two simple zeros $a_{2s}$ and $a_{2s+1}$ of $A$ (simple zeros of $R_t$) collide (\emph{fusion of two cuts}).
 \item \textbf{Singularity of type III:} at a time $t=T$ a pair of complex conjugate zeros $b$ and $\overline b$ of $B$ (double zeros of $R_t$) collide with a simple zero $a$ of $A$, so that $\lambda_T'(x)=\mathcal O(|x-a|^{5/2})$ as $x\to a$.
 \end{itemize}
 
Additionally, in Subsection \ref{subs:IV} we analyze the scenario when at a time $t=T$ a pair of complex conjugate zeros $b$ and $\overline b$ of $B$ (double zeros of $R_t$) collide at $b_0\in \R\setminus S_t$ and either bounce back to the complex plane or become two simple real zeros $b_-<b_+$ of  $B$ (these real zeros are new local extrema of the total potential $W^{\lambda_t}$ on $\R\setminus S_t$). 
 
Obviously, for a general polynomial external field some of these phase transitions can occur simultaneously: for a given value $t=T$ two or more cuts could merge, while a new cut is open elsewhere, together with a type III singularity at some endpoints of $S_T$. Still, the basic ``building blocks'' of all phase transitions are precisely the  cases described above, which we proceed to study.

It was established in \cite{MR2187941} that in the second case, when at $t=T$ two components of $S_t$ merge into a single cut in such a way that $\varphi$ is regular for $t\in (T-\varepsilon, T)$, the equilibrium energy can be analytically continued through  $t=T$ from both sides. For the case of a quartic potential $\varphi$ it was shown in  \cite{MR1986409} that  the energy and its first two derivatives are continuous at $t=T$, but the third derivative has a finite jump. 
This is  a third order phase transition, observed also in a circular
ensemble of random matrices \cite{PhysRevD.21.446}. With respect to the singularity of type III, it is mentioned in \cite{Bleher:2008fk} that in this case the free energy is expected to have an algebraic singularity at $t = T$, ``but this problem has not been studied yet in details''.

In this section we extend this result to a general multi-cut case and show that for the three types of singularities the energy and its first two derivatives  (but in general not the third one) are continuous at the critical value $t=T$. The character of discontinuities is also analyzed, and we can summarize our findings as follows: in all cases there is a parameter, that we call $\delta$, expressing geometrically the ``distance'' to the singularity. In the case of a birth of a new cut, this is the size of this new component of $S_t$; for a fusion of two cuts, this is the size of the vanishing gap, while for the singularity of type III we can take it as the distance between two colliding zeros of $R_t$. In the three cases the first two derivatives of the equilibrium energy $I_{\varphi }[\lambda_t]$ are continuous, while the third derivative has a discontinuity $\Delta \dot{\rho}(S_t)$. We gather in the table below a rough information about the character of the dependence of these parameters from $\Delta t =|t-T|$ at the corresponding singularities, see Subsections \ref{subs:birth}--\ref{subs:III} for a detailed asymptotics and further discussion.
$$
\begin{tabular}{|c|c|c|c|}
\hline
 & \text{\textbf{Type I}} & \text{\textbf{Type II}} & \text{\textbf{Type III}}    \\[1mm] \hline
$\bm{\delta}$ & $\mathcal O(\sqrt{\Delta t /\log(\Delta t)}) \strut $  & $ \mathcal O(\sqrt{\Delta t })$ & $ \mathcal O( (\Delta t)^{1/3} )$   \\[1mm] \hline
$\bm{\Delta \dot{\rho}(S_t)}$ & $\mathcal O((\Delta t \log^2(\Delta t) )^{-1})$  &   $\mathcal O(1) \strut $  & $ \mathcal O( (\Delta t)^{-2/3} )$   \\[1mm] \hline
\end{tabular}
$$

Clearly, all zeros $a_k$, $b_k$ of $R_t$ (and hence, $R_t$ itself) are continuous functions of $t$. Moreover, it follows from \eqref{odeB}  that when a singularity occurs, all $a_k$'s and $b_k$'s not involved in the phase transition are smooth functions of $t$, with the degree of smoothness, roughly speaking, 1 more than the smoothness of $\delta$ in the corresponding column of the table above.

We insist that not all results here are new; in particular, phase transitions of Types II and III have been carefully analyzed in \cite[Section 8.1]{MR1744002}, and several entries in the table above are consistent with their findings.

\subsection{Singularity of type I: birth of a cut}\label{subs:birth}

Recall that at   $t=T$,  a real and simple zero $b$ of $B$ is an isolated point of the set $S^T\setminus S_T$, and that by \cite{MR1744002}, it implies that at this critical value of the parameter $t$,  a simple real zero $b\in \R\setminus S_T$ of $B$ (a double zero of $R_t$) splits into two simple zeros $a_-<a_+$, and the interval $[a_-,a_+]$ becomes part of $S_t$. 

First of all we want to estimate the size of the new cut as a function of $t$.
We use the notation introduced above, indicating explicitly the dependence from the parameter $t$. Let us remind the reader in particular that  polynomial $h$ is the numerator of the derivative of the complex Green function, see \eqref{greenF}.  From our assumptions it  follows that  for a small $\varepsilon >0$, there exist polynomials $\bm A$, $\bm B$ and $\bm h$, such that
\begin{align*}
A(x;t)&=\begin{cases}
\bm A(x;t), & t\in (T-\varepsilon, T],\\
  (x-a_-)(x-a_+)\bm A(x;t), & t\in (T, T+\varepsilon),\\
\end{cases} \\
B(x;t)&=\begin{cases}
(x-b) \bm B(x;t), & t\in (T-\varepsilon, T],\\
   \bm B(x;t), & t\in (T, T+\varepsilon),\\
\end{cases} \\ 
h(x;t)&=\begin{cases}
\bm h(x;t), & t\in (T-\varepsilon, T],\\
  (x-\zeta)\bm h(x;t), & t\in (T, T+\varepsilon),\\
\end{cases}
\end{align*}
where $a_\pm $, $b$ and $\zeta$ are real-valued continuous functions of $t$ such that 
$$
a_-(t=T+)=a_+(t=T+)=b(t=T-)=\zeta(t=T+);
$$ 
we denote this common value by $b_0$; this is the place where the new cut is born at $t=T$. Remember that it is imposed only by the saturation of the inequality constraint in \eqref{equilibrium1} outside of $S_t$. Polynomials $\bm A$, $\bm B$ and $\bm h$ are continuous with respect to the parameter $t\in (T-\varepsilon, T+\varepsilon)$, but represent, generally speaking,  different real-analytic functions of $t$ for $t<T$ and $t>T$. 

It is convenient to introduce two new variables for $t>T$:
\begin{equation} \label{def:zianddelta}	
\xi =\xi(t)=(a_++a_-)/2, \qquad \delta =\delta(t)=(a_+-a_-)/2.
\end{equation}

\begin{theorem}\label{thm:Deltaphase1}
The function, defined piecewise as
$$
\begin{cases}
b(t), & t\leq T, \\
\xi(t), & t>T,
\end{cases}
$$
is  $C^1$ in $(T-\varepsilon, T+\varepsilon)$. Furthermore, there exists a constant $C>0$, independent of $t$, such that
\begin{equation} \label{asymptoticsDeltaNewCut}	
\delta^2(t) = -C \frac{t-T}{\log(t-T)}\,\left(1+\mathcal O\left(\frac{\log \log (t-T)}{\log (t-T)}\right) \right), \quad t\to T+.
\end{equation}
\end{theorem}
Observe that formula \eqref{asymptoticsDeltaNewCut} is consistent with the $\log(n)/n$ scaling used e.g.~in \cite{MR2427458}.
\begin{proof}
Let us denote $\bm Q(x)=\bm A (x)\bm B(x)$, omitting from the notation when possible the explicit dependence on $t$. From \eqref{odeB} it follows that
\begin{align}
\nonumber 
\dot{ b} &=  \frac{h (b)}{\bm Q(b)}=  \frac{\bm h (b)}{\bm Q(b)}, \quad t\in (T-\varepsilon, T],
\\ 
\label{localA}	
\dot{a_\pm} &= \pm \frac{2 h (a_\pm )}{(a_+-a_-)\bm Q(a_\pm)} =  \pm \frac{2 (a_\pm -\zeta) \bm h (a_\pm )}{(a_+-a_-)\bm Q(a_\pm)}, \quad t\in (T, T+\varepsilon).
\end{align}
Adding both equations in \eqref{localA} we get
\begin{equation*}
\begin{split}
\dot{\xi } = & \frac{ 1}{a_+-a_-} \, \left( \frac{ h (a_+ )}{ \bm Q(a_+)} - \frac{ h (a_- )}{ \bm Q(a_-)} \right) \\
= &  \frac{ 1}{a_+-a_-} \, \left( \frac{(a_+-\zeta)  \bm h (a_+ )}{ \bm Q(a_+)} - \frac{ (a_--\zeta)  \bm h (a_- )}{ \bm Q(a_-)} \right) , \quad t\in (T, T+\varepsilon), 
\end{split}
\end{equation*}
which shows  that
$$
\lim_{t\to T+} \dot{\xi }  =  \left( \frac{ h( x)}{ \bm Q(x)}  \right)' \bigg|_{x=b_0} =  \frac{\bm h(b_0)}{\bm Q(b_0)} = \lim_{t\to T-}\dot{b }  ,
$$
proving the first assertion of the theorem. 

Analogously, subtracting equations in \eqref{localA} yields that for $t\in (T, T+\varepsilon)$,
\begin{equation*} 
\frac{d}{dt}\, (\delta^2)    =  \frac{ h (a_+ )}{ \bm Q(a_+)} + \frac{ h (a_- )}{ \bm Q(a_-)}  = \frac{(a_+-\zeta) \bm h (a_+ )}{ \bm Q(a_+)} + \frac{ ( a_--\zeta) \bm   h (a_- )}{ \bm Q(a_-)} ,
\end{equation*}
or equivalently,
\begin{equation} \label{localDelta2}	
\begin{split}
\frac{d}{dt}\, (\delta^2)
  = ( a_- - \zeta) \left( \frac{ \bm h (a_+ )}{  \bm Q(a_+)} +\frac{ \bm h (a_- )}{  \bm Q(a_-)}  \right) +2\delta\, \frac{\bm h(a_+)}{ \bm Q(a_+)}  .\end{split}
\end{equation}
Let us denote by $a$ the largest zero of $\bm A$ satisfying $a<b_0$. Then condition \eqref{cond1_hk} reads as
\begin{equation} \label{asymptDelta0}	
\int_a^{a_-} \frac{x-\zeta}{\sqrt{A(x;t)}} dx =0 \quad \Rightarrow \quad (\zeta- a_-)^{-1 }= \int_a^{a_-} \frac{dx}{\sqrt{A(x;t)}}  \left( \int_a^{a_-} \frac{x-a_-}{\sqrt{A(x;t)}} dx \right)^{-1}.
\end{equation}
Clearly,
$$
\lim_{t\to T+} \int_a^{a_-} \frac{x-a_-}{\sqrt{A(x;t)}} dx = \int_a^{b_0} \frac{dx }{\sqrt{\bm A(x;T)}} dx \neq 0.
$$
On the other hand,
\begin{align*}
\int_a^{a_-} \frac{dx}{\sqrt{A(x;t)}}  = &\frac{1}{\sqrt{\bm A(a_-;t)}} \int_a^{a_-} \frac{dx}{\sqrt{(a_- -x)(a_+-x))}} \\
& + \int_a^{a_-} \frac{1}{\sqrt{(a_- -x)(a_+-x))}} \left( \frac{1}{\sqrt{\bm A(x;t)}} -\frac{1}{\sqrt{\bm A(a_-;t)}} \right)dx \\
& = \frac{1}{\sqrt{\bm A(a_-;t)}} \log\frac{1}{\delta} + \mathcal O(1), \quad \delta\to 0+.
\end{align*}
From \eqref{localDelta2} and \eqref{asymptDelta0} it follows now that  if we denote
$$
C =  4 \sqrt{\bm A(b_0;T)} \,  \frac{ \bm h (b_0 )}{  \bm Q(b_0)}  \int_a^{b_0} \frac{dx }{\sqrt{\bm A(x;T)}} dx > 0,
$$
then 
\begin{equation} \label{limitDelta}	
\frac{d}{dt}\, (\delta^2) =  \frac{-C +o(1)}{ \log (\delta^2) + \mathcal O(1)} + 2\delta\, \frac{\bm h(b_0)}{ \bm Q(b_0)}  , \quad t\to T+.
\end{equation}
Notice that the solution of the ODE $f'(x)=a/(\log(f(x)+b)$, with $f(0+)=0$, satisfies
$$
f(x) =\frac{a}{ \log(f(x)) +b-1 }=\frac{ax}{\log(x)}\left( 1+\mathcal O\left(\frac{\log \log x}{\log x}\right)\right), \quad x\to 0+,
$$
so that \eqref{limitDelta} implies  \eqref{asymptoticsDeltaNewCut}.
\end{proof}

We turn next to the asymptotic behavior of the Robin constant $\rho(S_t)$ of the support $S_t$, which according to Theorem \ref{lem:Buyarov/Rakhmanov:99} is the second derivative of the infinite volume free energy. Let us study first the following model situation, from which the general conclusion is readily derived.

Assume that $E\subset \R$ is a union of a finite number of disjoint real intervals, and $\Delta_\delta =[\xi-\delta, \xi + \delta]$ is disjoint with $E$ for $\delta <\varepsilon$. We denote $E_\delta=E\cup \Delta_\delta$, and slightly abuse notation, writing $\rho(\delta)$ instead of $\rho(E_\delta)$. 
Observe that with this notation, $\rho(0)=\rho(E)$. Our goal is to study the behavior of $\rho(\delta)$ as $\delta\to 0+$. 

As usual, $g_E(z,a)$ stands for the Green's function of $\overline{\C}\setminus E$ with a pole at $a\notin E$. The well-known relation holds:
\begin{equation} \label{Green}	
g_E(z,\xi) = \log\frac{1}{|z-\xi|} +\gamma(\xi) + o(1), \quad z\to \xi,
\end{equation}
which defines the constant $\gamma(\xi)$.
\begin{lemma}\label{lem:localRobin}
There exists a constant $C>0$ (uniform in $E$ and $\xi$) such that for $\delta <\varepsilon$, 
$$
\left| \rho(\delta) -  \rho(0) + \frac{g^2_E(\xi, \infty) }{\gamma(\xi) + \log(2/\delta)} \right| \leq C \delta.
$$
\end{lemma}
\begin{proof}
Let $\mu$ be the balayage of the unit point mass at $\xi$ onto $E$, so that (see \cite{Landkof:72})
\begin{equation}\label{balayage}
V^\mu (z) = \log \frac{1}{|z-\xi|} + g_E(\xi, \infty).
\end{equation}
Comparing it with \eqref{Green} we conclude that
\begin{equation} \label{identityForGreen}	 
g_E(z,\xi) =g_E(\xi, \infty) -V^\mu(z) -\log |z-\xi|.
\end{equation}
Indeed, the function in the right hand side is harmonic in $\overline{\C}\setminus (E\cup \{\xi\})$, vanishing on $E$, and has the appropriate logarithmic behavior when $z\to \xi$.

Recall that we denote by $\omega_E$ and $\omega_{\Delta_\delta}$ be the Robin measures of $E$ and $\Delta_\delta$, respectively. For a  parameter $m>0$ define
\begin{equation}
\label{defSigma}
\sigma = \sigma_\delta = \omega_E + m\, \omega_{\Delta_\delta} - m \mu;
\end{equation}
observe that since $\mu'$ is bounded on $E$, and $\omega_E'$ is bounded away from zero on $E$, $\sigma$ is a positive measure for  sufficiently small values of $m>0$. Consider the potential $V^\sigma$ for such values of $m$; we want to find its bounds on $E_\delta$.

Let us start with $x\in E$. We have
\begin{equation}\label{firstEstimateRobin}
\begin{split}
V^\sigma(x)-\rho(0) = & \left( V^{\omega_E}(x)-\rho(0)\right) - m \left( V^{\mu}(x)- \log \frac{1}{|x-\xi|} \right) \\
& + m \left( V^{\omega_{\Delta_\delta}}(x)- \log \frac{1}{|x-\xi|} \right) .
\end{split}
\end{equation}
The first parenthesis in the right hand side is $=0$ by definition of $\omega_E$; the second one is equal to $g_E(\xi, \infty)$, where we use \eqref{balayage}. So, we need to estimate the third parenthesis.

We have
$$
V^{\omega_{\Delta_\delta}}(x)=\log \left| \frac{2}{x-\xi + \sqrt{(x-\xi)^2-\delta^2}}\right|, 
$$
so that
$$
V^{\omega_{\Delta_\delta}}(x)- \log \frac{1}{|x-\xi|} =\log  \frac{2}{1 + \sqrt{1-\left(\frac{\delta}{x-\xi}\right)^2}}, \quad |x-\xi|>\delta.
$$
Using the straightforward bounds 
$$
\log \frac{2}{1+\sqrt{1-s^2}} =\log \left( 1+ \frac{s^2}{(1+\sqrt{1-s^2})^2}\right)<s^2
$$
we conclude that
\begin{equation*}
0< V^{\omega_{\Delta_\delta}}(x)- \log \frac{1}{|x-\xi|} < \left(\frac{\delta}{x-\xi} \right)^2 , \quad |x-\xi|>\delta.
\end{equation*}
Gathering all estimates in \eqref{firstEstimateRobin} we see that 
\begin{equation}
\label{estimate2}
-m g_E(\xi, \infty) < V^\sigma(x)-\rho(0) < -m g_E(\xi, \infty) + m \left(\frac{\delta}{\dist (\xi, E)} \right)^2 ,  \quad x\in E.
\end{equation}
Now we look for similar bounds on $\Delta_\delta$. Observe that
\begin{equation*}
V^\sigma(\xi)-\rho(0) =   \left( V^{\omega_E}(\xi)-\rho(0)\right) - m   V^{\mu}(\xi)   + m   V^{\omega_{\Delta_\delta}}(\xi) .
\end{equation*}
By the well-known property of $\omega_E$,  the first parenthesis is  equal to $-g_E(\xi, \infty)$;  by identity \eqref{identityForGreen}, 
$$
V^{\mu}(\xi) =g_E(\xi, \infty)- \gamma(\xi), 
$$ while by the equilibrium condition \eqref{equilibriumRobin}, $
V^{\omega_{\Delta_\delta}}(\xi)=\log (2/\delta)$. Thus,
$$
V^\sigma(\xi)-\rho(0) = -g_E(\xi, \infty) - m \left( g_E(\xi, \infty)- \gamma(\xi)\right) + m \log \frac{2}{\delta}.
$$
Now we set the value of the parameter $m$ to guarantee that
\begin{equation}
\label{valueAtxi}
V^\sigma(\xi)-\rho(0)  = - m g_E(\xi, \infty) ;
\end{equation}
in other words, we take
\begin{equation*}
m =\frac{g_E(\xi, \infty) }{\gamma(\xi) + \log \frac{2}{\delta}}>0.
\end{equation*}
Observe that $m\to 0+$ for $\delta \to 0+$, so that we can choose $\varepsilon>0$ such that $\sigma_\delta$ defined in \eqref{defSigma} is a probability measure for $0<\delta<\varepsilon$.

We are ready to establish the sought bounds on $\Delta_\delta$. Since $V^{\omega_{\Delta_\delta}}$ is constant on $\Delta_\delta$, we have
$$
V^\sigma(x)- V^\sigma(\xi)= \left(V^{\omega_E}(x) - V^{\omega_E}(\xi)\right) - m \left(V^{\mu}(x) - V^{\mu}(\xi)\right).
$$
Clearly, both terms are harmonic on ${\Delta_\delta}$, so we conclude that
\begin{equation}
\label{valueAtDelta}
\left|  V^\sigma(x)- V^\sigma(\xi) \right| \leq C_1 \delta, \quad \delta \to 0+,
\end{equation}
where $C_1$ is uniform with respect to small variations of $E$ and $\xi$.

It remains to gather \eqref{estimate2},  \eqref{valueAtxi} and \eqref{valueAtDelta}, and use the well-known property 
\begin{equation*} 
\min_{x\in E_\delta} V^\sigma(x) \leq \rho(\delta) \leq \max_{x\in E_\delta} V^\sigma(x),
\end{equation*}
which concludes the proof of the lemma.
\end{proof}

In order to use this result in the situation of a birth of a cut it is enough to notice that $\bm Q$ is a differentiable function of $t$ at $t=T$, and so is $\xi(t)$. Taking into account \eqref{asymptoticsDeltaNewCut} we get
$$
  \rho(S_t) -  \rho(S_T) = -\frac{ g^2_{S_T}(b_0, \infty)}{2 }\, \frac{ 1}{  \log(t-T)}\left(1+\mathcal O\left(\frac{\log \log (t-T)}{\log (t-T)}\right) \right), \quad t\to T_+.
$$
In particular, $\rho(S_t)$ is continuous  but non-differentiable at $t=T$:
\begin{theorem}\label{thm:RhoPhase1}
With the notations above,
\begin{equation} \label{infiniteJumpNewCut}	
\frac{d}{dt}  \rho(S_t) = -\frac{ g^2_{S_T}(b_0, \infty)}{2 }\, \frac{ 1}{ (t-T) \log^2(t-T)}\left(1+\mathcal O\left(\frac{\log \log (t-T)}{\log (t-T)}\right) \right) , \quad t \to T+,
\end{equation}
\end{theorem}
In fact, $\frac{d}{dt}  \rho(S_t) $ has an infinite jump at $t=T$, since 
$$
\lim_{t\to T-} \frac{d}{dt}  \rho(S_t) 
$$
obviously exists and is finite. 

Expression \eqref{infiniteJumpNewCut} is consistent with the $\mathcal O\left((t-T)^{-1} \log^{-2}(t-T) \right) $ jump found in \cite{MR2629605} for the quartic potential (the simplest transition from one to two cuts, see Section \ref{sec:Quartic}).

\subsection{Singularity of type II: fusion of two cuts}\label{subs:merge}

We turn our attention  to the critical transition when two intervals from $S_t$ collide and merge into a single interval. More precisely, we assume that at a time $t=T$, a real and double zero $b$ of $B$ belongs to the interior of the support $S_T$; by \cite[Theorem 1.3(iv)]{MR1744002}, this implies that at  $t=T$, two simple zeros $a_{2s}$ and $a_{2s+1}$ of $A$ (simple zeros of $R_t$) coalesce. 

Again, we start by estimating the size of the gap between two merging cuts as a function of $t$.

Using the notation introduced above, from our assumptions it  follows that 
 for a small $\varepsilon >0$, there exist polynomials $\bm A$, $\bm B$ and $\bm h$, such that
\begin{align*}
A(x;t)&=\begin{cases}
  (x-a_{-})(x-a_{+})\bm A(x;t), & t\in (T-\varepsilon,T),\\
  \bm A(x;t), & t\in [T,T+\varepsilon, T),\\
\end{cases} \\
B(x;t)&=\begin{cases}
  (x-b)  \bm B(x;t), & t\in (T-\varepsilon, T),\\ 
(x-b)  (x-\overline{b}) \bm B(x;t), & t\in [T, T+\varepsilon, T),\\
\end{cases} \\ 
h(x;t)&=\begin{cases}
 (x-\zeta) \bm h(x;t), & t\in (T-\varepsilon, T],\\
 \bm h(x;t), & t\in (T, T+\varepsilon),\\
\end{cases}
\end{align*}
where $a_{\pm} $,  $b$ and $\zeta$ are continuous functions of $t$ satisfying  
$$
a_{-}(t=T-)=a_{+}(t=T-)=b(t=T\pm)=\zeta(t=T-);
$$ 
again, we denote this common value by $b_0$. Notice that $a_{\pm} $, $b$ and $\zeta$  are real-valued for $t\in (T-\varepsilon, T]$, while  $b\in \C\setminus \R$ for $t\in (T, T+\varepsilon)$.  Polynomials $\bm A$, $\bm B$ and $\bm h$ are continuous with respect to the parameter $t\in (T-\varepsilon, T+\varepsilon)$, but represent, generally speaking,  different real-analytic functions of $t$ for $t<T$ and $t>T$. 

\begin{theorem}\label{thm:Deltaphase2}
The function, defined piecewise as
$$
\begin{cases}
\xi(t), & t<T, \\
2 \Re b(t), & t\geq T,
\end{cases}
$$
is  $C^1$ in $(T-\varepsilon, T+\varepsilon)$. Furthermore,  
\begin{equation} \label{closingDelta}	
\delta^2 (t)=   - \frac{ 2 \bm h(b_0 )}{ \bm A (b_0)\bm B(b_0)}   (T-t) \left( 1 + o(1)\right), \quad t\to T-.
\end{equation}
\end{theorem}
\begin{proof}
We denote $\bm Q(x)=\bm A (x)\bm B(x)$, omitting from the notation when possible the explicit dependence on $t$.
From \eqref{odeB} it follows that
\begin{align}
\label{localAmerging}	
\dot{a_\pm} &=  \mp \frac{2 (\zeta-a_\pm )}{(a_+-a_-) (a_\pm-b)} \frac{\bm h(a_\pm)}{\bm Q(a_\pm)}, \quad \dot b =    \frac{\zeta-b}{(b-a_-) (b-a_+)} \frac{\bm h(b)}{\bm Q(b)} ,   \quad t\in ( T-\varepsilon, T),
\\
\label{localBmerging}
\dot b &= \frac{1}{b-\overline{b}}\, \frac{\bm h(b)}{\bm Q(b)} , \quad \dot{ \overline{b}} = \frac{-1}{b-\overline{b}}\, \frac{\bm h(\overline{b} )}{\bm Q(\overline{b} )} , \quad t\in (T,T+\varepsilon).
\end{align}
Adding the first two equations in \eqref{localAmerging} and using the notation introduced in \eqref{def:zianddelta} we get
\begin{equation*}
\begin{split}
\dot \xi = &   \frac{-1}{a_+-a_-} \, \left( \frac{(\zeta- a_+)  \bm h (a_+ )}{(a_+-b) \bm Q(a_+)} - \frac{ (\zeta-a_- )  \bm h (a_- )}{(a_--b) \bm Q(a_-)} \right) , \quad t\in (T-\varepsilon, T), 
\end{split}
\end{equation*}
so that
$$
\lim_{t\to T-}\dot \xi =  -\left( \frac{(\zeta- x)  \bm h (x )}{(x-b) \bm Q(x)}  \right)' \bigg|_{x=b_0} =  \left( \frac{  \bm h (x )}{  \bm Q(x)}  \right)' \bigg|_{x=b_0} .
$$
Analogously, from \eqref{localBmerging},
$$
\lim_{t\to T+}\frac{d}{dt}\, (2 \Re b) =  \left( \frac{  \bm h(x )}{  \bm Q(x)}  \right)' \bigg|_{x=b_0} ,
$$
which proves the first assertion of the theorem. 

Furthermore, by \eqref{localAmerging},
\begin{equation} \label{asymptDelta}	
\frac{d}{dt}\,  (\delta^2)   =    \frac{(a_+-\zeta) \bm h (a_+ )}{(a_+-b) \bm Q(a_+)} + \frac{ (\zeta - a_-) \bm   h (a_- )}{(b-a_-) \bm Q(a_-)} .
\end{equation}
Condition \eqref{cond1_hk} reads as
$$
\int_{a_-}^{a_+} \frac{x-\zeta}{\sqrt{A(x;t)}} dx =0.
$$
With the change of variables $x=(a_+-a_-) y+a_-$ in the integrand we obtain that
$$
\frac{\zeta-a_-}{a_+-a_-}= \int_0^1 \frac{y \bm g_1(y;t) }{\sqrt{y(1-y)}} \,  dy \left( \int_0^1 \frac{  \bm g_1(y;t) }{\sqrt{y(1-y)}} \, dy \right)^{-1},
$$
where
$$
\bm g_1(y;t) = - \frac{\bm h ((a_+-a_-) y+a_-)}{\sqrt{\bm A((a_+-a_-) y+a_-)}} ,
$$
so that 
$$
\lim_{t\to T-} \bm g_1(y;t)  =- \frac{\bm h(b_0)}{\sqrt{\bm A(b_0)}} 
$$
uniformly for $y\in [0,1]$. Thus,
\begin{equation} \label{limitZeta1}	
\lim_{t\to T-} \frac{\zeta-a_-}{a_+-a_-} = \lim_{t\to T-} \frac{a_+-\zeta}{a_+-a_-} =\frac{1}{2}.
\end{equation}

Analogously, by \eqref{periods},
$$
\int_{a_-}^{a_+}  (x-b) \sqrt{(a_+-x)(x-a_-)}\sqrt{\bm A(x;t)} \bm B(x;t) dx =0.
$$
With the same change of variables as above, we get
$$
\frac{b-a_-}{a_+-a_-} = \int_0^1  y \sqrt{y(1-y)} \bm g_2(y;t)  \,  dy \left( \int_0^1 \sqrt{y(1-y)} \bm g_2(y;t)  \,  dy \right)^{-1},
$$
where now
$$
\bm g_2(y;t) =  \bm B((a_+-a_-) y+a_-) \sqrt{\bm A((a_+-a_-) y+a_-)}  ,
$$
so that 
$$
\lim_{t\to T-} \bm g_2(y;t)  = \bm B(b_0) \sqrt{\bm A(b_0)} 
$$
uniformly for $y\in [0,1]$. Consequently,  
\begin{equation} \label{limitZeta2}	
\lim_{t\to T-} \frac{b-a_-}{a_+-a_-} =\lim_{t\to T-} \frac{a_+-b}{a_+-a_-} = \frac{1}{2}.
\end{equation}
Combining \eqref{limitZeta1} and \eqref{limitZeta2} we see that
$$
\lim_{t\to T-} \frac{\zeta-a_-}{b-a_-} =\lim_{t\to T-} \frac{a_+-\zeta}{a_+-b}= 1.
$$
Using it in \eqref{asymptDelta} we conclude that 
$$
\lim_{t\to T-} \frac{d}{dt}\,  (\delta^2)   =  2 \frac{ \bm h (b_0 )}{ \bm Q(b_0)}<0,
$$
which establishes \eqref{closingDelta}. Furthermore, this formula shows that two cuts can come together only at zeros $b_0$ of $B$ for which $( \bm A \bm B \bm h)(b_0)<0$.   
\end{proof}

Now we turn to the simplified model problem.

Assume that  $p\geq 2$,
$$
E_\delta=\bigcup_{k=1}^p \Delta_k,
$$
where $\Delta_k= [a_{2k-1},a_{2k}]$ are pairwise-disjoint intervals, and assume that for certain $s\in \{1, \dots, p-1\}$,
$$
a_{2s} =\xi -\delta, \quad a_{2s+1}=\xi + \delta,
$$
where $\xi$ is fixed and $\delta\to 0+$. No other endpoint $a_k$ depends on $\delta$. By \eqref{greenF}, for $G(z)=G_{E_\delta}(z;\infty)$, we have 
$$
G'(z)=  \frac{h (z;\delta)}{A^{1/2} (z;\delta)}, \quad z\in \C\setminus E_\delta,
$$
with
$$
A(z;\delta)= \prod_{k=1}^{2p}\,(z-a_k) =  (z-\xi - \delta)(z-\xi+\delta) \bm A(z),
$$
and 
$$
 h(z;\delta)= \prod_{k=1}^{p-1} (z-\zeta_k)= (z-\zeta_s) \bm h (z;\delta), 
$$
where all $ \zeta_k\in [a_{2k}, a_{2k+1}]$, for $k=1, \dots, p-1$; notice that  $\deg ( \bm h )= p-2$. Zeros $\zeta_k$ do depend on $\delta$, so we will write $\zeta_k(\delta)$ when we want to make it explicit. Observe also that $\zeta_s(0)=\xi$ and $\zeta_k(0)$, $k\in \{1, \dots, p-2\}\setminus \{s\}$, are the zeros of the derivative of the complex Green's function for $E_0$.
\begin{lemma}\label{lemma:merger}
With the notation above,
\begin{equation} \label{deltasquare}	
\zeta_k(\delta)-\zeta_k(0) = \mathcal O\left(\delta^2 \right), \quad \delta \to 0+, \qquad k= 1, \dots, p-1.
\end{equation}
Moreover,
\begin{equation} \label{deltasquare2}	
\zeta_s(\delta)-\xi =  K_\xi \, \delta^2  + \mathcal O\left(\delta^3 \right), \quad \delta \to 0+,
\end{equation}
with
\begin{equation} \label{defOfK}	
K _\xi =\frac{1}{2}\, \left( \sum_{j\neq s} \frac{1}{\xi-\zeta_j(0)} - \frac{1}{2} \sum_{k\neq 2s, 2s+1} \frac{1}{\xi-a_k}  \right).
\end{equation}
\end{lemma}
\begin{proof}
Consider first the vector-valued function $\mathcal F:[0,\varepsilon)\to \R^{p-1}$, assigning each value of  $\delta$ to the corresponding vector $ (\zeta_1(\delta), \dots, \zeta_{p-1}(\delta))$, defined by equations  \eqref{cond1_hk}. It is clearly differentiable for any $\delta\in (0,\varepsilon)$. 
If we denote
$$
G_k(\delta; \zeta_1, \dots, \zeta_{p-1}) =  \int_{a_{2k}}^{a_{2k+1}} \frac{h(x;\delta)}{A^{1/2} (x;\delta) } \, dx   =0, \quad k=1, \dots, p-1,  
$$
then by the implicit function theorem,
\begin{equation} \label{linearsystem}
\begin{pmatrix}
\frac{\partial}{\partial\delta}  G_1  \\
\vdots \\
\frac{\partial }{\partial\delta}G_{p-1} 
\end{pmatrix} = -
\begin{pmatrix}
\frac{\partial G_1}{\partial \zeta_1 } & \dots & \frac{\partial G_1}{\partial \zeta_{p-1} } \\
\vdots & &\vdots\\
\frac{\partial G_{p-1}}{\partial \zeta_1 } & \dots & \frac{\partial G_{p-1}}{\partial \zeta_{p-1} } 
\end{pmatrix}
\begin{pmatrix}
\frac{d}{d\delta}  \zeta_1   \\
\vdots \\
\frac{d }{d\delta}\zeta_{p-1}  
\end{pmatrix} .
\end{equation}
Consider the $s$-th row of the matrix in the right-hand side of \eqref{linearsystem} for $\delta\to 0+$. We have
$$
G_s(\delta; \zeta_1, \dots, \zeta_{p-1}) =\int_{\xi-\delta}^{\xi+\delta}  \frac{(x-\zeta_s) \bm h(x;\delta)}{\sqrt{(x-\xi - \delta)(x-\xi+\delta) \bm A(x)}} \, dx   .
$$
Since only the numerator of the integrand depends on $\zeta_k$, it is easy to see that
$$
\lim_{\delta \to 0+} \frac{\partial G_s}{\partial \zeta_s } =- \pi\frac{\bm h(\xi;0)}{\sqrt{- \bm A(\xi)}} \neq 0 \quad \text{and} \quad \lim_{\delta \to 0+} \frac{\partial G_s}{\partial \zeta_k } =0 \quad \text{for } k\neq s.
$$
On the other hand, the minor of the matrix in the right-hand side of \eqref{linearsystem} obtained after eliminating the $s$-th row and column is clearly invertible for $\delta \to0+$: it corresponds to the system \eqref{linearsystem}  for $E_0$, and the endpoints of $E_0$ are free ends where no phase transition occurs.

Hence, expanding the matrix in \eqref{linearsystem} along its $s$-th row we conclude that it is invertible for small values of $0<\delta<\varepsilon$, and we can write
$$
\begin{pmatrix}
\frac{d}{d\delta}  \zeta_1   \\
\vdots \\
\frac{d }{d\delta}\zeta_{p-1}  
\end{pmatrix} 
 = -
\begin{pmatrix}
\frac{\partial G_1}{\partial \zeta_1 } & \dots & \frac{\partial G_1}{\partial \zeta_{p-1} } \\
\vdots & &\vdots\\
\frac{\partial G_{p-1}}{\partial \zeta_1 } & \dots & \frac{\partial G_{p-1}}{\partial \zeta_{p-1} } 
\end{pmatrix}^{-1}
\begin{pmatrix}
\frac{\partial}{\partial\delta}  G_1  \\
\vdots \\
\frac{\partial }{\partial\delta}G_{p-1} 
\end{pmatrix}.
$$ 
Observe that for $j\neq s$,
$$  
\frac{\partial G_j}{\partial \delta} = \delta \int_{a_{2j}}^{a_{2j+1}} \frac{(x-\zeta_s)\bm h(x;\delta)}{\sqrt{(x-\xi - \delta)^3(x-\xi+\delta)^3 \bm A(x)}} =\mathcal O\left( \delta \right), \quad \delta \to 0+,
$$ 
while for $j=s$, 
$$
\lim_{\delta \rightarrow 0+} \frac{\partial G_s}{\partial \delta} = \frac{\bm h(\xi;0)}{\sqrt{-\bm A (\xi)}} \int_{-1}^1 \frac{s}{\sqrt{1-s^2}} ds = 0\,.
$$
This proves \eqref{deltasquare}.

Let us be more precise about the asymptotics of $\zeta_s$. By  \eqref{cond1_hk},  
$$ 
\int_{\xi-\delta}^{\xi+\delta} \frac{(x-\zeta_s) \bm h(x;\delta)}{\sqrt{\delta^2-(x-\xi)^2} \sqrt{-\bm A(x)}} dx = 0.
$$
Thus, defining $\eta = \zeta_s - \xi$ and making in the integral above the appropriate change of variables we obtain an expression for $\eta$:
\begin{equation*}
\eta = \frac{1}{\pi} \int_{-\delta}^\delta \frac{\tau f(\tau,\delta)}{\sqrt{\delta^2-\tau^2}}\, d\tau \bigg/ \left( \frac{1}{\pi} \int_{-\delta}^\delta \frac{ f(\tau,\delta)}{\sqrt{\delta^2-\tau^2}}\, d\tau  \right),
\end{equation*}
where
\begin{equation} \label{defFaux}	
f(\tau, \delta)=\frac{\bm h(\tau + \xi;\delta)}{\sqrt{-\bm A(\tau + \xi)}}.
\end{equation}
Since $f(\tau, \delta)=f(0, \delta) + f_\tau'(0,\delta) \tau + f_{\tau \tau}''(0,\delta) \tau ^2/2 +
\varepsilon \tau ^3 $, and defining  $M_2=\max_{|\tau|\leq\delta} |f_{\tau \tau}''(\tau,\delta)|$, $ M_3=\max_{|\tau|\leq\delta} |f_{\tau \tau \tau}'''(\tau,\delta)|$, we get that
\begin{align*}
 \frac{1}{\pi} \int_{-\delta}^\delta \frac{ f(\tau,\delta)}{\sqrt{\delta^2-\tau^2}}\, d\tau & = f(0,\delta) +\varepsilon_1, \quad |\varepsilon _1|\leq \frac{M_2 \delta^2}{4}, \\
\frac{1}{\pi} \int_{-\delta}^\delta \frac{\tau f(\tau,\delta)}{\sqrt{\delta^2-\tau^2}}\, d\tau & = \frac{f_\tau'(0,\delta) \delta^2}{2} +\varepsilon_2, \quad | \varepsilon _2 |\leq \frac{M_3\delta^4}{16}, 
\end{align*}
so that
\begin{equation} \label{estimateEta}	
\eta= \frac{f_\tau'(0,\delta)}{2 f(0,\delta)}\, \delta^2 + \varepsilon_3, 
\end{equation}
where 
$$
\varepsilon_3 =  \frac{ \varepsilon_2-\varepsilon_1\delta^2  f_\tau'(0,\delta)/(2f(0,\delta))}{f(0,\delta) +\varepsilon_1}=\mathcal O\left( \delta^4\right) , \quad \delta\to 0+.
$$

We claim that in the asymptotic expression above we can replace $f_\tau'(0,\delta)/f(0,\delta)$ by $f_\tau'(0,0)/f(0,0) + \mathcal O\left( \delta^2\right)$. This is a direct consequence of \eqref{deltasquare}  and the fact that
 in $f(\tau, \delta)$, only the numerator depends on $\delta$.

This proves the lemma.
\end{proof}

As in the case of the birth of a cut, we study the asymptotics of the Robin constant $\rho(\delta)=\rho(E_\delta)$ as $\delta\to 0+$.

Observe that for $x\neq \xi$ and $ \zeta_s-\xi =K_\xi \, \delta^2  + \mathcal O\left(\delta^3 \right)$, $\delta\to 0+$,
\begin{equation} \label{expansion1}	
\frac{x-\zeta_s}{\sqrt{(x-\xi+\delta)(x-\xi -\delta)}}  =1+\left(  \frac{1}{2(x-\xi)^2}-\frac{K_\xi}{x-\xi}\right)\delta^2 + 
 \mathcal O(\delta^3),
\end{equation}
with the $\mathcal O(\delta^3)$ term uniform in $x$.

Recall that for all $\delta>0$, $\bm h(z;\delta)$ is a monic polynomial of degree $p-2$, and by \eqref{deltasquare},
\begin{equation} \label{existenceLimitH}	
\bm h(z;\delta) - \bm h(z;0) =  \mathcal O\left(\delta^2 \right), \quad \delta \to 0+
\end{equation}
uniformly in the endpoints of $E_0$. This motivates us to define 
\begin{equation} \label{limitH}	
\bm H(z) =\lim_{\delta\to 0+} \frac{\bm h(z;\delta) - \bm h(z;0)}{\delta^2};
\end{equation}
the existence of this limit will be established next. Meanwhile, it is clear that $\bm H$ is a polynomial of degree at most $p-3$ (for $p= 2$, function $f(\tau, \delta)$ in \eqref{defFaux} is a constant, so that by \eqref{estimateEta}, $\bm H\equiv 0$ in this case).

Assume that $p\geq 3$; equations  \eqref{cond1_hk} for $k\neq s$ yield:
$$
\int_{a_{2k}}^{a_{2k+1}} \left(\frac{x-\zeta_s}{\sqrt{(x-\xi+\delta)(x-\xi -\delta)}} \frac{\bm h(x;\delta)}{\sqrt{ \bm A(x)}} - \frac{\bm h(x; 0)}{\sqrt{ \bm A(x)}} \right)\, dx = 0.
$$
Dividing it through by $\delta^2$, using \eqref{deltasquare2}--\eqref{defOfK} and considering limit when $\delta\to 0+$ we obtain the following equations on $\bm H$:
\begin{equation} \label{equations_on_H}	
\int_{a_{2k}}^{a_{2k+1}}   \frac{\bm H(x )}{\sqrt{ \bm A(x)}}   \, dx =  \int_{a_{2k}}^{a_{2k+1}}  \left( \frac{K_\xi}{x-\xi}- \frac{1}{2(x-\xi)^2}\right) \frac{\bm h(x;0 )}{ \sqrt{\bm A(x)}}   \, dx, \quad k=1, \dots, p-1, \; k\neq s,
\end{equation}
where $K_\xi$ is defined in \eqref{defOfK}. This renders a  system of $p-2$  linear equations with $p-2$ unknowns (the coefficients of $\bm H$), which has a unique solution (as the consideration of the corresponding homogeneous system clearly shows), and in particular, establishes the existence of the limit in \eqref{limitH}.

From \eqref{greenF} we have for $y>a_{2p}$,
\begin{align*}
G(y, E_\delta) &=   \int_{a_{2p}}^{y} \frac{h (x;\delta)}{\sqrt{A(x; \delta)}}\, dx =\log (y) +\rho(\delta) + o(1), \quad y \to +\infty, \\
G(y, E_0) &=   \int_{a_{2p}}^{y} \frac{h (x; 0)}{\sqrt{A(x; 0)}}\, dx =\log (y) +\rho(0) + o(1), \quad y \to +\infty,
\end{align*}
so that
\begin{align*}
\rho(\delta)- \rho(0) & =  \int_{a_{2p}}^{+\infty} \left( \frac{h (x;\delta)}{\sqrt{A(x; \delta)}} - \frac{h (x; 0)}{\sqrt{A(x; 0)}} \right)\, dx \\
& =   \int_{a_{2p}}^{+\infty} \left(\frac{x-\zeta_s}{\sqrt{(x-\xi+\delta)(x-\xi -\delta)}} \frac{\bm h (x;\delta)}{\sqrt{ \bm A(x)}} - \frac{\bm h (x; 0)}{\sqrt{ \bm A(x)}} \right)\, dx,
\end{align*}
and the integral is convergent for every sufficiently small $\delta>0$. Taking into account \eqref{deltasquare}, \eqref{expansion1} and \eqref{existenceLimitH}, we conclude that $\rho(\delta)- \rho(0) =\mathcal O(\delta^2)$.  Hence, dividing the identity above through by $\delta^2$ and using the definition of $\bm H$ we get
\begin{equation}\label{asymptoticConstRobin}
\lim_{\delta\to 0+} \frac{\rho(\delta)- \rho(0) }{\delta^2}   =  \int_{a_{2p}}^{+\infty} \left(   \bm H(x) - \left( \frac{K_\xi}{x-\xi}- \frac{1}{2(x-\xi)^2}\right)\bm h (x; 0) \right)
\frac{ dx }{\sqrt{\bm A(x)}} .
\end{equation}
 
 We summarize this in the following theorem:
 \begin{theorem}\label{lem:RobinAsympt}
 Under the assumptions above, 
 $$
 \rho(\delta)- \rho(0) = K_\rho\,  \delta^2 \left( 1+\mathcal O(1)\right), \quad \delta\to 0+,
 $$
 where the constant $K_\rho$ is given by the right hand side of \eqref{asymptoticConstRobin}, and the polynomial $\bm H$ is uniquely defined by the equations \eqref{equations_on_H} for $p\geq 3$, or $\bm H\equiv 0$ for $p=2$.
 \end{theorem}

Now we go back to the phase transition when two cuts merge. Using formula \eqref{closingDelta} we see that 
$$
\rho(S_t)-\rho(S_T)= 2  K_\rho\, \frac{ \bm h (b_0 )}{ \bm A(b_0) \bm B(b_0)}    (t-T) (1+o(1)), \quad  t\to T-.
$$
Observe that this expression involves explicitly the point $b_0$ where two cuts merged; clearly, this is not the case if we consider the limit of $d\rho(S_t)/dt$ as $t\to T+$. We could conclude from here that $d\rho(S_t)/dt$ at $t=T$ has finite but, in general, different values from the left and from the right. In particular, $\rho(S_t)$ is not differentiable at $t=T$.

\begin{example} \label{example:merge2}
 
As an illustration, let us consider the case when $S_t$ has only two cuts  that merge into a single interval at $t=T$. Using the notation above, this means that $\bm A(x;t)=(x-a_1)(x-a_2)$, with $a_1< a_-<b_0<a_+ < a_2$ for $t\in (T-\varepsilon, T)$,  $\bm h(x;t)=1$, and $\bm H \equiv 0$. Thus, by \eqref{closingDelta},
$$
\delta^2(t)=  \frac{ 2}{ \bm A (b_0)\bm B(b_0)}   (t-T) \left( 1 + o(1)\right), \quad t\to T-.
$$
Since $S_t=[a_1, a_2]$ for $t\in (T,T+\varepsilon)$, we have $\rho(S_t)=\log(4/(a_2-a_1))$. Using \eqref{odeB},
 $$
\frac{d}{dt}\, \rho(S_t) =- \frac{ \dot{ a_2} - \dot{ a_1} }{a_2-a_1}=-\frac{ 2 }{(a_2-a_1)^2}  \left( \frac{1}{ B(a_2)} +   \frac{1}{B(a_1)} \right),
$$
so that 
\begin{equation} \label{onesidedExample1}
\lim_{t\to T+} \frac{d}{dt}\,  \rho(S_t)= -\frac{ 2 }{(a_2-a_1)^2}  \left( \frac{1}{ (a_2-b_0)^2 \bm B(a_2;T)} +   \frac{1}{(a_1-b_0)^2 \bm B(a_1;T)} \right)<0, 
\end{equation}
where we take the values $ a_j=a_j(t=T)$.

Consider now $t\in (T-\varepsilon, T)$;  with the notations above,  
\begin{align*}
K _\xi & =- \frac{1}{4}\, \left(     \frac{1}{\xi-a_1} +  \frac{1}{\xi-a_4}  \right), \quad   K_\xi\,  \int_{a_{2}}^{+\infty}  \frac{ dx}{(x-\xi) \sqrt{\bm A(x)}}= -\frac{1}{2}\,\frac{\theta \arccos ( \theta)}{ \bm A(\xi)\sqrt{1-\theta^2}},
\end{align*}
where
$$
\theta =  \frac{a_1 +a_2-2\xi}{a_2 - a_1 }\in (-1,1).
$$
Furthermore,
\begin{equation*} 
\frac{1}{2}\, \int_{a_2}^{+\infty} \frac{dx}{(x-\xi)^2 \sqrt{\bm A(x)}} =  \frac{2}{(\theta^2-1) (a_2-a_1 )^2}\,\left(\frac{\theta \arccos (\theta)}{\sqrt{1-\theta^2}} -1\right)\,.
\end{equation*}
Gathering these identities, setting $\xi = b_0$ and taking into account the expression of $\theta$, we get
\begin{equation} \label{onesidedExample2}
\lim_{t\rightarrow T-} \frac{d}{dt}\,  \rho(S_t) = - \frac{1}{\bm A(b_0;T)^2 \bm B(b_0;T)} <0.
\end{equation}
From \eqref{onesidedExample1} and \eqref{onesidedExample2} we can easily compute the finite jump of $\frac{d}{dt}\,  \rho(S_t)$ at $t=T$.

Further simplifications are obtained for the quartic potential, when  $\bm B\equiv 1$. Equations above boil down now to:
$$
\delta^2(t)=  \frac{ 2}{(b_0-a_1)(a_2-b_0)}   (T-t) \left( 1 + o(1)\right), \quad t\to T-
$$
(this expression matches, after due transformations, formula (2.41) in \cite{MR1986409}), and
$$
\lim_{t\to T-} \frac{d}{dt}\, \rho(S_t)= -  \frac{1}{ (a_1-b_0)^2(a_2-b_0)^2  } ,
\quad \lim_{t\to T+} \frac{d}{dt}\, \rho(S_t)=-\frac{ 2 }{(a_2-a_1)^2}  \left( \frac{1}{ (a_2-b_0)^2 } +   \frac{1}{(a_1-b_0)^2 } \right),
$$
so that
\begin{equation}\label{jump}
\lim_{t\rightarrow T+} \frac{d}{dt}\,  \rho(S_t)  - \lim_{t\rightarrow T-} \frac{d}{dt}\,  \rho(S_t) = - \left( \frac{a_1+a_2-2b_0 }{(a_2-a_1)  (a_1-b_0) (a_2-b_0) }\right)^2 \leq0.
\end{equation}

In order to  compare these formulas with those obtained in \cite{MR1986409}, we must set   $a_1  = -2=- a_2$ and $b_0 = 2c_1 = 2\cos \pi \varepsilon$, and denote $s_1 = \sin \pi \varepsilon$. It yields
\begin{equation} \label{jumpBl}
\lim_{t\rightarrow T+} \frac{d}{dt}\,  \rho(S_t) - \lim_{t\rightarrow T-} \frac{d}{dt}\,  \rho(S_t) = - \frac{c_1^2}{16 s_1^4} \leq 0.
\end{equation}
This value differs slightly from the one 
in \cite{MR1986409}, probably due to a minor error in formula (2.55) therein.

Observe finally that the jump in \eqref{jump} or \eqref{jumpBl} is strictly negative, unless $b_0$ (place of collision) coincides with the midpoint of the interval $[a_1, a_2]$. 
 For the quartic external field it takes place if and only if it is symmetric, i.e.~attains its global minimum at two distinct points. This conclusion is a straightforward consequence of the formulas in \eqref{polynomialpart}. Indeed, if $a_1=-a_2$ and $b_0=0$, then the first formula in \eqref{polynomialpart} gives us that
 $$
 \varphi'(x)=x^3-\frac{a_1^2}{2}x, \qquad  \varphi(x) = \frac{1}{4}x^2 (x^2-a_1^2).
 $$
 Since the situation is invariant by translation in $\R$, we can conclude that if $b_0$ is the midpoint of the interval $[a_1, a_2]$ then $\varphi$ has two global minima, situated at 
 $$
 \frac{a_1+a_2}{2} \pm \frac{a_2-a_1}{2\sqrt{2}},
 $$
 and a global maximum at $(a_1+a_2)/2$. The reciprocal is also immediate.

\end{example}

\subsection{Singularity of type III: higher order vanishing of $\lambda_t'$}\label{subs:III}

In this section we want to clarify  the character of the phase transition in the case of a type III singularity. To keep things simple, let us restrict our attention to the  quartic case ($m=2$ in \eqref{charactMeasure1}), but the conclusions readily follow for the general situation. Assume that for $t=T>0$  the type III singularity occurs, without loss of generality,  at the right endpoint of the support. This means that at a time $t=T$, a real and double zero $b$ of $B$ coincides with $a_2$, and by \cite[Theorem 1.3(iv)]{MR1744002}, it implies that
 for $t\in (T-\varepsilon, T+\varepsilon)$, $t\neq T$,
$$
A(x)   = (x-a_1)(x-a_2), \quad B(x)  = (x-b)(x-\overline{b}),  
$$
where $a_1<a_2$. From Remark~\ref{remark8} in Section~\ref{sec:Quartic} it follows also that $b\notin \R$, so that we take $\Im b>0$, in such a way that  both $\Im b$ and $(\Re b-a_2)$ are small with respect to $a_2-a_1$, as $t\to T$. By  \eqref{odeB}, for $0<|t-T|<\varepsilon$,
\begin{align}\label{derivativesA}
\dot{a_1} &= \frac{2}{(a_1-a_2) B(a_1) }, \quad \dot{a_2} = \frac{1}{(a_2-a_1)  B(a_2) }, \\
\label{derivativesB}
\dot{b} &= \frac{1}{(b-\overline b) A(b)}, \quad \dot{\overline b} = \overline{\dot{ b} } = \frac{2}{(\overline b-b) A(\overline b)}.
\end{align}
For the sake of brevity we denote
$$
a = a_1, \quad d= a_2-a_1, \quad \delta = \Re b -a_2, \quad v = (\Im b)^2\geq 0.
$$

\begin{theorem}\label{thm:Deltaphase3}
Let $d_0=d(t=T)$. With the notation above,
\begin{equation} \label{expansionfordeltaIII}	
\delta(t)= - \frac{ 3  }{ 2 d_0^{1/3}    } (t-T)^{1/3} +    \frac{ 21 }{ 32 d_0^{5/3}    } (t-T)^{2/3} + \mathcal O(t-T), \quad t\to T,
\end{equation}
and
\begin{equation} \label{expansionforImBIII}	
\Im b(t)  =  \frac{\sqrt{3}  }{ 2 d_0^{1/3}    } |t-T|^{1/3} +    \frac{ 21 }{ 32 d_0^{5/3}    } (t-T)^{2/3} + \mathcal O(t-T), \quad t\to T.
\end{equation}
\end{theorem}
\begin{proof}
With the notation introduced above,
\begin{equation} \label{Absquare}	
\begin{split}
A(b) & =   \delta (d+ \delta) - v + i(2 \delta  +d) (\Im b) , \quad
|A(b)|^2  = \left(\delta ^2+v\right)  \left((d+\delta)^2 +v\right),  \\
B(a_1)  & = (d+\delta)^2 + v, \quad  B(a_2)    = \delta^2 + v, \quad B(a_1) B(a_2) =|A(b)|^2.
\end{split}
\end{equation}
By \eqref{derivativesB},
\begin{align*}
\dot{(\Re b)}  & = - \frac{\Im A(b)}{2(\Im b)  |A(b)|^2} =  - \frac{d + 2 \delta     }{2   |A(b)|^2}, \quad \dot{(\Im b)}   = - \frac{\Re A(b)}{2 (\Im b)  |A(b)|^2} = \frac{v-\delta(d+\delta)}{2(\Im b)|A(b)|^2}.
\end{align*}
Combining it with \eqref{derivativesA}, we arrive at
\begin{align} \label{dotD}
\dot{d}  & = \dot{a_2}-\dot{a_1}= \frac{2}{d} \frac{  (\delta+d)^2+  \delta^2 + 2 v    }{  |A(b)|^2 }, \\
\label{dotDelta}
\dot{\delta}  &=\dot{(\Re b)} - \dot{a_2} = -  \frac{1}{d} \frac{  (d +\delta)^2+  \frac{1}{2} d (d+\delta)+v }{ |A(b)|^2 }   ,
\\
\label{dotV}
\dot{v}  &=  \frac{v-\delta(d+\delta)}{|A(b)|^2}.
\end{align}
Observe that by \eqref{dotDelta}, $\dot{\delta} <0$ for $t$ in a small neighborhood of $t=T$, so that we can use $\delta$ as a new variable. 

Dividing \eqref{dotD} and \eqref{dotV} by  \eqref{dotDelta}  we get
\begin{align} \label{derivD}
\frac{\partial d}{\partial \delta}  & = -2  \frac{  (\delta+d)^2+  \delta^2 + 2 v }{   (d +\delta)^2+  \frac{1}{2} d (d+\delta)+v}, 
\\
\label{derivV}
\frac{\partial v}{\partial \delta} &=   d \frac{\delta(d+\delta)-v}{ (d +\delta)^2+  \frac{1}{2} d (d+\delta)+v}.
\end{align}
Notice that the right hand sides are analytic at $\delta=0$ and $v=0$ as long as $d>0$. Thus, for any $d_0>0$ there exists an analytic solution $(d(\delta), v(\delta))$ of \eqref{derivD}--\eqref{derivV} satisfying the initial  conditions 
$$
d(0)=d_0, \quad v(0)=0.
$$
The series expansion of this solution at $\delta=0$ gives
\begin{align}
d  & =   d_0 -\frac{4}{3} \delta  -\frac{2}{9 d_0}   \delta ^2 -\frac{20}{27 d_0^2}\delta^3- \frac{43}{81 d_0^3}\delta^4 +\mathcal O(\delta^5) , \nonumber \\
v & =  \frac{1}{3}\delta^2 - \frac{2}{9 d_0} \delta^3 - \frac{1}{18 d_0^2}\delta^4 +\mathcal O(\delta^5). \label{asymptVbisIII}
\end{align}
Replacing it in \eqref{dotDelta} 
yields
\begin{align*} 
\frac{d}{dt}\, (\delta^3)=3 \delta^2 \dot{\delta}  & = - \frac{ 27  }{ 8 d_0    } \left(1 + \frac{ 7  }{ 6 d_0  } \delta +\mathcal O(\delta^2) \right).
\end{align*}
If we denote $q(t)=\delta^3(t)$, we will obtain from here that
$$
\left(1 - \frac{ 7  }{ 6 d_0  } q^{1/3}(t) +\mathcal O(q^{2/3}(t)) \right) dq = - \frac{ 27  }{ 8 d_0    } dt,
$$
or
$$
q(t) - \frac{  7  }{ 8 d_0  } q^{4/3}(t) +\mathcal O(q^{5/3}(t))   = - \frac{ 27  }{ 8 d_0    } (t-T).
$$
We can rewrite it as
$$
\left(\delta - \frac{ 7  }{ 24 d_0  } \delta^2 +\mathcal O(\delta^{3}) \right)^3= - \frac{ 27  }{ 8 d_0    } (t-T),
$$
and finally,
\begin{equation} \label{asymptDeltaTypeIII}	
\delta - \frac{ 7  }{ 24 d_0  } \delta^2 +\mathcal O(\delta^{3}) =  - \frac{ 3  }{ 2 d_0^{1/3}    } (t-T)^{1/3}.
\end{equation}
The analytic function of $\delta$ in the left hand side of \eqref{asymptDeltaTypeIII} is invertible, and straightforward computations yield \eqref{expansionfordeltaIII}.

On the other hand, by \eqref{asymptVbisIII},
$$
\Im b  =  \frac{|\delta|}{\sqrt{3}} \left( 1 - \frac{1}{3 d_0} \delta + \mathcal O(\delta^2) \right), \quad \delta \to 0,  
$$
so that, using \eqref{expansionfordeltaIII}, we arrive at \eqref{expansionforImBIII}.

\end{proof}

Formulas \eqref{expansionfordeltaIII} and \eqref{expansionforImBIII} show that both $\Re(b)-a_2$ and $\Im(b)$ have the asymptotic order $\mathcal O(|t-T|^{1/3})$ in a neighborhood of the critical value $t=T$ at which $b$ impacts $a_2$ and returns to the complex plane, i.e.~where the type III phase transition occurs. This is consistent with the result obtained previously in \cite[Lemma 8.1]{MR1744002}.

It is convenient to point out also a certain universal behavior this transition exhibits, expressed in the fact that the leading terms in \eqref{expansionfordeltaIII} and \eqref{expansionforImBIII} depend on the initial value $d_0$ only. In particular, we see that $b$ impacts $a_2$ from the complex plane with the asymptotic slope of $\pi/6$, both for the incidence and the reflexion angles, as it were for the actual reflection law in Optics. 

We switch now to the analysis of the Robin constant $\rho(S_t)$; in the case we are studying, $S_t=[a_1, a_2]$ in a small neighborhood of $t=T$, so that with $d=a_2-a_1$,
$$
\rho(S_t) = 2\log 2 - \log d .
$$
From \eqref{dotD} we obtain that 
$$
 \frac{d}{dt}\, \rho(S_t) = - \frac{\dot{d}}{d} = - \frac{2}{d^2} \frac{  (\delta+d)^2+  \delta^2 + 2 v    }{  |A(b)|^2 } ,\quad \text{for } 0<|t-T|<\varepsilon.
$$
Using that $\delta, v\to 0$ as $t\to T$ and expressions \eqref{Absquare} and \eqref{asymptVbisIII}, we get that
$$
\lim_{t\to T} \frac{  (\delta+d)^2+  \delta^2 + 2 v    }{  d^2 }=1, \quad \lim_{t\to T} \frac{   |A(b)|^2   }{  \delta^2 } = \frac{4}{3}d_0^2,
$$
and 
$$
\lim_{t\to T} \delta(t)^2 \frac{d}{dt}\,  \rho(S_t) = -\frac{3}{2 d_0^2 } .
$$
Taking into account \eqref{expansionfordeltaIII} we conclude that
\begin{equation} \label{asymptRhoCaseIII}	
\frac{d}{dt}\,  \rho(S_t) =- \frac{2   }{3\, d_0^{4/3}} |t-T|^{-2/3}\left( 1+\mathcal O(1)\right), \quad t\to T.
\end{equation}
In other words, in the case of a singularity of type III, we have again a third order phase transition with an infinite algebraic jump of the third derivative of the free energy at the critical time, with the exponent $-2/3$.

\subsection{Birth of new local extrema}\label{subs:IV}

We finally turn our attention  to the situation created when a pair of complex conjugate zeros $b$ and $\overline b$ of $B$ collide at $b_0\in \R\setminus S_t$ and become two new simple zeros $b_-$, $b_+$ of  $B$. It was mentioned that all $a_k$'s are analytic through $t=T$, and thus by \eqref{polynomialpart}, $B$ and $R_t$ are also analytic functions of the parameter $t$.

Using the notation introduced above, our assumptions can be written as follows:
 for a small $\varepsilon >0$, there exist polynomials $\bm A$, $\bm B$ and $\bm h$, 
  continuous with respect to the parameter $t\in (T-\varepsilon, T+\varepsilon)$, but representing, generally speaking,  different real-analytic functions of $t$ for $t<T$ and $t>T$,  such that
\begin{align*}
A(x;t)&= \bm A(x;t), \quad t\in (T-\varepsilon,T+\varepsilon), \\
B(x;t)&=\begin{cases}
 (x-b)  (x-\overline{b})  \bm B(x;t), & t\in (T-\varepsilon, T),\\ 
(x-b_-)  (x-b_+) \bm B(x;t), & t\in [T, T+\varepsilon, T),\\
\end{cases} \\ 
h(x;t)&=  \bm h(x;t), \quad  t\in (T-\varepsilon, T+\varepsilon),
\end{align*}
where $b$ and $b_{\pm} $  are continuous functions of $t$ such that  
$$
b(t=T-)=b_{-}(t=T+)= b_{+}(t=T+)=b_0\in \R\setminus S_t.
$$ 
Notice that  $b\in \C\setminus \R$ for $t\in (T-\varepsilon, T)$; without loss of generality, $\Im b>0$.  A priori, we do not assume that $b_{\pm} $  are real-valued for $t\in (T, T+\varepsilon)$, so the two possibilities are  either $b_-<b_+$ or $b_-=\overline{b_+}\in \C\setminus \R$.

We denote  $\bm g(x)=\bm h(x)/(\bm A (x)\bm B(x))$, omitting from the notation when possible the explicit dependence on $t$.

From \eqref{odeB} it follows that
\begin{align}
\label{localBmergingIV}
\dot b &= \frac{\bm g(b)}{b-\overline{b}}  , \quad \dot{ \overline{b}} = \overline{ \dot{b}} , \quad t\in ( T-\varepsilon, T), \\
\label{localAmergingIV}	
\frac{d}{dt} (b_\pm)&= \frac{\bm g(b_\pm)}{b_\pm-b_\mp},   \quad t\in (T, T+\varepsilon).
\end{align}
Subtracting/adding both equations in \eqref{localBmergingIV} we easily get that
$$
\frac{d}{dt} \left( (\Im b)^2\right) = - \Re \bm g(b), \quad \frac{d}{dt}   \Re b   = \frac{\Im \bm g(b)}{2 \Im b} .
$$
It follows in particular that  if $b_0$ is not a pole of $\bm g$, 
$$
\Im b(t)-b_0 = \sqrt{\bm g(b_0)} (t-T)^{1/2} \left(1 +o(1)\right), \quad t\to T-.
$$
Observe that these formulas show that the collision of $b$ and $\overline{b}$ on $\R\setminus S_t$ can occur only at a position $b_0$ where
$$
\bm g(b_0)\geq 0.
$$
Moreover, assuming that for $t>T$ the new zeros $b_\pm$ are complex conjugate, the same formulas apply. This yields the partial conclusion: \emph{the scenario when the complex zeros of $B$ collide at $\R\setminus S_t$ and bounce back to the complex plane can occur only when $b_0$ is either a zero or a pole of $\bm g$.} This is the situation, for instance, when $b_0$ coincides with one of the endpoints of $S_t$, and in this case we get a type III phase transition.

Thus, let us assume that
\begin{equation} \label{positivityQ}	
\bm g(b_0) > 0.
\end{equation}
This is always the case, for example, if $S_T$ is a single interval. Then, $b_-<b_+$ for $t\in (T, T+\varepsilon)$.  Denoting again $\delta=(b_+-b_-)/2$, we obtain in the same fashion that
$$
\delta(t)=\sqrt{\bm g(b_0)} (t-T)^{1/2} \left(1 +o(1)\right), \quad t\to T+.
$$

Finally, we have mentioned that a singularity of type III is a limit case of the situation analyzed here, when $b_0$ coincides with one of the $a_k$'s. Furthermore, under assumption \eqref{positivityQ}, a collision of a pair of complex conjugate zeros of $B$  at $b_0\in \R\setminus S_t$  is followed by a type I phase transition (birth of a new cut). However, these two phenomena cannot occur simultaneously: if $b_0\in (a_{2k}, a_{2k+1})$, conditions
$$
\int_{a_{2k}}^{b_0} \sqrt{R_t(y)} dy = \int_{a_{2k}}^{a_{2k+1}} \sqrt{R_t(y)} dy = 0
$$
(see \eqref{periods}) are incompatible.

\section{The quartic external field} \label{sec:Quartic}

In this section we consider in detail a particularly important case of a quartic potential, i.e.~when $m=2$ in the representation \eqref{phiPolyn}. Observe that this is the first non-trivial situation, since for $m=1$ (quadratic polynomial) all calculations are rather straightforward. According to   Section \ref{sec:polyn},  for $t>0$, $S_t$ has the form \eqref{assupport} with either $p=1$ (``one-cut regime" or ``one-cut case") or $p=2$ (``two-cut case").  Additionally to the description of all possible scenarios for the evolution of $S_t$ as $t$ travels the positive semi axis, we characterize here the quartic potentials $\varphi$ for which $S_t$ is connected \emph{for all values of} $t>0$ ($\varphi \in \frak F$ in the notation \eqref{frakF}), as well as those for which the singularity  of type III (higher order vanishing of the density of the equilibrium measure $\lambda_t$) or the birth of new local extrema occur.

Roughly speaking,   the evolution of $S_t$ can be described qualitatively as follows: for the quartic potential $\varphi$  there exists a two-sided infinite sector on the plane, centered at a global minimum of $\varphi$ and symmetric with respect to the horizontal line passing through this minimum and with the slope $0.27872057\dots$, such that  $\varphi \in \frak F$ if and only if the other critical points of $\varphi$ lie outside of this sector. Otherwise,  the positive $t$-semi axis splits into two finite subintervals and an infinite ray. The finite subinterval  containing $t=0$ (which may degenerate to a single point $t=0$) corresponds to the one-cut situation, and for the neighboring finite interval   the support $S_t$ has two connected components. Finally, the infinite ray corresponds again to the one-cut case.

Recall  (see Sections \ref{subsec:classification} and \ref{sec:phasetrans}) that  the transition from one to two cuts occurs always by saturation of the  inequality in \eqref{equilibrium1}, while the transition from two to one cut occurs by collision of some zeros of the right hand side of \eqref{charactMeasure1}.

Let us give the rigorous statements. For any quartic real polynomial $\varphi$ 
with positive leading coefficient we define the value $sl(\varphi)$ as follows: let $\zeta_0\in \R $ denote a point where $\varphi$ attains its global minimum on $\R$ (which can be unique or not), and $\zeta_1\in \C$,  any other critical point of $\varphi$ (zero of $\varphi'$). Then 
\begin{equation*} 
sl(\varphi) = \begin{cases}	\left( \dfrac{\Im \zeta_1}{\zeta_0 - \Re \zeta_1}\right)^2 , & \text{if } \Re (\zeta_1) \neq \zeta_0, \\
 +\infty, & \text{if } \Re (\zeta_1) = \zeta_0. \end{cases}
\end{equation*}
Geometrically, $sl(\varphi)$ is the square of the slope of the straight line joining $\zeta_0$ and $\zeta_1$. Notice that a real cubic polynomial has either 3 real zeros, or one real and two complex conjugate zeros, so that there is no ambiguity in the definition of $sl(\cdot)$.

Next, we define the critical slope: let $s=0.077685\dots$ denote the only positive root of the equation
\begin{equation} \label{identityfors}
32s^3-17s^2+14s-1=0.
\end{equation}	
Explicitly,
$$
s = \frac{1}{96} \left(\tau -\frac{1055}{\tau}+ 17 \right) , \quad \text{with} \quad  \tau = \sqrt[3]{5 \left(3072 \sqrt{6}-3107\right)}>0.
$$
Alternatively, 
\begin{equation} \label{eqfors1}
s=\frac{10-\gamma^2}{30},
\end{equation}	
where
\begin{equation} \label{eqfors2}
\gamma=  \frac{5}{4} \left( \left(\frac{59-24\sqrt{6}}{5}\right)^{1/3} +\left(\frac{59-24\sqrt{6}}{5}\right)^{-1/3} -1 \right)  = 2.76938\dots
\end{equation}	
is the only real solution of the equation
$$
4 \gamma^3 +15 \gamma^2 -200=0.
$$
\begin{theorem} \label{thm:characterization}
Let the  quartic potential $\varphi$ be given, and $s$ denote the critical value as described above. Then,
\begin{description}
\item[Case $sl(\varphi)> s$:] $\varphi \in \frak F$, that is, $S_t$ is a single interval for all values of $t>0$, no singularities occur. Moreover, non-real zeros of $R_t$ in \eqref{charactMeasure1}  move monotonically away from the real line if and only if $sl(\varphi)\geq 1$.

\item[Case $sl(\varphi)= s$:] $\varphi \in \frak F$, that is, $S_t$ is a single interval for all values of $t>0$, but there exists a (unique) value of $t$ for which a type III singularity occurs (and this is the unique phase transition). 

\item[Case $0<sl(\varphi)< s$:] $S_t$ evolves from one cut to two cuts, and then back to one cut, presenting once the birth of new local extrema, a singularity of type  I and a singularity of type II, in this order. No other singularities occur.

\item[Case $ sl(\varphi)= 0$:] if $\varphi$ attains its global minimum at a single point, then $S_t$ evolves from one cut to two cuts, and then back to one cut, presenting once a singularity of type I and a singularity of type II, in this order. No other singularities occur.

If $\varphi$ attains its global minimum at two different points, then $S_t$ evolves from two cuts to one cut, and only a singularity of type II is present.
\end{description}

\end{theorem}
We obviously consider  $t>0$; the value $t=0$ is not regarded as a singularity.

\begin{remark}
Observe that for $\varphi$ with more than one local extrema on $\R$ there are no type III phase transitions. 

We can also easily characterize the quartic external fields $\varphi$ for which singularity of type III occurs (i.e. such that the zeros of $\varphi'$ lie on the critical line) directly in terms of their coefficients. Indeed, if $\varphi'(x)= x^3 + d_2\,  x^2 + d_1 \, x + d_0$, then 
\begin{equation} \label{subs1}	
\varphi'\left( x-\frac{d_2}{3}\right) = x^3 + \left(d_1-\frac{d_2^2}{3} \right)x + \frac{2 d_2^3}{27}-\frac{d_1
   d_2}{3}+d_0.
\end{equation}
So, without loss of generality, we may assume that
\begin{equation} \label{phispecial}	
\varphi'(x)= x^3 +  d_1 \, x + d_0.
\end{equation}
Let $x_0$ be a real zero of $\varphi'$; according to Theorem \ref{thm:characterization}, $\lambda_t$ will have a type III singularity for a certain value of $t$ if and only if the other two zeros of $\varphi'$ are of the form $u \pm i v$, $u, v\in \R$, and the value
$$
s =  \left( \frac{v}{u-x_0}\right)^2 
$$
is a root of the polynomial \eqref{identityfors}. 
In particular, 
\begin{equation}\label{deriv2}
\varphi' (x) = (x-x_0)\left((x-u)^2+v^2\right) = (x-x_0)\left((x-u)^2+s (u-x_0)^2\right).
\end{equation}
Comparing  \eqref{phispecial} and \eqref{deriv2} we conclude that
\begin{equation}\label{system}
 2u+x_0=0,\quad   u^2+ s (x_0-u)^2+2x_0 u = d_1,\quad  x_0(d_1 -2x_0 u) = d_0.
\end{equation}
Eliminating $x_0$ and $ u$ from \eqref{system} we find that
\begin{equation}\label{surface}
  729 d_0^2 \, s^3 -  81 (4d_1^3+ 9d_0^2)s^2+9 ( -8 d_1^3+27 d_0^2)s   -4d_1^3 -27d_0^2 = 0.
\end{equation}
The resultant of  the polynomials in the left hand side of \eqref{identityfors} and \eqref{surface} is an integer multiple of
$\left(128 d_1^3+135 d_0^2\right)^3$. Since both polynomials share a common root, $s$, the resultant must vanish, and we conclude the following: 
\begin{corollary}
For an external field $\varphi'$ with derivative of the form \eqref{phispecial} the equilibrium measure $\lambda_t$ develops a type III singularity if and only if
$$
128 d_1^3+135 d_0^2 =0.
$$
\end{corollary}
Using the substitution \eqref{subs1} we can easily extend this result to the general case: \emph{for an external field $\varphi'$ such that 
$$
\varphi'(x)= x^3 + d_2\,  x^2 + d_1 \, x + d_0,
$$
the equilibrium measure $\lambda_t$ develops a type III singularity if and only if} 
$$
128 d_1^3+ 135 d_0^2- d_2\left( d_2 (2 d_2^2-9 d_1)^2 +2  (16 d_1^2 d_2 +45 d_0 d_1 -10 d_0 d_2^2 )\right)=0.
$$
For instance, direct substitution shows that $\varphi'(x) = x^3+4x^2+2x-8$ satisfies this condition. 

According to \eqref{(A,B)-representation}, in this case we will have a singularity of Type III for a finite value of $t=T$, where the density of the equilibrium measure vanishes with the exponent $5/2$. The value of $T$ can be found by the procedure described at the end of this section. 
\end{remark}

Theorem \ref{thm:characterization} is a consequence of Theorems \ref{thm:mainresult} and \ref{thm:mainresult2} below, where some additional finer results on the dynamics of the equilibrium measure as a function of $t$ are established.

Since the problem is basically invariant under homotopy, horizontal and vertical shifts in the potential $\varphi$, as well as mirror transformation $x\mapsto -x$ of the variable, in the rest of this section without loss of generality we assume that
\begin{equation}\label{newexternalfield1}
\varphi(0)=0=\min \{\varphi(x):\, x\in \R \}, \quad \text{and} \quad \varphi' (x) =   x(x-\alpha)(x-\beta) ,  
\end{equation}
with both $\alpha$ and $\beta$ in the closed right half plane. We have
\begin{equation}\label{expressionPhi}
\varphi(x)=\frac{1}{4} \, x^4 + t_3 \, x^3 + t_2\,  x^2, \quad t_3= -\frac{1}{3}    (\alpha+\beta), \quad t_2= \frac{1}{2} \alpha \beta,
\end{equation}
so that we may suppose that one of the following two generic situations takes place:
\begin{description}
\item[Case 1:] $\varphi'$ has three real roots, and $0<\alpha<\beta< 2\alpha$;
\item[Case 2:] $\varphi'$ has one real root, at $x=0$, $\alpha $ is in the first quadrant, and  $\overline{\alpha}=\beta$.
\end{description}
Case 1 is equivalent to saying that $\varphi$ has on $\R$ two local minima, 
at $x=0$ and $x=\beta$, and a maximum at $x=\alpha$, in such a way that $0=\varphi (0)<  \varphi (\beta)<\varphi(\alpha)$ (obviously, any general situation when $\varphi'$ has three real roots can be reduced to Case 1 by an affine change of variables and by adding a constant to $\varphi$). 
Case 2 means that $\varphi$ has only one local extremum on the real line.
The limit cases $\alpha=\beta$ and $\beta=2\alpha$ will be discussed in Remark~\ref{remark8} after Theorem~\ref{thm:mainresult}. 
In this way, the only case excluded from the analysis is  $\varphi(x)=x^4/4$, for which $sl(\varphi)=+\infty$ and the situation is trivial.

For $t>0$ the polynomial $R_t$ in the right hand side of \eqref{charactMeasure1} has degree 6, and the identity \eqref{RcoefGeneral} takes the form
\begin{equation}\label{Rcoef}
R_t(x) = (\varphi'(x))^2 - 2t\, x^2-d_t\, x-e_t =A(x)B^2(x),
\end{equation}
for certain constants $d_t, e_t\in \R$.
The following technical result will be useful in what follows:
\begin{lemma}\label{lemma:equations}
Assume that for $R_t$ in \eqref{Rcoef},  
\begin{equation}\label{2ndtransition}
A(z)=(z-a)(z-c), \quad B(z)=(z-b)^2,
\end{equation}
or equivalently, $R_{t}(z)\,=\,(z-a)(z-b)^4(z-c)$. Then if $\varphi$ is given by \eqref{expressionPhi}, the following identities hold: 
\begin{equation}\label{system2nd}
\begin{cases}
a+c+4b = 2(\alpha+\beta), \\
6b^2 +4b(a+c)+ac = (\alpha+\beta)^2+2\alpha \beta , \\
b(2b^2+3b(a+c)+2ac)\,=\,\alpha \beta(\alpha +\beta), \\
b^2(b^2+4b(a+c)+6ac)=\alpha^2\beta^2-2t.
\end{cases}
\end{equation}
\end{lemma}
\begin{proof}
This is a straightforward consequence of replacing \eqref{2ndtransition} in  \eqref{Rcoef} and equating the coefficients in both sides.
\end{proof}
\begin{remark}\label{remk1}
It is important to observe that the identities \eqref{system2nd} remain valid under a homothetic transformation
$$
a \mapsto q a, \quad b \mapsto q b, \quad c \mapsto q c, \quad \alpha \mapsto q \alpha, \quad \beta \mapsto q \beta, \quad t \mapsto q^4 t,
$$
for $q>0$. In other words, a linear scaling in space yields a quartic scaling in time (or temperature).
\end{remark}

Let us consider first Case 1 (with strict inequalities). The theorem below is the quantitative description of the following evolution: for small values of  temperature $t$ a single cut is born in a neighborhood of the origin. The other two (double) zeros of $R_t$ are real and close to $x=\alpha$ and $x=\beta$, moving in opposite directions. At the first critical temperature $t=T_1$ a bifurcation of type I occurs: the rightmost double zero splits into two simple real zeros, giving birth to a second cut in the spectrum. This configuration is preserved until both cuts merge at a quartic point at a temperature $t=T_2$ (phase transition of type II). After that, two complex conjugate double roots of $R_t$ drift away to infinity in the complex plane, so that for the remaining situation we are back in the one-cut case.

For the sake of convenience, here we use all introduced notations interchangeably, 
$$
\dot{a} =\frac{d}{dt}a =\partial_0 a.
$$
\begin{theorem}\label{thm:mainresult}
Let $\varphi$ be an external field given by \eqref{newexternalfield1} with $0<\alpha<\beta<2\alpha$. Then there exist two critical values $0<T_1<T_2$ such that:  \begin{itemize}
\item \textbf{(Phase 1):}
for $0<t<T_1$, $A$ and $B$ in \eqref{Rcoef} have the form
\begin{equation} \label{phase1}
A(x)=(x-a_1)(x-a_2), \quad B(x)=(x-b_1)(x-b_2),
\end{equation}	
with  $a_1<0<a_2<b_1 <b_2$, and $S_t=[a_1, a_2]$ (one-cut case).

Parameters $a_k$, $b_k$ are functions of $t$ and satisfy the system of differential equations
\begin{equation}
\label{ode1}
\dot{ a_k }=  \frac{2 }{A'(a_k) B(a_k)}, \quad \dot{b_k}    =  \frac{1 }{A(b_k) B'(b_k)},  \quad k=1,   2 ,
\end{equation}
with the initial values
$$
a_1(t=0)=a_2(t=0)=0, \quad b_1(t=0)=\alpha, \quad b_2(t=0)=\beta.
$$
For $0<t<T_1$, 
\begin{equation} \label{signs1}
\dot{a_1} <0, \quad \dot{ a_2} >0, \quad \dot{ b_1} <0, \quad \dot{ b_2} >0,
\end{equation}	
and function
\begin{equation} \label{def:F}
F(t)=\int_{a_2(t)}^{b_2(t)} \sqrt{A(s)} B(s)\,ds 
\end{equation}	
is monotonically decreasing and positive in $(0,T_1)$, with
$$
F(0)=\varphi(\beta)=\frac{(2\alpha-\beta)\beta^3}{12}>0 \quad \text{and} \quad F(T_1)=0.
$$

\item \textbf{(1st Transition, phase transition of type I):} for $t=T_1$ such that $F(T_1)=0$ we have
$S^{T_1}\setminus S_{T_1}=\{b^{(1)}\}$, with $b^{(1)}:=b_2(t=T_1)$, which is a singular point of type I (``birth of a cut'').

\item  \textbf{(Phase 2):}
for $T_1<t<T_2$, $A$ and $B$ in \eqref{Rcoef} have the form
\begin{equation} \label{formulaPhase2}
A(x)=(x-a_1)(x-a_2)(x-a_3)(x-a_4), \quad B(x)=x-b_1 ,
\end{equation}	
with  $a_1<0<a_2<b_1 <a_3<b^{(1)}<a_4$, and $S_t=[a_1, a_2]\cup [a_3,a_4]$ (two-cut case).

Moreover, these values satisfy the system of differential equations
\begin{equation}
\label{ode2}
\dot{a_k}   =  \frac{2 (a_k-\zeta) }{A'(a_k) B(a_k)},   \quad k=1,   \dots, 4 , \quad \text{and} \quad \dot{b_1}   =  \frac{b_1-\zeta }{A(b_1) }.
\end{equation}
Parameters $a_1(t)$ and $a_2(t)$ are continuous at $t=T_1$, while
\begin{equation} \label{initial1}
a_3(t=T_1)=a_4(t=T_1)= \zeta(t=T_1)=b^{(1)}.
\end{equation}	
The value of $\zeta\in (a_2,a_3)$ is determined by
 $$
 \int_{a_{2}}^{a_{3}} \frac{x-\zeta}{\sqrt{A (x)}} \, dx =0.
 $$
In particular,
$$
\dot{a_1}, \dot{a_3}<0, \quad \text{and} \quad \dot{a_2}, \dot{a_4}>0.
$$
The critical temperature $T_2$ is determined by the collision condition $a_2(t=T_2)=a_3(t=T_2)$.

\item \textbf{(2nd Transition, phase transition of type II):} for $t=T_2$,  $A$ and $B$ in \eqref{Rcoef} have the form
$$
A(x)=(x-a_1)(x-a_4), \quad B(x)=(x-b^{(2)})^2,
$$
where
$$
b^{(2)}=\lim_{t\to T_2-}a_2=\lim_{t\to T_2-}a_3=\lim_{t\to T_2-}\zeta=\lim_{t\to T_2-}b_1,
$$
so that  $a_1<0< b^{(2)}< b^{(1)}<a_4$,  and $S_t=[a_1, a_4]$. Points $a_1$, $b^{(2)} $ and $a_4$ can be found from the first three equations in \eqref{system2nd} (with $a=a_1$, $b=b^{(2)}$ and $c=a_4$), while the value of $T_2$ is obtained from the fourth equation \eqref{system2nd}.

Since $b^{(2)} \in S_{T_2}$,  $x=b^{(2)}$ is a singular point of type II, and $R_{T_2}$ has a root of order 4 at $z=b^{(2)}$.

\item  \textbf{(Phase 3):}
for $t>T_2$, $A$ and $B$ in \eqref{Rcoef} have the form
\begin{equation} \label{phase3}
A(x)=(x-a_1)(x-a_4), \quad B(x)=(x-b_1)(x-\overline{b_1}),
\end{equation}	
with $a_1<0< b^{(2)} < b^{(1)}<a_4$, $\Im b_1>0$, and $S_t=[a_1, a_4]$ (one-cut case).

Moreover, these values satisfy the system of differential equations of the form \eqref{ode1} (setting now $b_2=\overline{b_1}$); in particular, $\Im b_1$ grows monotonically with $t$, and 
\begin{equation} \label{center_mass}	
\lim_{t\to +\infty} \Im b_1=+\infty, \quad  \lim_{t\to +\infty} \Re b_1=\lim_{t\to +\infty} \frac{a_1+a_4}{2}=\frac{\alpha+\beta}{3} .
\end{equation}

\end{itemize}
\end{theorem}
\begin{remark} \label{remark8}
Let us consider the limiting situations in the Case 1. 

If $\alpha=\beta>0$, $\varphi'$ has a double zero at $\alpha$. 
Setting $\beta = \alpha$ in the system of nonlinear equations \eqref{system2nd} and eliminating the variables $a$ and $c$ from the first two  equations yields
$$5b^3-10\alpha b^2+6\alpha^2 b-\alpha^3 = 0\,,$$
which has only  $3$ real solutions for $b$: obviously, $b=\alpha$ (for which  $a=c=0$, and by the fourth equation in \eqref{system2nd}, $t=0$) and $b = \frac{\alpha}{2}\,\left(1\pm\,\frac{\sqrt{5}}{5}\right)$. Thus, by Lemma~\ref{lemma:equations} there are at most $3$  values of $t$ for which $B$ can have a double zero. In particular, this implies that there exists $\varepsilon>0$ such that for $0<t<\varepsilon$ equation \eqref{phase1} holds, so that Theorem~\ref{thm:mainresult} is valid in this case also.

On the other hand, if $\beta=2\alpha$, we have $ \varphi (\beta)=0$, which is essentially equivalent to an even external field. In this case $F(0)=0$, and the evolution is as described in Theorem~\ref{thm:mainresult}, except for $T_1=0$ (Phase 1 is missing).
\end{remark}

\begin{proof}
Using \eqref{limitT0} and representation \eqref{phirepre} we conclude that there exists $\varepsilon>0$ such that for $0<t<\varepsilon$ we are in the one-cut case, and formulas \eqref{phase1} hold, with both $b_1$ and $b_2$ close to $\alpha$ and $\beta$, respectively. In this situation, $h=1$, so that
  \eqref{ode1} is a particularization of \eqref{odeB}, and the inequalities \eqref{signs1} are just straightforward consequences of \eqref{ode1}.
Observe also that  \eqref{periods}  implies that
$$ \int_{a_2}^{b_2}  \sqrt{A(s)}B(s)\,ds > 0.$$

Furthermore, for any finite initial positions $a_k(t=0):=a_k^0$, $b_k(t=0):=b_k^0$, with  $a_1^0< a_2^0<b_1^0 <b_2^0$, the solution of the system of differential equations \eqref{ode1}
 exhibits collision (of $a_2$ and $b_1$) in finite time. Indeed, by \eqref{signs1}, for $t>0$ (and before the collision),
 $$
a_1<a_1^0<a_2^0<a_2<b_1<b_1^0<b_2^0<b_2.
 $$
Since $ \dot{b_2}\leq (b_2-b_1^0)^{-3}$, solving the corresponding ODE we conclude that
$$
b_2\leq G_1(t):= b_1^0 + \left( 4 t + (b_2^0-b_1^0)^4\right)^{1/4}.
$$
Analogously,
$$
a_1\geq G_2(t):=a_2^0- \left( 8 t + (a_2^0-a_1^0)^4\right)^{1/4}.
$$
Replacing these bounds in \eqref{ode1} we get
\begin{align*}
\dot{a_2}> & \frac{2}{(b_1^0-G_2(t))(G_1(t)-a_2^0)(b_1-a_2)}, \\
\dot{b_1}< & \frac{1}{(b_1^0-G_2(t))(G_1(t)-a_2^0)(a_2-b_1)},
\end{align*}
or
$$
(b_1-a_2) (\dot{b_1}-\dot{a_2})=\frac{1}{2}\, \frac{d}{dt}\, (b_1-a_2)^2<-\frac{3}{(b_1^0-G_2(t))(G_1(t)-a_2^0)}.
$$
Thus, a collision will occur by time $t=T$ if
$$
\int_0^T \frac{3}{(b_1^0-G_2(t))(G_1(t)-a_2^0)}\, dt > \frac{(b_1^0-a_2^0)^2}{6}.
$$
Since $G_1(t)\sim t^{1/4}$,  $G_2(t)\sim t^{1/4}$ as  $ t\to \infty$, the integral in the left-hand side diverges as $T\to + \infty$, which proves that there will always be a collision in a finite time $T$. Observe that function $F$ in \eqref{def:F} is well-defined in the whole interval $(0, T)$, and that the integrand in \eqref{def:F} is, up to a constant, the analytic continuation of the density of $\lambda_t$, see \eqref{(A,B)-representation}. Using \eqref{derivativesMeasure} we conclude that
\begin{equation*}
F'(t) =-  \int_{a_2(t)}^{b_2(t)} \frac{ds}{\sqrt{A(s)} }< 0\,.
\end{equation*}
However, at the collision time $T$,
$$
F(T)=\int_{a_2}^{b_2} \sqrt{(x-a_1)(x-a_2)} (x-b_1)(x-b_2)\,dx<0,
$$
which shows that there is a unique time $T_1<T$ for which $F(T_1)=0$.

From the positivity of the measure $\lambda_t$ and expression \eqref{R-representation}  it is easy to conclude that all roots of $R_t$, which are not endpoints of the support, need to be double. Hence, for $t> T_1$,  the rightmost double root $b_2$ splits into a pair of simple real roots $a_3$ and $a_4$, giving rise to formula  \eqref{formulaPhase2}. We apply again Theorem \ref{thm:Dynamical System-t} with $h(x)=x-\zeta$, which yields \eqref{ode2}.

Observe that it is not straightforward to deduce the sign of $\dot{b_1}$ from
$$
\dot{b_1}   =  \frac{b_1-\zeta }{A(b_1) }.
$$
Taking into account the initial values \eqref{initial1} we see that  for a small $\varepsilon>0$ and $T_1<t<T_1+\varepsilon$, $\zeta$ is close to $a_3$, so that $b_1<\zeta$. This implies that immediately after the birth of a cut ($t=T_1$), $\dot{b_1}<0$, so that point $b_1$ still moves to the left ``by inertia''. In that range of time, $a_2$ and $a_3$ (and in consequence, also $b_1$ and $\zeta$) are in the collision course, and collision occurs in finite time, when all these four points merge \emph{simultaneously}. This critical time $T_2>T_1$ can be characterized by the appearance of a quadruple root $b^{(2)}$ of $R_{T_2}$ inside $S_{T_2}$, so that the system \eqref{system2nd} is valid.

Taking into account the monotonicity of $S_t$, we see that for $t>T_2$ the quadruple root $b^{(2)}$ of $R_{t}$ splits into two complex-conjugate roots
 $b_1$ and $\overline{b_1}$, and formulas \eqref{phase3} hold. 
 
Adding the equations \eqref{odeB} for $b_k$'s ($b_2=\overline{b_1}$) we obtain
\begin{equation}
\label{ode1bis}
\frac{d}{dt} \, (\Im b_1)^2   = -\Re \frac{1 }{A(b_1)}=-\frac{\Re A(b_1)}{|A(b_1)|^2}.
\end{equation}
Observe that $\Re A(z)=0 $ is an equation of an ``East-West opening'' rectangular (or equilateral) hyperbola $\Gamma=\Gamma_t$ with its vertices  at  $a_1$ and $a_4$, and $\{z\in \C:\, \Re A(z)<0 \}$ corresponds to the connected component $U_t$ of its complement in $\C$ containing the segment joining $a_1$ and $a_4$.

\begin{figure}[htb]
\centering \begin{overpic}[scale=0.43]%
{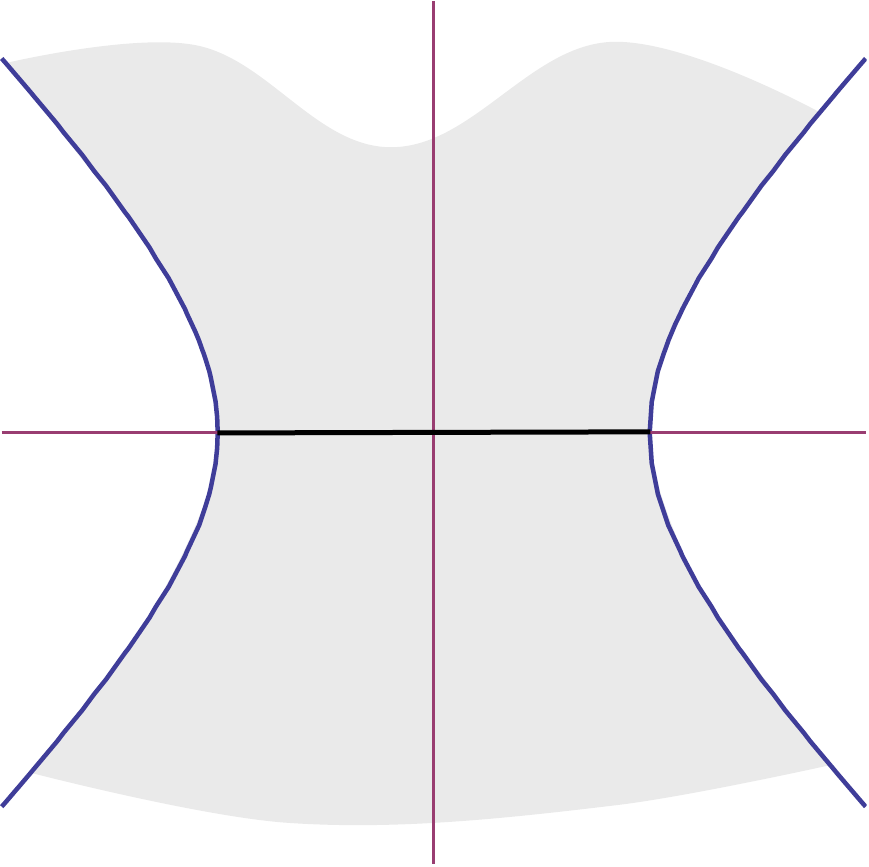}%
 \put(17,46){$a_1 $}
 \put(77,46){$a_4 $}
 \put(3,79){$\Gamma_t $} 
  \put(33,69){$U_t $}
\end{overpic}
\caption{Hyperbola $\Gamma_t$ and the domain $U_t$ (shaded).}
\label{fig:hyperbola}
\end{figure}
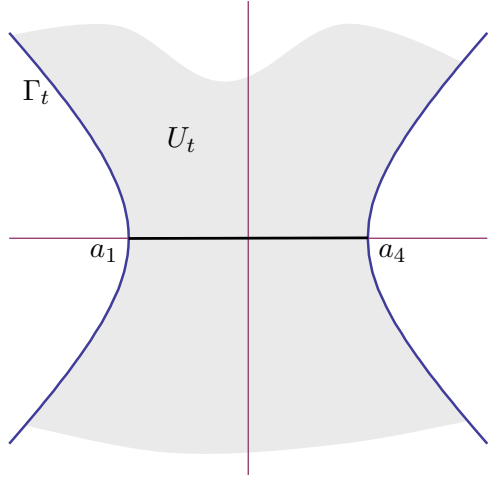

Hence, we conclude from \eqref{ode1bis} that
\begin{equation} \label{signb1}
\frac{d}{dt} \, \Im (b_1)>0 \quad \Leftrightarrow \quad b_1\in U_t.
\end{equation}	
Moreover,  the monotonicity of the support $S_t$ (or equivalently, the fact that $\dot{a_1}<0<\dot{a_4}$) implies that
$$
\tau_1<\tau_2 \quad \Rightarrow \quad U_{\tau_1}\subset U_{\tau_2}.
$$

Clearly, for $t>T_2$,  $b_1\in U_t$, so that in this range of $t$, $\Im (b_1)$ grows  monotonically, and we have a one-cut case for all $t>T_2$.

Finally, since by assumptions \eqref{newexternalfield1}, $\varphi(x)>0$ for all $x\in \R\setminus \{0\}$, using the representation \eqref{phirepre} we conclude that
$$
\lim_{t\to + \infty } (-a_1)=\lim_{t\to + \infty } a_4=+\infty, \quad \bigcup_{t>0} S_t =\R.
$$

Regarding the second limit in \eqref{center_mass}, observe that from equations \eqref{odeB}, taking $b_2=\overline{b_1}$, we get
\begin{equation*}
\label{odere}
\frac{d}{dt}  \left(\frac{a_1+a_4}{2}\right) = -2\,\frac{d}{dt} \, (\Re b_1)   =\frac{2\Re b_1-(a_1+a_4)}{|A(b_1)|^2}.
\end{equation*}
An immediate consequence of these identities is that the centers of masses of the zeros of $A$ and of the zeros of $B$ are always in a collision course.

A comparison of the coefficients at $x^3, x^4$ and $x^5$ in both sides of \eqref{Rcoef} yields the system
\begin{equation}\label{systemph3}
\begin{split}
4\Re b_1 +a_1+a_4 & = 2(\alpha+\beta), \\
6(\Re b_1)^2+2(\Im b_1)^2+4(a_1+a_4)\Re b_1+a_1a_4 & = (\alpha+\beta)^2+2\alpha\beta, \\
4|b_1|^2\Re b_1+6\left((\Re b_1)^2+2(\Im b_1)^2\right)(a_1+a_4)-4a_1a_4\Re b_1& = - 2\alpha\beta(\alpha+\beta)  .
\end{split}
\end{equation}
An assumption that $a_1+a_4=2 \Re b_1$ (collision) in the last two equations in \eqref{systemph3} implies that $2\alpha^2+2\beta^2-5\alpha\beta=0$, which is possible only if $\beta=2\alpha$. This is the  symmetric case  not considered here.

Hence, we conclude that in our situation there exist the limits
$$
\lim_{t\to + \infty } \frac{a_1+a_4}{2}= \lim_{t\to + \infty } \Re b_1.
$$
By using this in the first identity in \eqref{systemph3} we conclude the proof of \eqref{center_mass}.
\end{proof}

Next, we turn to Case 2. 

\begin{theorem}\label{thm:mainresult2}
Consider the external field given by \eqref{expressionPhi} with $\beta=\overline \alpha$, and
with $ \alpha $ in the first quadrant. Then:  
\begin{enumerate}	
\item[(a)]  if 
$$
\arg (\alpha)\geq \arg (1+ i \sqrt{s}) ,
$$
where $\sqrt{s}=0.27872057\dots$, and $s$ is the only positive root of the equation \eqref{identityfors}, then for all $t>0$, polynomials $A$ and $B$ in \eqref{Rcoef} have the form 
\begin{equation} \label{phase1case2}
A(x)=(x-a_1)(x-a_2), \quad B(x)=(x-b_1)(x-\overline{b_1}), \quad \Im b_1>0,
\end{equation}	
with    $a_1<0<a_2$, and $S_t$ consists of a single interval $[a_1, a_2]$ ($\varphi \in \frak F$, one-cut case). 

If $\arg (\alpha)\geq \pi/4$, then functions $-a_1$, $a_2$ and $\Im  b_1 $ grow monotonically with $t$ from their initial positions $0$, $0$ and $\Im \alpha$, respectively, to $+\infty$. 
Otherwise (i.e., if $\arg (1+ i \sqrt{s}) \leq  \arg (\alpha)<  \pi/4$), there exists a critical value $T_0>0$ such that $\Im b_1 $ is monotonically decreasing for $0<t<T_0$,  monotonically increasing for $t>T_0$, and $\Im(b_1(t=T_0))\geq 0$. 
 
\item[(b)] if 
$$
0<\arg (\alpha)< \arg (1+ i \sqrt{s}) ,
$$
 then there exist three critical values $0<T_0<T_1<T_2$ such that:
\begin{itemize}
\item for $0<t<T_0$, polynomials $A$ and $B$ in \eqref{Rcoef} have the form \eqref{phase1case2} with $a_1<0<a_2$; $\Im b_1 $ is  monotonically decreasing for $0<t<T_0$. At the critical value $T_0$,  $  b^{(0)}:=b_1(t=T_0) \in \R\setminus (a_1, a_2) $, which is a birth of new local extrema; without loss of generality, we assume $ b^{(0)}>a_2(t=T_0)$. During this phase,  $S_t$ consists of a single interval $[a_1, a_2]$.
\item for $T_0<t<T_1$, polynomials $A$ and $B$ in \eqref{Rcoef} have the form \eqref{phase1}, with $a_1<0<a_2<b_1 <b_2$, and $S_t=[a_1, a_2]$ (still one-cut case).

At $t=T_1$, function $F$ defined in \eqref{def:F} satisfies $F(T_1)=0$, and $b_2(t=T_1)=b^{(1)}$ is a singular point of type I (``birth of a cut'').

\item for $T_1<t<T_2$, Phase 2 described in Theorem \ref{thm:mainresult}, takes place. Second transition, also described in Theorem  \ref{thm:mainresult}, occurs at $t=T_2$ by the same mechanism (collision of four zeros of $R_t$). This situation corresponds to the two-cut case.

\item finally, for $t>T_2$ we have Phase 3 as described in Theorem \ref{thm:mainresult}. We are back in the one-cut case.
\end{itemize}
\end{enumerate}
\end{theorem}
\begin{proof}
As in the proof of  Theorem \ref{thm:mainresult}, using \eqref{limitT0} and representation \eqref{phirepre} we conclude that there exists $\varepsilon>0$ such that for $0<t<\varepsilon$ we are in the one-cut case, and formulas \eqref{phase3} hold. 
Using the equivalence \eqref{signb1} we conclude, in particular, that if $\alpha\in U_0$, that is, if $|\Re \alpha|\leq \Im \alpha$, then $\Im (b_1)$ grows monotonically with $t$, varying from $\Im (\alpha)$ to $+\infty$ as $t$ travels $(0,+\infty)$. Clearly, in this situation we have a one-cut case for all $t>0$, which establishes the first part of (a).

As it was discussed before, for a fixed $A$, the solutions $b_1$ and $\overline{b_1}$ of \eqref{ode1bis} exhibit collision in a finite time. The second part of Case (a) corresponds to the situation when for a $T_0>0$ before the collision, $b_1\in \Gamma_{T_0}$ (see Figure~\ref{fig:hyperbola}).

Finally, consider the case when $b_1$ hits the real line in a time $T_0$. Clearly, for $t<T_0$, $b_1\notin U_t$, so that $b^{(0)}=b_1(t=T_0)\in \R\setminus (a_1^0, a_2^0)$, where $a_k^0:=a_k(t=T_0)$. Assume first that $b^{(0)} \neq a_k^0$, $k=1,2$. At $t=T_0$ we have that $A$ and $B$ in \eqref{Rcoef} have the form
\begin{equation*} 
A(x)=(x-a_1)(x-a_2), \quad B(x)=(x-b_1)^2,
\end{equation*}	
with (without loss of generality)  $a_1<0<a_2<b_1$. As $t$ becomes greater than $T_0$, the only possibility is that $b_1$ splits into two real zeros, so we are left in the situation of Phase 1 of Case 1 (see Theorem \ref{thm:mainresult}). From this point the evolution of $A$, $B$ and $S_t$ follows exactly the Case 1.

The boundary between  (a) and  (b) is precisely when
$$
b_1(t=T_0)=a_2(t=T_0),
$$
that is, when $R_{T_0}$ has a zero at $b_1$ of order 5; using the expression in \eqref{Rcoef} we get that for $t=T_0$,
\begin{equation*} 
R_t(x) = x^2(x-\alpha)^2(x-\overline{\alpha})^2 - 2t\, x^2-d_t\, x-e_t =(x-a_1) (x-b_1)^5.
\end{equation*}
We can use the identities \eqref{system2nd} from Lemma \ref{lemma:equations} with $a=a_1$, $b=c=b_1$ and $\beta=\overline{\alpha}$. Moreover, by the homogeneity of these identities (see Remark \ref{remk1}) we may assume that  $\alpha=1+i \sqrt{s}$, with $ s>0$: the boundary between case (a) and case (b) will be given by the ray $u \alpha$, with $u>0$.

Thus, from  \eqref{system2nd} we obtain
\begin{equation*}
\begin{split}
a_1+5b_1  & = 4, \\
5b_1(a_1+2b_1)  & = 2 (3+   s ) , \\
5b_1^2(a_1+b_1) & = 2(1+s), \\
2t=  (1+s )^2 & -5b_1^3(2a_1+b_1).
\end{split}
\end{equation*}
These equations must have a solution with $a_1<0<b_1$ for $t>0$. Solving the first two for $a_1$ and $b_1$ and replacing the solutions in the third equation (preferably, using a symbolic algebra software) we obtain
$$
4 r^3 +15r^2 -200=0, \quad r=\sqrt{10-30s^2},
$$
which is equivalent to \eqref{identityfors}. Direct calculations show that the discriminant of the cubic polynomial above is negative, so it  has a unique (positive) root $s= 0.077685\dots$, which satisfies identities \eqref{eqfors1}--\eqref{eqfors2}, and two complex conjugate roots. 
Finally, a substitution of the obtained values for $s, a_1, b_1$ in the last equation renders $t = 0.0339206\dots>0$.
\end{proof}

We conclude this section with a number of remarks that we consider relevant.

The results from Theorem \ref{thm:odeBis} allow us to explore the dependence of the support $S_t$ (and of the corresponding equilibrium measure $\lambda_t$) from the rest of the parameters of the problem, in our setting, from the position of the critical points $\alpha$ and $\beta$. For instance, when $S_t$ consists of a single interval (``one-cut case'') differential relations \eqref{newDiffRelat2_1}--\eqref{newDiffRelat2_3} apply. Observe e.g.~that in the case when the center of masses of $\alpha$ and $\beta$ increases (hence, $t_3$ in \eqref{expressionPhi} decreases), both endpoints of $S_t$  move to the right.

In a certain sense, the considerations above show that the most general situation corresponds to Case 2 with $\alpha$ lying in the complex plane but below the critical ray emanating from the origin and passing through $1+i \sqrt{s}$ (or, in the terminology introduced at the beginning of this section, when $0<sl(\varphi)< s $). This and only this situation exhibits all possible transitions occurring in the quartic case. 

Observe that the external field \eqref{expressionPhi} with $\beta=\overline \alpha$, and
with $ \alpha $ in the first quadrant is convex if and only if $\arg (\alpha) \geq \pi/6> \arg (1+ i \sqrt{s})$,  where $s$ is the critical value given by \eqref{eqfors1}--\eqref{eqfors2}. Hence, the persistence of the one-cut case for all $t>0$ under assumptions of convexity (see e.g.~\cite{Saff:97}) is a consequence of the statement above.

\begin{example}\label{rem:bleher}
Let us return to the analysis of  Bleher and Eynard in  \cite{MR1986409}, where they considered the external field of the form
\begin{equation} \label{Eynard}	
\varphi(x;c_1) =  \frac{x^4}{4}-\frac{4c_1x^3}{3}+(2c_1^2-1)x^2+8c_1x , \quad \text{with } c_1\in (-1,1).
\end{equation}
Notice that $c_1=0$ gives a particular case of the so-called planar diagram model \cite{MR0471676}.

This family always exhibits a second transition (see Theorem \ref{thm:mainresult}) or a singular point of type II for $t=T_2=1+4c_1^2$, in such a way that $-2=a_1(t=T_2)<b^{(2)}<a_2(t=T_2)=2$. According to our analysis, for $c_1\neq 0$ there is also a first transition (birth of a cut) at $T_1<T_2$, not described in \cite{MR1986409}, but predicted in \cite{MR2453314} and found numerically in \cite{MR2629605}.

Expression \eqref{Eynard} is an alternative parametrization of the quartic external field; the range  $c_1\in (-1,1)$ covers all the cases when a two-cut $S_t$ arises, described in Theorem \ref{thm:characterization}. However, with the purpose of characterizing the existence of phase transitions we find the description given in our   Theorem \ref{thm:characterization} more transparent. 
 
The analysis in \cite{MR1986409} was extended in \cite{Marchal:2011fk} to the case when $\varphi$ is a polynomial of even degree (greater than $4$) such that in the critical time $t=T_2$   the density of the equilibrium measure has a zero of higher order $2m$, $m\geq 1$.
\end{example}

We finish this section observing that using the system of equations \eqref{system2nd}  we can find algebraic equations for the quadruple critical points $b^{(0)}$ and $b^{(2)}$ and for the corresponding endpoints of the support $S^{T_0}$ and $S^{T_2}$, respectively (so that these values  are algebraic functions of $\alpha$ and $\beta$, and hence, of the coefficients of the external field). Indeed, solving the first two equations in \eqref{system2nd} for $a$ and $c$, we get
\begin{equation} \label{a&c}	
\begin{split}
a &= \alpha+\beta-2b  -\sqrt{4 b (\alpha
   +\beta )-2
   \alpha  \beta -6 b^2} , \\
c &= \alpha+\beta-2b  +\sqrt{4 b (\alpha
   +\beta )-2
   \alpha  \beta -6 b^2} ,
\end{split}
\end{equation}
 and replacing it in the third one yields
\begin{equation}\label{pb}
p(b)\,=\,10 b^3-10 (\alpha+\beta)b^2+2 ((\alpha+\beta)^2+2\alpha\beta) b - \alpha\beta(\alpha+\beta)=0.
\end{equation}
If $0<\alpha<\beta\leq 2\alpha$, this is equivalent to finding real roots of the polynomial 
\begin{equation}\label{positiveroot}
p^*(x)=10 x^3-10 (1+u)x^2+2 ((1+u)^2+2  u) x - u(1+u)
\end{equation}
for the values of the parameter $u\in (1,2]$. Its discriminant (easily found with the help of a computer algebra system),
$$
20 \left(4 u^6-12 u^5+29 u^4-38 u^3+29
   u^2-12 u+4\right),
$$
is positive for $u\in (1,2]$, and the roots of $p^*$ for $u=1$ are  $(5\pm\sqrt{5})/10$ and $1$. Hence, this polynomial has three positive simple roots in the indicated range of $u$. However, only the middle one guarantees that 
$$
 2 x (1    +u )- u -3 x^2\geq 0,
$$
which, according to \eqref{a&c}, is a necessary condition for having real solutions  $a$ and $c$ in \eqref{system2nd}. As observed, for $u=1$ this root is $(5+\sqrt{5})/10\approx 0.7236$, while for $u=2$ it is $1$.

Hence, in the case when $0<\alpha<\beta\leq 2\alpha$, the recipe for finding the (unique) critical point is as follows: with $u=\beta/\alpha\in (1,2]$, find the  root $x^{*}$ of \eqref{positiveroot} lying in 
$$
\left( \frac{5+\sqrt{5}}{10}, 1\right],
$$
and take  $b^{(2)}= \alpha \, x^{*}$. Then the endpoints of the support are obtained by replacing  $b=b^{(2)}$ in \eqref{a&c}, and taking $a_1(t=T_2):=a \leq b^{(2)} \leq a_2(t=T_2):=c$; the value of $T_2$ is finally computed from the fourth equation in \eqref{system2nd}. In this case,  as we have seen, $x=b^{(2)}$ is a singular point of type II (zero of the density of the equilibrium measure).

Let us consider the case when $\alpha $ is in the first quadrant, and  $\overline{\alpha}=\beta$. Denoting $\zeta:=\Im(\alpha)/\Re(\alpha)> 0$ and dividing   \eqref{pb} by $(\Re(\alpha))^3$, we arrive at the polynomial equation
\begin{equation}\label{pb2}
p^*(x)=5 x^3-10 x^2+2 (3  +\zeta^2) x - (1+ \zeta^2)=0,
\end{equation}
whose discriminant, as natural, is given by the left hand side of \eqref{identityfors} with $s$ replaced by $s^2$. Thus, for $\zeta>\sqrt{s}$, with $s$ defined in \eqref{eqfors1}--\eqref{eqfors2}, polynomial $p^*$ has only one real root, which yields non-real values of $a$ and $c$ in \eqref{a&c}. In this case, no quadruple critical point appears.

If on the contrary, $0<\zeta<\sqrt{s}$, then $p^*$ has three positive roots; only two of them give real values of $a$ and $c$ in \eqref{a&c}.

Summarizing, in the case when $0<\zeta:=\Im(\alpha)/\Re(\alpha) < \sqrt{s}$, the recipe for finding two critical points is as follows:  find the real roots $x^{*}$ of \eqref{pb2}, take  $b= \Re(\alpha) \, x^{*}$ and replace it in \eqref{a&c}. If $a<c<b$, then $a=a(t=T_0)$, $b=b(t=T_0)=b^{(0)}$ and $c=c(t=T_0)$; the value of $T_0$ is  computed from the fourth equation in \eqref{system2nd}.

If on the contrary, $a<b<c$, then $a=a(t=T_2)$, $b=b(t=T_2)=b^{(2)}$ and $c=c(t=T_2)$; the value of $T_2$ is  computed again from the fourth equation in \eqref{system2nd}.

\section*{Acknowledgements}

The first  and the second authors have been supported in part by the research projects MTM2011-28952-C02-01 (A.M.-F.) and MTM2011-28781 (R.O.) from the Ministry of Science and Innovation of Spain and the European Regional Development Fund (ERDF). Additionally, the first author was supported by the Excellence Grant P09-FQM-4643, the research group FQM-229 from Junta de Andaluc\'{\i}a, and by Campus de Excelencia Internacional del Mar (CEIMAR) of the University of Almer\'{\i}a. 
 
A.M.-F.\ and E.A.R.\ also thank AIM workshop ``Vector equilibrium problems and their applications to random matrix models'' and useful discussions in that research-stimulating environment. We gratefully acknowledge also several constructive remarks from Razvan Teodorescu (University of South Florida, USA),  Arno Kuijlaars (University of Leuven, Belgium) and Tamara Grava (SISSA, Italy), as well as from the anonymous referee, whose detailed critical reports were a valuable contribution.

\def\cprime{$'$} \def\cprime{$'$}

\obeylines
\texttt{
A. Mart\'{\i}nez-Finkelshtein (andrei@ual.es)
Department of Mathematics
University of Almer\'{\i}a, SPAIN, and
Instituto Carlos I de F\'{\i}sica Te\'{o}rica y Computacional
Granada University, SPAIN
\medskip
E.~A.~Rakhmanov (rakhmano@mail.usf.edu)
Department of Mathematics,
University of South Florida, USA
\medskip
R.~Orive (rorive@ull.es)
Department of Mathematical Analysis,
University of La Laguna
Tenerife, Canary Islands, SPAIN
}

\end{document}